\documentclass[letterpaper,12pt,leqno]{article}

\usepackage{fullpage}

\usepackage{hyperref}

\usepackage{amsmath}
\usepackage{amsfonts}
\usepackage{amssymb}
\usepackage{tabularx}
\usepackage{color}
\usepackage{fancyhdr}
\usepackage{tikz-cd}
\usepackage{enumitem}
\usetikzlibrary{calc}
\usetikzlibrary{patterns}
\usetikzlibrary{decorations.pathmorphing}

\def\theoremparent{section}
\usepackage{notes}

\numberwithin{equation}{theorem}

\def\cA{\mathcal{A}}

\def\cW{\mathcal{W}}

\def\d{\partial}

\def\Epi{\mathrm{Epi}}

\def\Fib{\mathrm{Fib}}

\def\Int{\mathrm{Int}}

\def\K{\mathcal{K}}

\def\Net{\textrm{-}\mathrm{Net}}

\def\Ord{\textrm{-}\mathrm{Ord}}
\def\ord{\mathrm{ord}}

\def\PreNet{\textrm{-}\mathrm{PreNet}}

\def\proof{\noindent{\em Proof:}\ }

\def\PsNet{\textrm{-}\mathrm{PsNet}}

\def\qed{\hfill\lower 1em\hbox{$\square$}\vskip 1em}

\def\simo{\,\mathop{\sim}\limits^\circ\,}

\def\slc{\mathrm{slc}}

\def\Tree{\textrm{-}\mathrm{Tree}}
\def\toto{\text{\ \raise0.2em\hbox to 0pt {$\to$}\lower0.2em\hbox{$\to$}\ }}
\def\ulc{\mathrm{ulc}}

\def\lesub{\mathop{<}\limits}
\def\pless{\lesub^\circ}

\begin{document}

\title{Towards Koszulity for categorical structures:\\ category of $2$-nets}
\author{Slava Pimenov, Angel Toledo\def\thefootnote{\arabic{footnote}}\footnotemark[1]}
\date{\today}
\titlepage
\maketitle

\tableofcontents
\footnotetext[1]{Both authors are supported by The Beijing Natural Science Foundation grant BJNSF-IS25027, ``Deformation of algebras over higher symmetric operads''.\\
Slava Pimenov: \texttt{slava.pimenov@gmail.com} \\
Angel Toledo: \texttt{angel.toledo@gmail.com}
}
\vfill\eject

\ifx

Abstract:
Recently Batanin and Markl developed a theory of Koszul duality for operads over a general operadic category in the spirit of Ginzburg-Kapranov. However, their approach doesn't
apply directly to the operadic category of n-trees, which are used to describe categorical structures. In this paper we present a category of (locally constant) 2-nets which addresses
the main difficulty of extending their theory to the case of 2-trees. Specifically, our category of 2-nets is an operadic category containing 2-trees
as a full subcategory, and every map between 2-nets admits a factorization into a chain of elementary maps.

\section{Nets with tassels}
We begin by introducing some terminology and notations.
Let $(S, <)$ be a totally ordered finite set. An interval in $S$ is a subset $I \subset S$ such that whenever $x < y < z$ with $x, z \in I$ we have $y \in I$.
We will denote by $\Int(S)$ the set of all intervals in $S$.

Now, consider an arbitrary subset $C \subset S$, the restriction of the total order $<$ in $S$ to $C$ can be also thought of as a partial order in $S$
that we will denote by $\lesub_C$. Specifically, we say that $x \lesub_C y$ whenever both $x, y \in C$ and $x < y$, and if either $x$ or $y$ is not in $C$
then they are incomparable. Furthermore, if we have a collection of subsets $C_i \subset S$ we can construct the minimal partial order $\lesub_{\{C_i\}}$ generated
by all $\lesub_{C_i}$.

\begin{definition}
A prenet of height $n$ (or an $n$-prenet) is a sequence of totally ordered finite sets $\{S_0, S_1, \ldots, S_n\}$ and maps $p_i \from S_i \to \Int(S_{i-1})$,
satisfying the following conditions.
\begin{enumerate}[label=\alph*)]
\item Denote by $p_i^{-1}(t) = \{x \in S_i \mid t \in p_i(x)\}$, and consider the partial order $\pless$ on $S_i$ generated by partial orders
$\lesub_{p_i^{-1}(t)}$ for all $t \in S_{i-1}$. Then for any $x, y \in S_i$, such that $x \pless y$, there exists $t \in S_{i-1}$, such that $t \in p(x) \cap p(y)$.

\item If $x, y \in S_i$ are incomparable with respect to partial order $\pless$ and $x < y$ then for any $s \in p(x)$ and $t \in p(y)$ we have $s < t$ in $S_{i-1}$.

\item For any $x \in S_i$ and $t \in p(x)$ at least one of the following two conditions hold:
$$
\begin{tikzcd}[row sep=tiny]
\forall y \in p_i^{-1}(t), y < x \quad\Rightarrow\quad p(x) \subset p(y)\\
\forall y \in p_i^{-1}(t), x < y \quad\Rightarrow\quad p(x) \subset p(y)\\
\end{tikzcd}
$$
\end{enumerate}
\end{definition}

We say that a net is {\em connected} if $S_0$ consists of a single element. Furthermore, we say that an $n$-prenet is {\em pruned} if for every $t \in S_i$ with $i < n$
there exists $x \in S_{i+1}$, such that $t \in p(x)$. In this paper we will be primarily interested in pruned connected prenets.

\begin{definition}
A map $f$ between two prenets $(S_\bullet, p_\bullet)$ and $(S'_\bullet, p'_\bullet)$ is a collection of maps $\{f_i\from S_i \to S'_i\}$, such that
\begin{enumerate}[label=\alph*)]
\item each $f_i$ preserves that partial orders $\pless$ on $S_i$ and $S'_i$,
\item for each $i \ge 1$ and $x \in S_i$ we have $f^{-1}(p(x)) = p(f^{-1}(x))$.
\end{enumerate}
\end{definition}

Notice that if $f$ is a map between two connected prenets, then all components $f_i$ are surjective maps.

\begin{definition}
A prenet with tassels is a prenet $(S_\bullet, p_\bullet)$ together with a collection of subsets $T_i \subset S_i$ called tassels such that 
\begin{enumerate}[label=\alph*)]
\item if $x \in S_i - T_i$ then $p(x) \subset S_{i-1} - T_{i-1}$,
\item if $x \in T_i$ then either $p(x) \subset T_{i-1}$ or $p(x) \subset S_{i-1} - T_{i-1}$.
\end{enumerate}
\end{definition}

\begin{lemma}
If $(S_\bullet, T_\bullet, p_\bullet)$ is a prenet with tassels, then the collection of subsets $\{S_i - T_i\}$ with induced total orders and the corresponding
restrictions of $p_i$ form a prenet that will be denoted by $S - T$.
\end{lemma}

\begin{definition}
A map $f$ between two prenets with tassels $(S_\bullet, T_\bullet, p_\bullet)$ and $(S'_\bullet, T'_\bullet, p'_\bullet)$ is a map of prenets
$f\from S \to S'$, such that for all $i \ge 0$ we have $f(S_i - T_i) \subset S'_i - T'_i$.
\end{definition}

We will denote by $n\PreNet$ the category of connected pruned prenets with tassels of height $n$ and maps between them.

\begin{definition}
A map $f\from S \to S'$ of prenets with tassels is called a {\em tassel collapse} if it induces an isomorphism of prenets $S - T$ and $S' - T'$, and
for every $i$ and $x \in S_i$ such that $p(x) \subset S - T$ we have $f(p(x)) = p'(f(x))$.
\end{definition}

\begin{definition}
A map $f\from S \to S'$ of prenets with tassels is called a {\em vertex collapse} if
\begin{enumerate}[label=\alph*)]
\item all $f_i$ except one are bijections;
\item for $f_j$ that is not a bijection, there exists $x \in S'_j - T'_j$ such that $f_j$ is a bijection outside of $x$;
\item the preimage $f^{-1}(x)$ consists of vertices identical to $x$, i.e. for every $y \in f^{-1}(x)$ we have $t \in p(y)$ if and only if $f(t) \in p'(x)$ and
$s \in p^{-1}(y)$ if and only if $f(x) \in p'^{-1}(x)$.
\end{enumerate}
\end{definition}

Consider a subcategory $\cW$ of $n\PreNet$ generated by tassel collapses and vertex collapses. We will refer to maps in $\cW$ as {\em weak equivalences}.
For a prenet $S \in n\PreNet$ let us write $\cW(S)$ for the connected component of $\cW$ containing $S$.

\begin{lemma}[Alternative definition]
A map of $n$-prenets with tassels $f\from S \to S'$ is in $\cW$ if and only if the following two conditions hold
\begin{enumerate}[label=\alph*)]
\item the map induced by $f_i$ between $S_i - T_i$ and $S'_i - T'_i$, is a bijection for $i = n$ and surjection for $i < n$,
\item for all $S_i$ and $x \in S_i$ such that $p(x) \subset (S_i - T_i)$, we have $p(x) = f^{-1} p' f(x) \cap (S_i - T_i)$.
\end{enumerate}
\end{lemma}

\begin{lemma}
Category of weak equivalences satisfies 2-out-of-3 condition.
\end{lemma}

\begin{definition}
A prenet $S$ is called {\em minimal} if the only weak equivalence map $f\from S \to S'$ is the identity.
\end{definition}

\begin{lemma}
Every span $\begin{tikzcd}[cramped,sep=small]T & S \ar[l] \ar[r] & P\end{tikzcd}$ and cospan $\begin{tikzcd}[cramped,sep=small]T \ar[r] & Q & P \ar[l] \end{tikzcd}$ of weak equivalences
can be completed to a commutative square, such that all maps are weak equivalences.
$$
\begin{tikzcd}
& S \ar[ld] \ar[rd] & \\
T \ar[rd] & & P \ar[ld] \\
& Q &
\end{tikzcd}
$$
\end{lemma}

\begin{lemma}
Every connected component of $\cW$ contains a unique minimal prenet.
\end{lemma}

\begin{definition}
The category of $n$-nets, denoted by $n\Net$, has the same objects as $n\PreNet$ and maps defined by
$$
\Hom_{n\Net}(S, T) \ =\ \colim_{\wtilde S \in \cW(S)} \Hom_{n\PreNet}(\wtilde S, T).
$$
\end{definition}

\begin{lemma}
If $f\from T \to T'$ is a weak equivalence in $n\PreNet$, then it induces an isomorphism
$$
\begin{tikzcd}
\Hom_{n\Net}(S, T) \ar["\isom", r] & \Hom_{n\Net}(S, T').
\end{tikzcd}
$$
\end{lemma}

\begin{proposition}
There is a fully faithful embedding $i \from n\Tree \into n\Net$.
\end{proposition}

\begin{definition}
A map $f\from S \to S'$ of prenets with tassels is a {\em fibration} if for every $x \in S_i$ it induces a surjection $p^{-1}(x) \cap T_{i+1} \epi p^{-1}(f(x))$.
\end{definition}

\begin{lemma}
Every map $f \from S \to S'$ in the category of $n$-nets can be represented by a span of $n$-prenets with tassels
$$
\begin{tikzcd}
S & \wtilde S \ar["(w)"', l] \ar["(f)", r] & S',
\end{tikzcd}
$$
where the first map is a weak equivalence and the second is a fibration.
\end{lemma}

Let $f\from S \to S'$ be a map of $n$-nets, and $x \in S'_n$. Take a span $(w, g)$ representing map $f$ and define the {\em fiber} of $f$ over $x$ as the component-wise
fiber product
$$
f^{-1}(x) := \wtilde S \times_{S'} (x).
$$

\begin{lemma}
The notion of the fiber doesn't depend on the choice of the span.
\end{lemma}

\vfill\eject
\fi

\setcounter{section}{-1}
\section{Introduction}
The notion of an operad has been an extremely fruitful concept, that over the last half century found application in numerous areas of mathematics. Although it has
certain limitations as to what kind of structures it can express, it still encompasses a wide range of commonly used notions in set theory, linear algebra, topology
and other settings, while at the same time providing manageable and explicit tools to work with them, that other more general approaches do not afford.

In the linear algebraic setting, where operads are typically referred to as algebraic operads, one of the key
aspects of the theory are the so called Bar and Cobar constructions that provide a systematic way to obtain resolutions of algebraic structures, their deformation complexes
along with their homotopy versions. Moreover, the theory has a notion of Koszulity which tremendously simplifies all of these constructions.

\begin{nparagraph}[Classical operads.]
We briefly recall the main concepts of the classical theory in order to highlight the key aspects that will be of interest to us in this paper. An algebraic operad $\cP$
is a collection of vector spaces $\cP_n$ indexed by non-negative integers $n \in \Z_{\ge 0}$, which are sometimes also equipped with an action
of symmetric groups $S_n$, and a collection of operadic composition maps
$$
\gamma_{k_1,\ldots,k_n}\from \cP_n \tensor \cP_{k_1} \tensor \cdots \tensor \cP_{k_n} \to \cP_{k_1 + \cdots + k_n},
$$
satisfying certain natural compatibility relations. The spaces $\cP_n$ should be thought of as spaces of operations of arity $n$ acting on some vector space,
and maps $\gamma$ provide a way of composing those operations. So for $f \in \cP_n$ and collection $g_i \in \cP_{k_i}$ for $1 \le i \le n$ and denoting by $x_{ij}$, for $1 \le j \le k_i$
the inputs of operation $g_i$ we have
$$
\gamma(f, g_1, \ldots, g_n)(x_{11}, \ldots, x_{nk_n}) = f(g_1(x_{11}, \ldots, x_{1k_1}), \ldots, g_n(x_{n1}, \ldots, x_{nk_n})).
$$
Visually this composition can be represented by a tree
$$
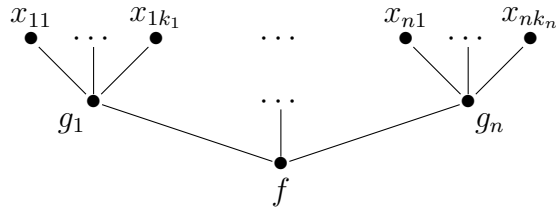
\begin{tikzpicture}[inner sep=0pt,baseline=(a11.base)]
\def\u{2em}
\node (a21) at (0, 0) {$\bullet$} node at ($(a21) + (0, 0.7em)$) {$x_{11}$};
\node (a22) at (1*\u, 0) {$\cdots$};
\node (a23) at (2*\u, 0) {$\bullet$} node at ($(a23) + (0, 0.7em)$) {$x_{1k_1}$};
\node (a24) at (4*\u, 0) {$\cdots$};
\node (a25) at (6*\u, 0) {$\bullet$} node at ($(a25) + (0, 0.7em)$) {$x_{n1}$};
\node (a26) at (7*\u, 0) {$\cdots$};
\node (a27) at (8*\u, 0) {$\bullet$} node at ($(a27) + (0, 0.7em)$) {$x_{nk_n}$};

\node (a11) at (1*\u, -1*\u) {$\bullet$} node at ($(a11) + (-0.7em, -0.7em)$) {$g_1$};
\node (a12) at (4*\u, -1*\u) {$\cdots$};
\node (a13) at (7*\u, -1*\u) {$\bullet$} node at ($(a13) + (+0.7em, -0.7em)$) {$g_n$};
\node (a01) at (4*\u, -2*\u) {$\bullet$} node at ($(a01) + (0, -0.9em)$) {$f$};

\draw (a21) -- (a11) -- (a01);
\draw (a22) -- (a11) -- (a23);
\draw (a25) -- (a13) -- (a01);
\draw (a26) -- (a13) -- (a27);
\draw (a12) -- (a01);
\end{tikzpicture}
$$

More conceptually, we can think that spaces $\cP_n$ indexed not by natural numbers, but by finite sets $\ul{n}$, and then operadic compositions $\gamma_p$ are indexed by surjective maps
$p\from \ul{m} \to \ul{n}$. For example, the map $\gamma_{k_1, \ldots, k_n}$ in the previous notation corresponds to $\gamma_p$, with $m = k_1 + \cdots + k_n$ and such that $|p^{-1}(i)| = k_i$.
If we think of $\ul{n}$ as totally ordered finite sets, and require that maps $p \from \ul{m} \to \ul{n}$ preserve the total order, then we arrive to the non-symmetric version of classical operads.
And if we think of $\ul{n}$ as a finite set without any additional structure, then the automorphism group $\Aut(\ul{n}) = S_n$ gives rise to the action of symmetric groups on spaces $\cP_n$.

\end{nparagraph}

\begin{nparagraph}
Algebraic operads describe many classical algebraic structures, such as operad of associative algebras, commutative algebras, Poisson algebras and many others. However, if one asks a rather
natural question whether there is an operad of operads then one arrives to an example of an algebraic structure that looks very operad-like but formally is not described by the classical operads.
As was indicated by E. Getzler and M. Kapranov in \cite{Getz-Kap} in order to handle such a structure, one should consider operations that accept instead of a sequence of arguments, a collection
or arguments arranged in a shape of a tree. If one generalizes further and considers operations with inputs arranged in the shape of certain kinds of graphs, then one arrives to various notions
of PROPs.

This point of view has been adopted and further developed by M. Batanin and M. Markl in \cite{Bat-Markl} where they propose a general theory of operad-like structures, where operations accept
inputs of various shapes, which are encoded by a certain type of category that they call an operadic category.

Alternative approach to general operad-like structures has been developed by Kaufmann under the name of Feynman categories (\cite{Kaufmann}), which leads to a theory very similar to that of
\cite{Bat-Markl}.

\end{nparagraph}

\begin{nparagraph}[Categorical structures.]
Another direction that necessitates consideration of operations with inputs of shapes other than a sequence of arguments comes from an attempt to describe various categorical structures.
The most simple example of this kind would be the structure of a category itself. Perhaps one of the first and most successful implementations of this idea, albeit without any attempt
at developing a general theory of operads, was done by Boardmann and Vogt in \cite{Board-Vogt} where they proposed a notion of a restricted Kan complex, which is commonly known today under
the name of $(\infty, 1)$-category.

Roughly speaking, the main idea here is that an (infinity-)category is a certain kind of simplicial set, so the operations expressing the structure of a category should be indexed
not by finite sets (or more precisely maps $\ul{n} \to \ul{1}$), as we had in the case of classical operads describing set-theoretic structures, but by simplicial sets. First, as an analog
of the terminal object $\ul{1}$ one could try to take a point, that is the terminal simplicial set, however such operations will only encode output objects, and thus are insufficient to capture
the structure of a category. So instead they proposed to index all (higher) compositions in the infinity-category by maps
$$
p\from \Lambda^i_n \to \Delta(n),\quad 0 < i < n,
$$
where $\Lambda^i_n$ is a simplicial horn with apex vertex $i$, and $\Delta(n)$ is the standard $n$-dimensional simplex. For instance, the operation of the shape $\Lambda^1_2 \to \Delta(2)$
gives a usual composition of maps (up to homotopy).

\end{nparagraph}

\begin{nparagraph}
While this works surprisingly well for infinity-categories, it pretty much fails at the next natural step~--- description of a monoidal category. To highlight the issue arising here,
let us attempt to encode a tensor product $f \tensor g$ of maps $f\from X \to X' $ and $g\from Y \to Y'$, in the same fashion. We can represent objects by $1$-dimensional simplices and
describe their tensor product in the same way as the composition before. Then map $f$ will be represented by a $2$-simplex spanning between edges $X$ and $X'$, and similar for $g$,
providing the following overall picture.

$$
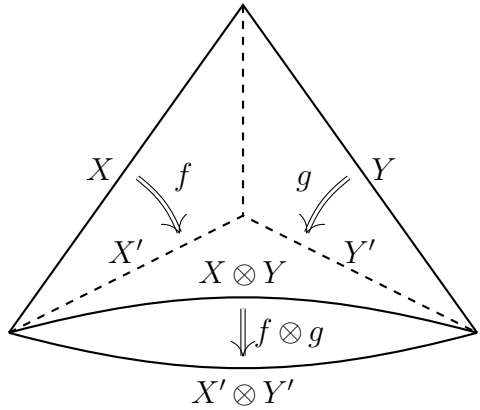
\begin{tikzpicture}
\def\u{15em}
\node (a) at (0, 0) {};
\node (b) at (0.5*\u, 0.7*\u) {};
\node (c) at (1*\u, 0) {};
\node (d) at (0.5*\u, 0.25*\u) {};
\draw [thick] (a.center) -- node [midway,left] {$X$} node [midway] (ab) {} (b.center) -- node [midway,right] {$Y$} node [midway] (bc) {} (c.center);
\draw [thick, dashed] (a.center) --node [midway, above] {$X'$} node [near end] (ad) {} (d.center) -- node [midway,above] {$Y'$} node [near start] (dc) {} (c.center);
\draw [double distance=0.1em, arrows={->[width=0.9em]}] (ab) to [bend left=15] node [midway,above right] {$f$} (ad);
\draw [double distance=0.1em, arrows={->[width=0.9em]}] (bc) to [bend right=15] node [midway, above left] {$g$} (dc);
\draw [thick, dashed] (b.center) -- (d.center);
\draw [thick] (a.center) to [bend left=15] node [midway,above] {$X \tensor Y$} node (acu) [midway] {} (c.center);
\draw [thick] (a.center) to [bend right=15] node [midway,below] {$X' \tensor Y'$} node (acd) [midway] {} (c.center);
\draw [double distance=0.1em, arrows={->[width=0.8em]}] (acu) -- (acd) node [midway,right] {$f \tensor g$};
\end{tikzpicture}
$$

The tensor product $X \tensor Y$ is realized by the front facing triangle, the product $X' \tensor Y'$ is realized by the bottom triangle and the product of maps $f \tensor g$
is represented by the face obtained by slicing off the front edge of the simplex.

The issue here is that the output of this operation is not a standard simplex. The usual way of dealing with such situations is to consider some subdivision of the output simplex
and then attempt to embed the simplicial set describing the input data, which in this case is the pair of maps $(f, g)$ into this subdivision. However, this construction depends on the choice
of both the subdivision and the subsequent embedding, which have nothing to do with the original structure that we are trying to describe. An attempt to manage this ambiguity has been
done in \cite{Lurie} by introducing a second type of $2$-dimensional simplices that they call thin simplices.

\end{nparagraph}

\begin{nparagraph}[Globular approach.]
An alternative approach to categorical structures has been proposed by M. Batanin in \cite{Batanin}. Instead if simplicial sets he proposed to work with globular sets.
Since $n$-dimensional globular cells have only two sides, the beginning and the end, as opposed to $n$-dimensional simplices which have $(n+1)$ faces,
the globular sets provide a coarser information, thus limiting the kinds of structures we can describe. Nevertheless it turns out to be enough to capture various interesting
categorical structures including the structure of monoidal category. In fact Batanin developed the whole theory of $n$-operads, which allowed him to define
globular monoidal $n$-categories.

The shapes of inputs (in other words the operadic category) for this theory of $n$-operads are encoded by certain rather simple kind of globular sets, an example of which
in the case of $n = 2$ is depicted below. The advantage of restricting to this kind of globular sets is that they can be easily described combinatorially by pruned leveled trees with $n$ levels.
Equivalently they can be described by $n$-ordinals, that naturally generalize the notion of the totally ordered finite sets, which were the shapes of inputs for the classical operads.
$$
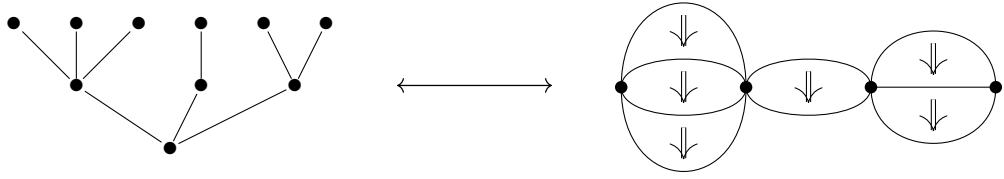
\begin{tikzcd}[sep=5em,cells={inner xsep=2em}]
\begin{tikzpicture}[inner sep=0pt,baseline=(a11.base)]
\def\u{2em}
\node (a21) at (0, 0) {$\bullet$};
\node (a22) at (1*\u, 0) {$\bullet$};
\node (a23) at (2*\u, 0) {$\bullet$};
\node (a24) at (3*\u, 0) {$\bullet$};
\node (a25) at (4*\u, 0) {$\bullet$};
\node (a26) at (5*\u, 0) {$\bullet$};
\node (a11) at (1*\u, -1*\u) {$\bullet$};
\node (a12) at (3*\u, -1*\u) {$\bullet$};
\node (a13) at (4.5*\u, -1*\u) {$\bullet$};
\node (a01) at (2.5*\u, -2*\u) {$\bullet$};

\draw (a21) -- (a11) -- (a01);
\draw (a22) -- (a11) -- (a23);
\draw (a24) -- (a12) -- (a01);
\draw (a25) -- (a13) -- (a01);
\draw (a26) -- (a13);
\end{tikzpicture} \ar[<->, r] &
\begin{tikzpicture}[baseline=-0.2em]
\def\u{4em}

\draw (0, 0) .. controls (0, 0.9*\u) and (1*\u, 0.9*\u) .. node (m11) [midway] {} (1*\u, 0);
\draw (0, 0) .. controls (0, 0.3*\u) and (1*\u, 0.3*\u) .. node (m12) [midway] {} (1*\u, 0);
\draw (0, 0) .. controls (0, -0.3*\u) and (1*\u, -0.3*\u) .. node (m13) [midway] {} (1*\u, 0);
\draw (0, 0) .. controls (0, -0.9*\u) and (1*\u, -0.9*\u) .. node (m14) [midway] {} (1*\u, 0);
\draw (1*\u, 0) .. controls (1*\u, 0.3*\u) and (2*\u, 0.3*\u) .. node (m21) [midway] {} (2*\u, 0);
\draw (1*\u, 0) .. controls (1*\u, -0.3*\u) and (2*\u, -0.3*\u) .. node (m22) [midway] {} (2*\u, 0);
\draw (2*\u, 0) .. controls (2*\u, 0.6*\u) and (3*\u, 0.6*\u) .. node (m31) [midway] {} (3*\u, 0);
\draw (2*\u, 0) -- node (m32) [midway] {} (3*\u, 0);
\draw (2*\u, 0) .. controls (2*\u, -0.6*\u) and (3*\u, -0.6*\u) .. node (m33) [midway] {} (3*\u, 0);

\draw [double distance=0.1em, arrows={->[width=0.9em]}] (m11) -- (m12);
\draw [double distance=0.1em, arrows={->[width=0.9em]}] (m12) -- (m13);
\draw [double distance=0.1em, arrows={->[width=0.9em]}] (m13) -- (m14);
\draw [double distance=0.1em, arrows={->[width=0.9em]}] (m21) -- (m22);
\draw [double distance=0.1em, arrows={->[width=0.9em]}] (m31) -- (m32);
\draw [double distance=0.1em, arrows={->[width=0.9em]}] (m32) -- (m33);

\fill (0, 0) circle [radius=0.2em];
\fill (1*\u, 0) circle [radius=0.2em];
\fill (2*\u, 0) circle [radius=0.2em];
\fill (3*\u, 0) circle [radius=0.2em];
\end{tikzpicture}
\end{tikzcd}
$$

For example, in the description of the monoidal category structure the composition of two maps $gf$ and their tensor product $f \tensor g$ are operations with the inputs of the following respective shapes.

$$
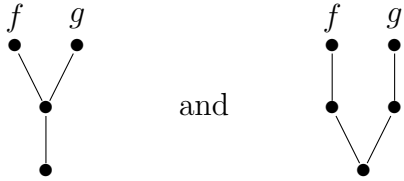
\begin{tikzpicture}[inner sep=0pt,baseline=(a11.base)]
\def\u{2em}
\node (a21) at (0, 0) {$\bullet$} node at ($(a21) + (0, 0.9em)$) {$f$};
\node (a22) at (1*\u, 0) {$\bullet$}  node at ($(a22) + (0, 0.9em)$) {$g$};
\node (a11) at (0.5*\u, -1*\u) {$\bullet$};
\node (a01) at (0.5*\u, -2*\u) {$\bullet$};

\draw (a21) -- (a11) -- (a01);
\draw (a22) -- (a11);
\end{tikzpicture}
\hskip 3em\text{and}\hskip 3em
\begin{tikzpicture}[inner sep=0pt,baseline=(a11.base)]
\def\u{2em}
\node (a21) at (0, 0) {$\bullet$} node at ($(a21) + (0, 0.9em)$) {$f$};
\node (a22) at (1*\u, 0) {$\bullet$}  node at ($(a22) + (0, 0.9em)$) {$g$};
\node (a11) at (0*\u, -1*\u) {$\bullet$};
\node (a12) at (1*\u, -1*\u) {$\bullet$};
\node (a01) at (0.5*\u, -2*\u) {$\bullet$};

\draw (a21) -- (a11) -- (a01);
\draw (a22) -- (a12) -- (a01);
\end{tikzpicture}
$$

The problem with this approach to $n$-operads is that it doesn't admit a notion of Koszul duality. In order to clarify the issue here we will first recall the relevant aspects of the classical picture.

\end{nparagraph}

\begin{nparagraph}[Koszulity.]
Let us briefly outline the key idea of the Bar construction for a classical algebraic operad following V. Ginzburg and M. Kapranov (\cite{Ginzburg-Kapranov}). Starting with an operad $\cP$
we first construct a cofree co-operad $\C$ generated by operations of $\cP$. Explicitly, we put
$$
\C_n \ =\ \bigoplus_{\text{trees }T \atop \text{with $n$ leaves}} \left(\det(T) \tensor \bigotimes_{v \in T} \cP_{|v|} \right),
$$
where the sum is taken over all rooted trees with $n$ leaves, the tensor product is taken over all vertices $v$ of a given tree and $|v|$ denotes the number of inputs of vertex $v$.
In other words, spaces $\C_n$ are spanned by trees with vertices decorated by operations of operad $\cP$.

Each of the spaces $\C_n$ can be graded by the number of edges in the tree $T$. Moreover, one can define a differential on $\C_n$, making it into a complex. This differential is constructed
by combining two pieces of information. The first piece is the composition in the operad $\cP$, and the second piece is the differential in the tree complex, which roughly speaking sends a tree
$T$ to a linear combination of all possible trees obtain by contracting an edge in $T$.

Operad $\cP$ is said to be Koszul if the resulting Bar complex has no higher cohomologies.
\end{nparagraph}

\begin{nparagraph}
The main hurdle in extending this construction to other operadic categories, is the existence of an analog of the tree complex, or more specifically an analog of the edge contraction operation.
Let us rephrase it slightly in order to better adapt it to the theory of operadic categories. Any tree $T$ can be turned into a pruned leveled tree by assigning to every vertex a level, such that
parent of any vertex $v$ lives one level below $v$. Such a leveled tree can be encoded by a sequence of surjective maps of finite sets
$$
\begin{tikzcd}
S_l \ar[r, "p_l"] & S_{l-1} \ar[r, "p_{l-1}"] & \ldots \ar[r] & S_1 \ar[r, "p_1"] & S_0,
\end{tikzcd}
$$
where $S_i$ is the set of all vertices at level $i$, $S_0$ consists of a single vertex~--- the root of the tree, and maps $p_i$ send a vertex to its parent.
Moreover, any tree can be leveled in a such a way the every map $p_i$ in this sequence is {\em elementary}, in the sense that there is at most one element $v \in S_{i-1}$ with
non-trivial preimage $p_i^{-1}(v)$. Then it is clear that the levels of this leveled tree correspond to the edges in the original tree $T$ and contraction of an edge corresponds to
elimination of a level.

While this procedure is self-evident for trees (of finite sets) it fails when we try to extend it to trees of leveled $n$-trees. For example, the following
map in the category of $2$-trees has two non-trivial fibers, but does not admit a factorization into elementary maps.
$$
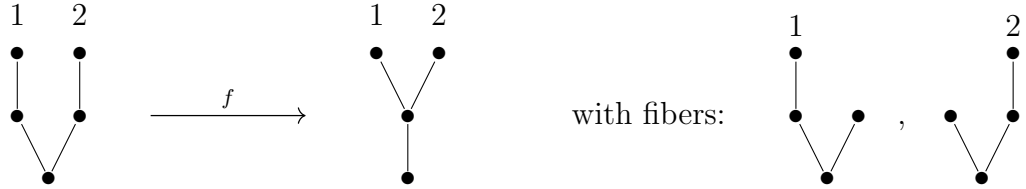
\begin{tikzcd}[sep=5em,cells={inner xsep=2em}]
\begin{tikzpicture}[inner sep=0pt,baseline=(a11.base)]
\def\u{2em}
\node (a21) at (0, 0) {$\bullet$} node at ($(a21) + (0, 0.9em)$) {$1$};
\node (a22) at (1*\u, 0) {$\bullet$}  node at ($(a22) + (0, 0.9em)$) {$2$};
\node (a11) at (0*\u, -1*\u) {$\bullet$};
\node (a12) at (1*\u, -1*\u) {$\bullet$};
\node (a01) at (0.5*\u, -2*\u) {$\bullet$};

\draw (a21) -- (a11) -- (a01);
\draw (a22) -- (a12) -- (a01);
\end{tikzpicture}
\ar[r, "f"] &
\begin{tikzpicture}[inner sep=0pt,baseline=(a11.base)]
\def\u{2em}
\node (a21) at (0, 0) {$\bullet$} node at ($(a21) + (0, 0.9em)$) {$1$};
\node (a22) at (1*\u, 0) {$\bullet$}  node at ($(a22) + (0, 0.9em)$) {$2$};
\node (a11) at (0.5*\u, -1*\u) {$\bullet$};
\node (a01) at (0.5*\u, -2*\u) {$\bullet$};

\draw (a21) -- (a11) -- (a01);
\draw (a22) -- (a11);
\end{tikzpicture}
\end{tikzcd}
\hskip 2em\text{with fibers:}\hskip 2em
\begin{tikzpicture}[inner sep=0pt,baseline=(a11.base)]
\def\u{2em}
\node (a21) at (0, 0) {$\bullet$} node at ($(a21) + (0, 0.9em)$) {$1$};
\node (a11) at (0*\u, -1*\u) {$\bullet$};
\node (a12) at (1*\u, -1*\u) {$\bullet$};
\node (a01) at (0.5*\u, -2*\u) {$\bullet$};

\draw (a21) -- (a11) -- (a01);
\draw (a12) -- (a01);
\end{tikzpicture}\quad,\quad
\begin{tikzpicture}[inner sep=0pt,baseline=(a11.base)]
\def\u{2em}
\node (a22) at (1*\u, 0) {$\bullet$}  node at ($(a22) + (0, 0.9em)$) {$2$};
\node (a11) at (0*\u, -1*\u) {$\bullet$};
\node (a12) at (1*\u, -1*\u) {$\bullet$};
\node (a01) at (0.5*\u, -2*\u) {$\bullet$};

\draw (a11) -- (a01);
\draw (a22) -- (a12) -- (a01);
\end{tikzpicture}
$$

\end{nparagraph}

\begin{nparagraph}

Let us describe the key idea behind our proposed solution to this problem. Consider the globular sets associated to the two leveled trees in the example above,
then one could attempt to do a factorization of the following kind.
$$
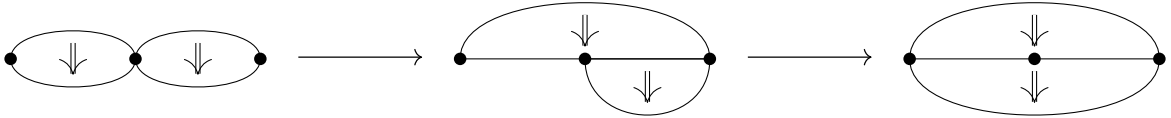
\begin{tikzcd}[sep=4em,cells={inner xsep=1em}]
\begin{tikzpicture}[baseline=-0.2em]
\def\u{4em}

\draw (0, 0) .. controls (0, 0.3*\u) and (1*\u, 0.3*\u) .. node (m11) [midway] {} (1*\u, 0);
\draw (0, 0) .. controls (0, -0.3*\u) and (1*\u, -0.3*\u) .. node (m12) [midway] {} (1*\u, 0);
\draw (1*\u, 0) .. controls (1*\u, 0.3*\u) and (2*\u, 0.3*\u) .. node (m21) [midway] {} (2*\u, 0);
\draw (1*\u, 0) .. controls (1*\u, -0.3*\u) and (2*\u, -0.3*\u) .. node (m22) [midway] {} (2*\u, 0);

\draw [double distance=0.1em, arrows={->[width=0.9em]}] (m11) -- (m12);
\draw [double distance=0.1em, arrows={->[width=0.9em]}] (m21) -- (m22);

\fill (0, 0) circle [radius=0.2em];
\fill (1*\u, 0) circle [radius=0.2em];
\fill (2*\u, 0) circle [radius=0.2em];
\end{tikzpicture}
\ar[r] &
\begin{tikzpicture}[baseline=-0.2em]
\def\u{4em}

\draw (0, 0) .. controls (0, 0.6*\u) and (2*\u, 0.6*\u) .. node (m11) [midway] {} (2*\u, 0);
\draw (0, 0) -- node (m12) [midway] {} (2*\u, 0);
\draw (1*\u, 0) -- node (m21) [midway] {} (2*\u, 0);
\draw (1*\u, 0) .. controls (1*\u, -0.6*\u) and (2*\u, -0.6*\u) .. node (m22) [midway] {} (2*\u, 0);

\draw [double distance=0.1em, arrows={->[width=0.9em]}] (m11) -- (m12);
\draw [double distance=0.1em, arrows={->[width=0.9em]}] (m21) -- (m22);

\fill (0, 0) circle [radius=0.2em];
\fill (1*\u, 0) circle [radius=0.2em];
\fill (2*\u, 0) circle [radius=0.2em];
\end{tikzpicture}
\ar[r] &
\begin{tikzpicture}[baseline=-0.2em]
\def\u{4em}

\draw (0, 0) .. controls (0, 0.6*\u) and (2*\u, 0.6*\u) .. node (m11) [midway] {} (2*\u, 0);
\draw (0, 0) -- node (m12) [midway] {} (2*\u, 0);
\draw (0, 0) .. controls (0, -0.6*\u) and (2*\u, -0.6*\u) .. node (m13) [midway] {} (2*\u, 0);

\draw [double distance=0.1em, arrows={->[width=0.9em]}] (m11) -- (m12);
\draw [double distance=0.1em, arrows={->[width=0.9em]}] (m12) -- (m13);

\fill (0, 0) circle [radius=0.2em];
\fill (1*\u, 0) circle [radius=0.2em];
\fill (2*\u, 0) circle [radius=0.2em];
\end{tikzpicture}
\end{tikzcd}
$$

The object in the middle is no longer a globular set, however it still admits a combinatorial description in the spirit of leveled $n$-trees, only now we have to 
extend the category of trees to the category of {\em nets}, in the sense that a vertex is now allowed to have more than one parent. For instance, the factorization above
can be described by the following nets.
$$
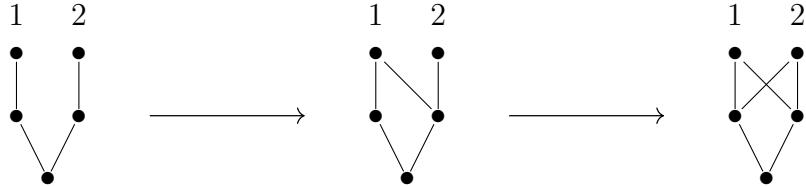
\begin{tikzcd}[sep=5em,cells={inner xsep=2em}]
\begin{tikzpicture}[inner sep=0pt,baseline=(a11.base)]
\def\u{2em}
\node (a21) at (0, 0) {$\bullet$} node at ($(a21) + (0, 0.9em)$) {$1$};
\node (a22) at (1*\u, 0) {$\bullet$}  node at ($(a22) + (0, 0.9em)$) {$2$};
\node (a11) at (0*\u, -1*\u) {$\bullet$};
\node (a12) at (1*\u, -1*\u) {$\bullet$};
\node (a01) at (0.5*\u, -2*\u) {$\bullet$};

\draw (a21) -- (a11) -- (a01);
\draw (a22) -- (a12) -- (a01);
\end{tikzpicture}
\ar[r] &
\begin{tikzpicture}[inner sep=0pt,baseline=(a11.base)]
\def\u{2em}
\node (a21) at (0, 0) {$\bullet$} node at ($(a21) + (0, 0.9em)$) {$1$};
\node (a22) at (1*\u, 0) {$\bullet$}  node at ($(a22) + (0, 0.9em)$) {$2$};
\node (a11) at (0*\u, -1*\u) {$\bullet$};
\node (a12) at (1*\u, -1*\u) {$\bullet$};
\node (a01) at (0.5*\u, -2*\u) {$\bullet$};

\draw (a21) -- (a11) -- (a01);
\draw (a21) -- (a12);
\draw (a22) -- (a12) -- (a01);
\end{tikzpicture}
\ar[r] &
\begin{tikzpicture}[inner sep=0pt,baseline=(a11.base)]
\def\u{2em}
\node (a21) at (0, 0) {$\bullet$} node at ($(a21) + (0, 0.9em)$) {$1$};
\node (a22) at (1*\u, 0) {$\bullet$}  node at ($(a22) + (0, 0.9em)$) {$2$};
\node (a11) at (0*\u, -1*\u) {$\bullet$};
\node (a12) at (1*\u, -1*\u) {$\bullet$};
\node (a01) at (0.5*\u, -2*\u) {$\bullet$};

\draw (a21) -- (a11) -- (a01);
\draw (a21) -- (a12);
\draw (a22) -- (a11);
\draw (a22) -- (a12) -- (a01);
\end{tikzpicture}
\end{tikzcd}
$$

While the precise description of the factorization in our category will be slightly different, this picture conveys the essence of the construction and serves as
the starting point for the development of our theory. Although the pictures here are a certain type of graphs, we use the word nets instead, as the category of graphs
is already used extensively in relation to study of operads. And since the maps in the category of graphs are very different from the maps that we study
in this paper, we prefer to use the word nets to avoid confusion.

Let us take a closer look at what we need to do in order to get the factorization above. First we take the left $2$-cell and stretch it over the top edge of the right cell,
which can also be thought of as composing with an ``identity cell'', which originally wasn't in the picture. Moreover, the last object in the decomposition is not
the target globular set of map $f$ that we are factorizing, but a ``stretched'' version of it. Therefore, we want to make the following two procedures isomorphisms in
our operadic category.
\begin{enumerate}[label=\alph*)]
\item Adding ``identity cells'' (or {\em tassels} as we call them) along the edges of existing cells.
\item Stretching existing cells.
\end{enumerate}

\end{nparagraph}

\begin{nparagraph}[Main results.]
Following the ideas outline above, we define the category of $2$-prenets with tassels, and a class of weak equivalences containing (and essentially generated by) the
two types of procedures (a) and (b) mentioned before. We also define a subcategory of $2$-prenets consisting of universally locally constant maps.

Then we define the category of (locally constant) $2$-nets as the localization of the category of $2$-prenets (respectively $2$-prenets and universally locally constant maps)
with respect to the class of weak equivalences, and give a description of maps in these localized categories that allows us to work with them effectively.

Finally, we show that the category of locally constant $2$-nets has the following properties, thus providing a solution to the original problem.
\begin{itemize}
\item The category of locally constant $2$-nets is an operadic category.
\item It contains the category of leveled $2$-trees as a full subcategory.
\item It admits factorization of maps into a chain of elementary maps.
\end{itemize}

\end{nparagraph}

\begin{nparagraph}[Shortcomings.]
In this paper we only work with nets and prenets with $n = 2$ levels. In order to generalize this to an arbitrary $n$, one first has to allow tassels on all levels of the
prenet, and then tweak the definitions of maps and weak equivalences to ensure that the lifting proposition \ref{prop_lift} still holds, since it serves as the crucial step
towards an effective description of the localized category.

There is also an issue with the factorization of a map into a chain of elementary maps in the category of locally constant $2$-nets. Namely, the set of non-trivial fibers of such factorization
in general depends on the choice of the factorization. However, this discrepancy is of a rather tame nature, as the fibers only differ by merging of some tassels into another tassel.
Moreover, by analyzing examples, it appears that this discrepancy is closely related to the failure of the category of (non locally constant) $2$-nets to be an operadic category.
Which leads us to conjecture that one can define a larger class of weak equivalences that both resolves this discrepancy between fibers of different factorizations, as well as makes
the category of further localized $2$-nets into an operadic category.

\end{nparagraph}

\begin{nparagraph}[Computational networks.]
Unrelated to the problem of Koszulity the categories constructed in this paper (and especially their generalizations for an arbitrary height $n$)
may be of interest in relation to study of computational networks. For example, we can consider a prenet as a kind of network, where we assign inputs to the vertices at the highest level,
and to every vertex at level $i$ we assign some operation that takes inputs from its children at level $i + 1$ and sends the outputs to its parents at level $i - 1$. With this interpretation
two prenets are weakly equivalent to each other if the pattern of exchange of information between vertices in one network is similar to that of the other. In particular a weak
equivalence map tells us which vertices can be combined together without the need to add new connections in the network.

\end{nparagraph}

\begin{nparagraph}[Outline of the paper.]
In section \ref{sec_recollect} we introduce some terminology and notations that are used throughout the paper. We also recall definitions of operadic category, $n$-trees and $n$-ordinals,
as well as the notion of localization of a category.

Section \ref{sec_prenets} introduces the notion of $2$-prenets and maps between them. Since the exposition is very formal the reader is encouraged to refer to the discussion of the
geometric realization of prenets at the end of section \ref{sec_prenets} and the list of examples in appendix \ref{sec_examples} for motivation behind and illustration of definitions
and statements.

Section \ref{sec_weak_equiv} is dedicated entirely to the definition of weak equivalences and studying their properties. The mains results here are the lifting construction of
proposition \ref{prop_lift} as well as discussion of minimal representatives of equivalence classes at the end of the section.

In section \ref{sec_nets} we define the category of $2$-nets as the localization of $2$-prenets with respect to the class of weak equivalences. Then in theorem \ref{thm_net_homs}
we give a description of maps in the localized category, combining ideas of Quillen and Ore localizations. We conclude the section with the construction of a fully faithful functor
from the category of $2$-trees with maps surjective at level $1$ into the category of $2$-nets.

Section \ref{sec_lc_nets} is devoted to the definition and study of the class of locally constant maps between $2$-prenets. We define the category of locally constant $2$-nets as
the localization of the subcategory of $2$-prenets and universally locally constant maps between them with respect to weak equivalences. The resulting category is a refinement
of the category of $2$-nets, in the sense that the natural functor from locally constant $2$-nets to $2$-nets is surjective on the sets of maps. And we conclude the section
by showing that the category of locally constant $2$-nets is an operadic category admitting factorization of maps into chains elementary maps.

We conclude the paper with two appendices. In appendix \ref{sec_q_tassels} we discuss the issue of non-functoriality of the lifting construction in proposition \ref{prop_lift},
and outline an enhancement of the category of $2$-prenets that admits a functorial lift. Finally, in appendix \ref{sec_examples} we provide a list of examples illustrating various
notions and statements in this paper.

\end{nparagraph}

\vskip 1em

The authors would like to thank Michael Batanin for an insightful discussion on the subject.

\vskip 5em

\section{Notations and recollections}
\label{sec_recollect}
In this section we recall some terminology and introduce notations that will be used throughout this paper.

\begin{nparagraph}[Partially ordered sets.]
Let $(S, <)$ be a totally ordered finite set. An interval in $S$ is a subset $I \subset S$ such that whenever $x < y < z$ with $x, z \in I$ we have $y \in I$.
We will denote by $\Int(S)$ the set of all intervals in $S$. If $U$ and $V$ are two subsets of $S$ (not necessarily intervals) such that for any $u \in U$
and $v \in V$ we have $u < v$, then we will write $U < V$.

Consider an arbitrary subset $C \subset S$, the restriction of the total order $<$ in $S$ to $C$ can be also thought of as a partial order in $S$
that we will denote by $\lesub_C$. Specifically, we say that $x \lesub_C y$ whenever both $x, y \in C$ and $x < y$, and if either $x$ or $y$ is not in $C$
then they are incomparable. Furthermore, if we have a collection of subsets $C_i \subset S$ we can construct the minimal partial order $\lesub_{\{C_i\}}$ generated
by all $\lesub_{C_i}$.
\end{nparagraph}

\begin{nparagraph}[Localization of a category.]
Let $\C$ be a category and consider a subcategory $\cW$ of $\C$. We recall the notion of localization of $\C$ with respect to $\cW$, that we will denote by $\C[\cW^{-1}]$
following \cite{Hovey}.

First, we construct the free category $F(\C, \cW^{-1})$ who's objects are the same as in the category $\C$ and morphisms are generated by maps in $\C$ and formal inverses of
maps in $\cW$. In other words a map from $S$ to $T$ in $F(\C, \cW^{-1})$ is an alternating zig-zag consisting of maps from $\C$ and $\cW$
$$
\begin{tikzcd}
S \ar[r] & S_1 & S_2 \ar[dashed, l] \ar[r] & \cdots & S_n \ar[dashed, l] \ar[r] & T,
\end{tikzcd}
$$
where dashed arrows represent maps in $\cW$. The composition is defined as the concatenation of zig-zags and the identity map is the zig-zag of length zero.

We define the localization $\C[\cW^{-1}]$ as the quotient of this category by the following relations.
\begin{enumerate}[label=\alph*)]
\item For any object $S \in \C$ the zig-zag of length $1$ consisting of $\id_S$, considered either as a map in $\C$ or in $\cW$, is identified with the identity map in $F(\C, \cW)$.
\item For any two composable maps $f$ and $g$ in $\C$ the zig-zag of length $1$ consisting of the composition $gf$ is identified with the zig-zag of the form
$$
\begin{tikzcd}
\phi\from S \ar[r, "f"] & P & P \ar[l, "="'] \ar[r, "g"] & T.
\end{tikzcd}
$$
\item For any two composable maps $u$ and $v$ in $\cW$ the zig-zag consisting of their composition $vu$ is identified with the zig-zag of the form
$$
\begin{tikzcd}
\psi\from S & P \ar[l, "v"'] \ar[r, "="] & P & T \ar[l, "u"'].
\end{tikzcd}
$$
\item For any map $w\from S \to T$ in $\cW$ the zig-zag of the form
$$
\begin{tikzcd}
S \ar[r, "w"] & T & S \ar[l, "w"']
\end{tikzcd}
$$
is identified with $\id_S$ and the zig-zag of the form
$$
\begin{tikzcd}
T & S \ar[l, "w"'] \ar[r, "w"] & T 
\end{tikzcd}
$$
is identified with $\id_T$.
\end{enumerate}

\end{nparagraph}

\begin{nparagraph}[Operadic category.]
We recall the notion of an operadic category introduced by M. Batanin and M. Markl (\cite{Bat-Markl-1}, \cite{Bat-Markl}).

Denote by $\FSets$ the category of finite sets and maps between them. An operadic category is a category $\O$ such that every connected component $\mathcal{K}$ of $\O$ has a chosen local terminal object, that will be denoted by $U_\K$.
Moreover $\O$ is equipped with a {\em cardinality} functor
$$
\begin{tikzcd}
\O \ar[r, "|-|"] & \FSets,
\end{tikzcd}
$$
such that for any component $\K$ the cardinality $|U_\K| = \ul{1}$.

Additionally for every object $T \in \O$ and $x \in |T|$ we have a {\em fiber functor} $\Fib_x \from \O / T \to \O$, satisfying the following conditions.
\begin{enumerate}[label=\alph*)]
\item Let $f\from S \to T$ be an object of $\O / T$, then $|\Fib_x(f)| = |f|^{-1}(x)$.
\item For any map $h\from S \to P$ in $\O / T$
$$
\begin{tikzcd}
S \ar[dr, "f"'] \ar[rr, "h"] && P \ar[dl, "g"] \\
& T &
\end{tikzcd}
$$
and every $x \in |T|$, the map $|\Fib_x(h)|$ coincides with the map induced by $|h|$ between preimages $|f|^{-1}(x)$ and $|g|^{-1}(x)$.
$$
\begin{tikzcd}
{|S|} \ar[dr, "|f|"'] \ar[rr, "|h|"] && {|P|} \ar[dl, "|g|"] \\
& {|T|} &
\end{tikzcd}
$$
\item For the identity map $\id_T$ considered as an object of $\O / T$ and every $x \in |T|$ the fiber $\Fib_x(\id_T)$ is the chosen local terminal object of some component $\K \subset \O$.
\item Let $T$ be the chosen terminal object of a connected component $\K$, and denote by $z$ the only element of $|T|$. Then the fiber functor $\Fib_z \from \O / T \to \O$ coincides with the
inclusion of the connected component $\K$ into $\O$.

\item Consider the category $\mathcal{A}_2(P, T)$, consisting of pairs of composable arrows
$$
\begin{tikzcd}
S \ar[r, "h"] & P \ar[r, "g"] & T,
\end{tikzcd}
$$
and let $x \in |P|$. We also consider two functors, the forgetful functor $F\from \cA_2(P, T) \to \O / P$, which forgets $T$, and the fiber functor $\Fib_{|g|(x)}\from \cA_2(P, T) \to \O / \Fib_{|g|(x)}(g)$.
In virtue of (a) the element $x$ can also be considered as an element of $|\Fib_{|g|(x)}(g)|$, and we require that the following diagram of functors commutes.
$$
\begin{tikzcd}[sep=3em]
\cA_2(P, T) \ar[d, "\Fib_{|g|(x)}"'] \ar[r, "F"] & \O / P \ar[d, "\Fib_x"] \\
\O / \Fib_{|g|(x)}(g) \ar[r, "\Fib_x"'] & \O.
\end{tikzcd}
$$
\end{enumerate}
Notice, that due to functoriality the last condition (e) implies both conditions (\romannumeral4) and (\romannumeral5) of section 1 in \cite{Bat-Markl}.

A special case of this notion was considered by C. Barwick in \cite{Barwick} under the name of operator category, where the cardinality functor is given by $|T| = \Hom_{\O}(Z, T)$, where $Z$ is the terminal object
of $\O$, and the fiber functor $\Fib_x$ is given by the Cartesian product
$$
\begin{tikzcd}[sep=3em]
\Fib_x(f) \ar[d] \ar[r] & S \ar[d, "f"] \\
Z \ar[r, "x"'] & T,
\end{tikzcd}
$$
who's existence is assumed. We would like to point out here, that the operadic category of locally constant $2$-nets constructed in this paper doesn't not have this kind of fibered products.
\end{nparagraph}

\begin{nparagraph}[Weak blow-up axiom.]
This definition covers a wide range of operadic categories, however it may be useful to impose some additional properties. For instance, in \cite{Bat-Markl-3} Batanin and Markl consider a list of properties
on the category $\O$ that allow them to develop a generalization of the Koszul duality theory for $\O$-operads. Of a particular importance for us here will be the so called weak blow-up axiom.

First we equip the finite sets in the category $\FSets$ with a total order, then we can consider a subcategory $\Delta \subset \FSets$ consisting of order preserving maps. We denote by $\O_\ord$ the subcategory
of $\O$ consisting of maps $f$, such that $|f|$ belongs to $\Delta$, and we refer to such maps as order preserving.

Let $f\from S \to T$ be an order preserving map, and consider a collection of maps $l_x\from \Fib_x (f) \to P_x$ for some $P_x \in \O$, indexed by $x \in |T|$.
We say that an operadic category $\O$ satisfies the weak blow-up axiom if there exists a unique factorization of $f$
$$
\begin{tikzcd}
S \ar[dr, "f"'] \ar[rr, "h"] && P, \ar[dl, "g"] \\
& T &
\end{tikzcd}
$$
such that $g \in \O_\ord$ and the map between fibers $h_x = \Fib_x(h)\from \Fib_x(f) \to \Fib_x(g)$ coincide with $l_x$.

\end{nparagraph}

\begin{nparagraph}[$n$-trees and $n$-ordinals.]
We begin by recalling the notion of an $n$-ordinal, for details we refer to \cite{Batanin-sym} and \cite{Bat-n-ordinal}. Let $T$ be a finite set equipped with $n$ partial orders $<_i$ for $0 \le i \le n-1$.
We say that $T$ is an $n$-ordinal if they satisfy the following conditions.
\begin{enumerate}[label=\alph*)]
\item For any two distinct elements $x, y \in T$ there exists unique $i$, such that $x$ and $y$ are comparable with respect to $<_i$.
\item If $x <_i y$ and $y <_j z$ then $x <_{\min\{i, j\}} z$.
\end{enumerate}

A map between two $n$-ordinals $f\from (S, <^S_\bullet) \to (T, <^T_\bullet)$ is a map $f\from S \to T$ of the underlying sets, such that whenever $x <^S_i y$ one of the following three possibilities holds:
$f(x) <^T_j f(y)$ for some $j \ge i$, $f(x) = f(y)$ or $f(y) <^T_j f(x)$ for some $j > i$.

We denote the category of $n$-ordinals and maps between them by $n\Ord$.
\end{nparagraph}

\begin{nparagraph}
A leveled $n$-tree $S$ is a collection of totally ordered finite sets $\{S_0, S_1, \ldots, S_n\}$ and maps
$$
\begin{tikzcd}
S_n \ar[r, "p_n"] & S_{n-1} \ar[r, "p_{n-1}"] & \cdots \ar[r] & S_1 \ar[r, "p_1"] & S_0,
\end{tikzcd}
$$
called parent maps, such that $S_0$ consists of a single element, referred to as the root of $S$. A tree is called {\em pruned} if all of the maps $p_i$, $1 \le i \le n$ are surjective.

For each level $l$ of the tree we have a structure of an $l$-ordinal on $S_l$ defined by putting for two distinct elements $x$ and $y$ in $S_l$
$$
x <^{S_l}_i y
$$
where $i$ is the maximal level of the tree, such that
$$
p_l \ldots p_{i+1}(x) = p_l \ldots p_{i+1}(y).
$$

A map between two leveled $n$-trees $f\from S \to T$ is a collection of maps $f_i\from S_i \to T_i$ for $0 \le i \le n$, such that all the squares in the diagram below commute.
$$
\begin{tikzcd}[sep=3em]
S_n \ar[d, "f_n"'] \ar[r, "(p_S)_n"] & S_{n-1} \ar[d, "f_{n-1}"] \ar[r, "(p_S)_{n-1}"] & \cdots \ar[r] & S_1 \ar[d, "f_1"] \ar[r, "(p_S)_1"] & S_0 \ar[d, "f_0"] \\
T_n \ar[r, "(p_T)_n"'] & T_{n-1} \ar[r, "(p_T)_{n-1}"'] & \cdots \ar[r] & T_1 \ar[r, "(p_T)_1"'] & T_0.
\end{tikzcd}
$$
Moreover, we require that on each level $l$ the map $f_l$ is a map of $l$-ordinals.

We denote the category of $n$-trees and maps between them by $n\Tree$. As was shown by M. Batanin (\cite{Bat-n-ordinal}) there is a natural functor from the category of $n$-ordinals to the category of $n$-trees,
which is fully faithful, and its image consists of the pruned $n$-tress.

Let us fix $l \le n$ and consider a subcategory of $n\Tree$ consisting of maps $f \from S \to T$, such that $f_i$ are surjective for all $0 \le i \le l$. We will denote this category by $(n\Tree, \Epi_l)$.

\end{nparagraph}

\vskip 5em


\section{2-PreNets with tassels}
\label{sec_prenets}
In this section we define the category of $2$-prenets with tassels and describe their geometric realization as two-dimensional cellular complexes.

\begin{definition}
\label{def_prenet}
A prenet of height $2$ (or a $2$-prenet) is a triple of totally ordered finite sets $\{S_0, S_1, S_2\}$ and maps $p_i \from S_i \to \Int(S_{i-1})$ for $i = 1, 2$,
with the property that for any $x, y \in S_i$, $i \ge 1$ such that $x < y$ and which are incomparable with respect to the partial order generated by $\lesub_{p_i^{-1}(t)}$ for all $t \in S_{i-1}$, we have $p(x) < p(y)$.
\end{definition}

We will refer to elements of $S_i$ as vertices of prenet of level $i$, to elements of $p(x)$ as {\em parents} of $x$.
For any $t \in S_{i-1}$ we will denote $p_i^{-1}(t) = \{x \in S_i \mid t \in p_i(x)\}$, and refer to elements of $p^{-1}(t)$ as {\em children} of $t$.
For every $S_i$ we will write $\pless$ for the partial order generated by partial orders $\lesub_{p_i^{-1}(t)}$ for all $t \in S_{i-1}$.

We say that a net is {\em connected} if $S_0$ consists of a single element. Furthermore, we say that an $2$-prenet is {\em pruned} if for every $t \in S_i$ with $i < 2$
there exists $x \in S_{i+1}$, such that $t \in p(x)$. It is clear from the definition that for a connected $2$-prenet $S$ the partial order $\pless$ on $S_1$ coincides with the total order.
In this paper we will be primarily interested in pruned connected prenets.

\begin{remark}
\label{rem_total_order}
Let us point out that for a connected prenet $S$ the total order for each level $S_i$, $i \ge 1$ is determined by the collection of partial orders $\pless$ on levels $S_j$ for $j \le i$.
As we already mentioned on $S_1$ the partial order $\pless$ coincides with the total order, and there is nothing to check. Consider two vertices $x$ and $y$ in $S_2$, then either they
are comparable with respect to $\pless$ or if they are not then the sets of their parents are disjoint and we put $x < y$ whenever $p(x) < p(y)$. It remains to check that this is a
well defined order.

Let us assume that $x \pless y$, and that $y$ and $z$ are incomparable and $p(y) < p(z)$. We need to show that one of the following two cases takes place, either $x \pless z$ or $x$ and $z$ are
incomparable, and $p(x) < p(z)$. First consider the case when $x$ and $z$ are comparable, then we can not have $z \pless x$, as that would imply that $z \pless y$, but by assumption
$y$ and $z$ are incomparable. Therefore we must have $x \pless z$.

Next, assume that if $x$ and $z$ are incomparable then $z < x$ would imply that $p(y) < p(z) < p(x)$. Since $x \pless y$ there exists a chain of vertices $\{x_0, x_1, \ldots, x_n\}$, such that
$x = x_0$, $y = x_n$ and every pair $(x_i, x_{i+1})$ has a common parent, and we take a minimal such chain. We will arrive to a contradiction using inductive argument on its length. For $n = 1$ vertices $x$ and $y$
share a parent, which immediately contradicts $p(y) < p(x)$. If $n > 1$ then consider the pair over vertices $x_1$ and $z$. If they are comparable, then by the previous argument we have $x_1 \pless z$,
but since $x \pless x_1$ we also have $x \pless z$ which contradicts the assumption that they are incomparable. Therefore, $x_1$ and $z$ must be incomparable, and since $x_1$ and $y$ are connected by a shorter
chain of vertices, using the inductive assumption we arrive to a contradiction. Hence we must have $p(x) < p(z)$.
\end{remark}

\begin{definition}
\label{def_sp_prenet}
A 2-prenet $S$ is called {\em special} if it satisfies the following two conditions.
\begin{enumerate}[label=\alph*)]
\item For any $x, y \in S_i$, such that $x \pless y$, there exists $t \in S_{i-1}$, such that $t \in p(x) \cap p(y)$.

\item For any $x \in S_i$, $i \ge 1$ and $t \in p(x)$ at least one of the following two conditions holds:
$$
\begin{tikzcd}[row sep=tiny]
\forall y \in p_i^{-1}(t), \ \text{\em such that}\  y < x \quad\Rightarrow\quad p(x) \subset p(y),\\
\forall y \in p_i^{-1}(t), \ \text{\em such that}\  x < y \quad\Rightarrow\quad p(x) \subset p(y).\\
\end{tikzcd}
$$
\end{enumerate}
\end{definition}

\begin{nparagraph}
\label{par_AB_notation}
Before defining maps between prenets we need to introduce some notations. Let $S$ be a prenet and consider a subset $U \subset S_i$. For any $z \in p(U)$ we form
the intersection $p^{-1}(z) \cap U$ and denote by $a(z)$ its minimal element and by $b(z)$ its maximal element. Denote by $A(U)$ the subset of $U$ formed by elements
$\{a(z) \mid z \in p(U)\}$ and similarly we denote $B(U) = \{b(z) \mid z \in p(U)\} \subset U$.

For any two subsets $V \in S_i$ and $W \in S_{i-1}$ we say that $V$ is a {\em disjoint cover} of $W$ if subsets $p(x)$ are pairwise disjoint for all $x \in V$ and $W = \bigcup_{x \in V} p(x)$.
\end{nparagraph}

\begin{definition}
\label{def_prenet_map}
A map $f$ between two prenets $S = (S_\bullet, p_S)$ and $T = (T_\bullet, p_T)$ is a collection of maps $\{f_i\from S_i \to T_i\}$, such that
\begin{enumerate}[label=\alph*)]
\item each $f_i$ preserves that partial orders $\pless$ on $S_i$ and $T_i$,
\item for each $i \ge 1$ and $x \in T_i$ we have $f^{-1}(p_T(x)) = p_S(f^{-1}(x))$,
\item for each $i \ge 1$ and $x \in T_i$ subsets $A(f^{-1}(x))$ and $B(f^{-1}(x))$ are disjoint covers of $f^{-1}(p_T(x)) = p_S(f^{-1}(x))$.
\end{enumerate}
\end{definition}

For the motivation and an illustration of condition (c) we refer to example \ref{exa_map_c}.

Let us point out several simple facts that immediately follow from the definition.
\begin{lemma} Let $f\from S \to T$ be a map of prenets.
\label{lemma_prenet_map}
\begin{enumerate}[label=\alph*')]
\item If $S$ and $T$ are two connected prenets, then all components $f_i$ are surjective maps, and properties \ref{def_prenet_map} (b) and (c) are trivially satisfied at level $1$.
\item The condition (b) of definition \ref{def_prenet_map} is equivalent to
$$
f(p^{-1}(x)) = p^{-1}(f(x))
$$
for any $x \in S_i$, $i < 2$.
\item Map $f$ preserves relation of parenthood, i.e. if $x, y \in S$ such that $x \in p(y)$ then $f(x) \in p(f(y))$.
\item If $x, y$ are two vertices of $T$ such that $x \in p(y)$, then for any $x' \in f^{-1}(x)$ there exists at least one $y' \in f^{-1}(y)$ such that $x' \in p(y')$.
\end{enumerate}
\end{lemma}
\proof
To show (a') first observe that under the assumptions of $S$ and $T$ map $f_0$ is necessarily surjective. Now, assume that we know that $f_i$ is surjective, and pick an element
$x \in T_{i+1}$. By surjectivity of $f_i$ we have $f_i^{-1}(p(x))$ is non-empty, and using condition (b) of definition \ref{def_prenet_map} we see that $f_{i+1}^{-1}(x)$ also
have to be non-empty, hence $f_{i+1}$ is surjective.

Next, to establish (c'), let us apply condition (b) to vertex $f(y)$. Clearly, $x$ belongs to $p(f^{-1}(f(y)))$, hence $x \in f^{-1}(p(f(y)))$, and therefore $f(x) \in p(f(y))$.

Verification of (d') is equally straightforward and we leave it to the reader.

Let us show that (b) implies (b'). First, it is clear that (c') implies the inclusion $f(p^{-1}(x)) \subset p^{-1}(f(x))$. To show the other inclusion, we apply (d') to a pair of vertices
$(f(x), y)$, where $y$ is a child of $f(x)$. Taking $x' = x$ in the conclusion of (d') we find a child $y'$ of $x$ mapping to $y$, hence establishing the other inclusion.

The other implication uses a similar argument. First we observe that condition (b') implies both (c') and (d') and then check that (c') implies the inclusion
$f^{-1}(p(x)) \supset p(f^{-1}(x))$, while (d') implies the opposite inclusion.

In particular we have established that the condition (b) of definition \ref{def_prenet_map} is equivalent to the combination of (c') and (d').
\qed

\begin{definition}
\label{def_prenet_tassels}
A $2$-prenet with tassels is a prenet $S = (S_\bullet, p_\bullet)$ together with a subset $t(S) \subset S_2$, who's elements are called tassels.
\end{definition}

We will denote by $c(S)$ the triple $(S_0, S_1, S_2 - t(S))$ with the parent maps being restrictions of $p_i$. Clearly $c(S)$ forms a $2$-prenet, which will be called
the {\em core} of $S$, and the vertices of $c(S)$ will be referred to as core vertices.

\begin{definition}
\label{def_prenet_map_tassels}
A map $f$ between two prenets with tassels $S$ and $T$ is a map of prenets $f\from S \to T$, such that $f(c(S)) \subset c(T)$.
\end{definition}

\begin{lemma}
\label{lemma_map_composition}
Let $f\from S \to T$ and $g\from T \to P$ be two maps of prenets with tassels. Then the levels-wise composites $h_i = g_i f_i$ form a map of prenets.
\end{lemma}
\proof
The property (a) of definition \ref{def_prenet_map} is clearly satisfied by $h$. Take a vertex $x \in P_i$, then we have
$$
h^{-1} p_P(x) = f^{-1} g^{-1} p_P (x) = f^{-1} p_T g^{-1}(x) = p_S f^{-1} g^{-1}(x) = p_S h^{-1}(x).
$$

It remains to check that subsets $A(h^{-1}(x))$ and $B(h^{-1}(x))$ are disjoint covers of $h^{-1} p(x)$. Since the two cases are similar we will only
work out the case of $A$. First, observe that
\begin{equation}
\label{equ_Ah_AfAg}
A(h^{-1}(x)) = A(f^{-1} A(g^{-1}(x))).
\end{equation}
Indeed, pick a vertex $t \in h^{-1} p(x)$ and consider intersections
$$
p^{-1}(t) \cap f^{-1} A(g^{-1}(x)) \ \ \subset\ \  p^{-1}(t) \cap h^{-1}(x).
$$
Let $y$ be the minimal vertex of the intersection on the right side, then since $f$ preserves the partial order $\pless$ it must send $y$ to the minimal vertex
in the intersection $p^{-1}(f(t)) \cap g^{-1}(x)$. Then by definition of $A$ we have $f(y) \in A(g^{-1}(x))$, which establishes equality \ref{equ_Ah_AfAg}.
We know that $A(g^{-1}(x))$ is a disjoint cover of $g^{-1} p(x)$, and for any $z \in A(g^{-1}(x))$ the subset $A(f^{-1}(z))$ is a disjoint cover of $f^{-1}p_T(z)$.
Because of disjointedness of the cover $A(g^{-1}(x))$ it is clear that $A(f^{-1}(z))$ are pairwise disjoint subsets of $S_i$ for various $z \in A(g^{-1}(x))$.
Therefore, we conclude that
$$
A(f^{-1} A(g^{-1}(x))) = \bigcup_{z \in A(g^{-1}(x))} A(f^{-1}(z)) 
$$
forms a disjoint cover of $f^{-1}g^{-1}(x) = h^{-1}(x)$.
\qed

We will denote by $2\PreNet$ the category of connected pruned prenets with tassels of height $2$ and maps between them. And denote by $2\PreNet^s$ the full
subcategory of special $2$-prenets.

\begin{nparagraph}
Starting with any prenet $S$ in $2\PreNet$ we can construct another prenet $S^\circ$, by putting level-wise $S^\circ_i = S_i$, for all levels $i$ and $p_{S^\circ} = p_S$, with the
only difference being that we put $t(S^\circ) = S^\circ_2$.

We will also denote by $I$ the prenet that has a single vertex at each level, with the parent maps determined uniquely by this condition and the empty set of tassels. Visually,
the prenets $I$ and $I^\circ$ can be depicted as follows.
$$
I \ =\  
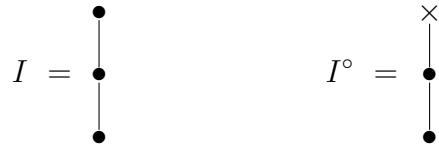
\begin{tikzpicture}[inner sep=0pt,baseline=(a11.base)]
\def\u{2em}
\node (a21) at (0, 0) {$\bullet$};
\node (a11) at (0, -1*\u) {$\bullet$};
\node (a01) at (0, -2*\u) {$\bullet$};

\draw (a21) -- (a11) -- (a01);
\end{tikzpicture}
\hskip 7em
I^\circ \ =\  
\begin{tikzpicture}[inner sep=0pt,baseline=(a11.base)]
\def\u{2em}
\node (a21) at (0, 0) {$\cross$};
\node (a11) at (0, -1*\u) {$\bullet$};
\node (a01) at (0, -2*\u) {$\bullet$};

\draw (a21) -- (a11) -- (a01);
\end{tikzpicture}
$$

Let us point out the following factorization property for maps of prenets.
\end{nparagraph}

\begin{lemma}
\label{lemma_map_factor}
Any map $f\from S \to T$ in $2\PreNet$ can be factorized via a prenet $P$
$$
\begin{tikzcd}
S \ar[rr, "f"] \ar[dr, "g"'] && T \\
& P \ar[ur, "h"'] &
\end{tikzcd}
$$
where $P_1 = S_1$, $P_2 = T_2$, $t(P) = t(T)$ and parents $p_P(x) = p_S(f^{-1}(x))$. And where the maps are defined by $g_1 = \id_{S_1}$, $g_2 = f_2$, $h_1 = f_1$ and $h_2 = \id_{T_2}$.
\end{lemma}
\proof
Let us check that $P$ is a prenet. First, by property \ref{def_prenet_map} (b) for map $f$ we also have $p_P(x) = f^{-1}(p_T(x))$. Since $f_1$ preserves the total order at level $1$
and $p_T(x)$ is an interval in $T_1$ we conclude that $p_P(x)$ is also an interval in $P_1 = S_1$. Now, take $x$ and $y$ in $P_2$, which are incomparable with respect to $\pless_P$ and such that $x < y$.
This implies that $x$ and $y$ considered as vertices of $T_2$ are also incomparable with respect to $\pless_T$. Indeed, since the partial order $\pless_T$ is generated by pairs of vertices $u$ and $v$ having
a common parent $t$ and such that $u < v$, it is enough to show that $u \pless_P v$. But this is immediate, since any vertex in the preimage $h_1^{-1}(t)$ is a common parent of $u$ and $v$ considered as vertices of $P$.
Therefore, since $T$ is a prenet we obtain $p_T(x) < p_T(y)$. Again, since $f_1$ preserves the
total order at level $1$ we conclude that
$$
p_P(x) = f^{-1}(p_T(x)) < p_P(y) = f^{-1}(p_T(y)).
$$

It remains to show that both $g$ and $h$ are maps of prenets. Let us first verify property (a) of \ref{def_prenet_map}. At level $1$ it immediately follows from the fact that $f$ preserves the total orders
on $S_1$ and $T_1$, so we focus on level $2$. By definition the partial order $\pless_P$ is generated
by relations $x < y$ for all pairs $x, y \in P_2$, that have a common parent, say $t' \in p_P(x) \cap p_P(y)$. But by definition of $p_P$ we immediately see that $f(t')$ is
the common parent of $x$ and $y$ in $T$, hence $x \pless_T y$. And so $x \pless_P y$ implies $x \pless_T y$. Similarly, let $u < v$ in $S_2$ and let $s$ be their common parent in $S$. Then, since
$g_2 = f_2$ we have $g_2(x) < g_2(y)$, and moreover by definition of $p_P$ we see that $s$ is a common parent of $g_2(x)$ and $g_2(y)$ in $P$. This implies that $g$ also preserves the partial order.

The property \ref{def_prenet_map} (b) for both maps $g$ and $h$ follows immediately from the definition of $p_P$.

Finally, let us verify property (c). Let $x \in T_2$, then $h^{-1}(x) = \{x\}$, hence both $A(h^{-1}(x)) = B(h^{-1}(x)) = \{x\}$ form a disjoint cover of $p_P(h^{-1}(x))$.
Let $y \in P_2 = T_2$, then $g^{-1}(y) = f^{-1}(y)$, but since $f$ is a map of prenets we conclude that both $A(g^{-1}(y))$ and $B(g^{-1}(y))$ are disjoint covers.
\qed

The second map in the factorization constructed in this lemma belongs to the class of maps that will be called {\em vertex collapses}, which are a special type of weak equivalences that will be studied in section \ref{sec_weak_equiv}.

\begin{nparagraph}[Cellular complex.]
To visualize the notions developed in this paper it may be helpful to consider a geometric realization of prenets as a two-dimension cellular complex of a particular type.
Recall that to a pruned $2$-tree $T$ one can associate a globular set $G(T)$, such that the set of $2$-dimensional globules is identified with the set of vertices $T_2$, and in turn,
the globular set $G(T)$ can be geometrically realized as a cellular complex (\cite{Batanin}). On the other hand any such $2$-tree can also be thought of as a pruned $2$-prenet, and the geometric
realization of prenets described here is compatible with the realization for trees. We would like to point out however, that realization for prenets no longer factors through
the category of globular sets.

Let $S$ be a connected pruned $2$-prenet, and let us label vertices of $S_1$ in the increasing order by natural numbers $\{1, 2, \ldots, N\}$, where $N = |S_1|$.
Consider a vertex $x \in S_2$ and let $p(x) = [a, b] \subset S_1$. We associate to it a two-dimensional cell $C_x$, equipped with the parametrization of its boundary by two maps
$$
s\from [a-1, b] \to \d C_x,\quad\text{and}\quad t\from [a-1, b] \to \d C_x,
$$
where $[a - 1, b]$ is understood as a segment of the real line $[a-1, b] \subset \R$. These two maps parametrize the entire boundary, in other words the union $\Im(s) \cup \Im(t) = \d C_x$,
and their intersection consists of two points $\Im(s) \cap \Im(t) = \{L_x, R_x\}$. Visually, $C_x$ can be depicted as follows.

$$
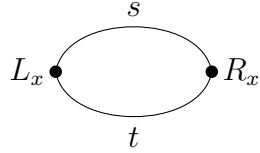
\begin{tikzpicture}
\def\u{5em}
\fill (0, 0) circle [radius=0.2em] node [left] {$L_x$};
\fill (1*\u, 0) circle [radius=0.2em] node [right] {$R_x$};
\draw [-, bend left=80] (0, 0) to node [above] {$s$} (1*\u, 0);
\draw [-, bend right=80] (0, 0) to  node [below] {$t$} (1*\u, 0);
\end{tikzpicture}
$$

Now we will describe how individual cells $C_x$ are glued into a cellular complex. Consider two vertices $x, y \in S_2$ with parents $p(x) = [a_x, b_x]$ and $p(y) = [a_y, b_y]$. We distinguish four cases.
\begin{enumerate}
\item If $a_y > b_x + 1$, then the two cells $C_x$ and $C_y$ are disjoint.
\item If $p(x)$ and $p(y)$ intersect, we denote by $J = [c, d] \in S_1$ their intersection interval, and assume that $x < y$ in $S_2$. Then the cells $C_x$ and $C_y$ are glued along their boundaries
by identifying points $t_x(z) \sim s_y(z)$ for all $z \in [c - 1, d] \subset \R$.
\item If $a_y = b_x + 1$ and vertices $x$ and $y$ are comparable with respect to the partial order $\pless$, then the cells $C_x$ and $C_y$ are disjoint.
\item If $a_y = b_x + 1$ and vertices $x$ and $y$ are incomparable with respect to the partial order $\pless$, then the cells $C_x$ and $C_y$ are glued by identifying points $R_x \sim L_y$.
\end{enumerate}

We illustrate these four cases in the example below.
$$
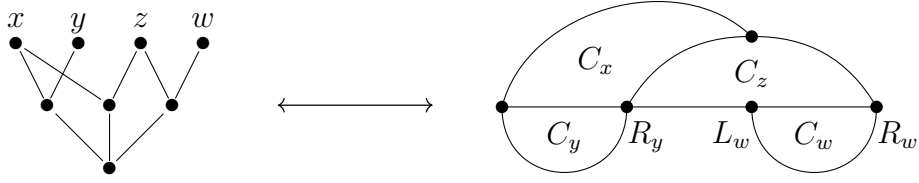
\begin{tikzcd}[sep=5em,cells={inner xsep=2em}]
\begin{tikzpicture}[inner sep=0pt,baseline=(a11.base)]
\def\u{2em}
\node (a21) at (0, 0) {$\bullet$} node at ($(a21) + (0, 0.5em)$) {$x$};
\node (a22) at (1*\u, 0) {$\bullet$} node at ($(a22) + (0, 0.5em)$) {$y$};
\node (a23) at (2*\u, 0) {$\bullet$} node at ($(a23) + (0, 0.5em)$) {$z$};
\node (a24) at (3*\u, 0) {$\bullet$} node at ($(a24) + (0, 0.5em)$) {$w$};
\node (a11) at (0.5*\u, -1*\u) {$\bullet$};
\node (a12) at (1.5*\u, -1*\u) {$\bullet$};
\node (a13) at (2.5*\u, -1*\u) {$\bullet$};
\node (a01) at (1.5*\u, -2*\u) {$\bullet$};

\draw (a21) -- (a11) -- (a01);
\draw (a21) -- (a12) -- (a01);
\draw (a22) -- (a11);
\draw (a23) -- (a12);
\draw (a23) -- (a13) -- (a01);
\draw (a24) -- (a13);
\end{tikzpicture} \ar[<->, r] &
\begin{tikzpicture}[inner sep=0pt,baseline=-0.2em]
\def\u{4em}
\node (m) at (2*\u, 0.5*\u) {};
\draw (0, 0) -- (1*\u, 0) -- (2*\u, 0) -- (3*\u, 0);
\draw (0, 0) .. controls (0, -0.7*\u) and (1*\u, -0.7*\u) .. (1*\u, 0);
\draw [bend left=30] (1*\u, 0) to (m.center) [bend left=30] to (3*\u, 0);
\draw (2*\u, 0) .. controls (2*\u, -0.7*\u) and (3*\u, -0.7*\u) .. (3*\u, 0);
\draw (0, 0) edge [bend left=60] (m.center);

\fill (0, 0) circle [radius=0.2em];
\fill (1*\u, 0) circle [radius=0.2em];
\fill (2*\u, 0) circle [radius=0.2em];
\fill (3*\u, 0) circle [radius=0.2em];
\fill (m) circle [radius=0.2em];

\node at (0.75*\u, 0.3*\u) {$C_x$};
\node at (0.5*\u, -0.3*\u) {$C_y$};
\node at (2*\u, 0.2*\u) {$C_z$};
\node at (2.5*\u, -0.3*\u) {$C_w$};
\node at (1*\u, 0) [below=0.9em, right] {$R_y$};
\node at (2*\u, 0) [below=0.9em, left] {$L_w$};
\node at (3*\u, 0) [below=0.9em, right] {$R_w$};
\end{tikzpicture}
\end{tikzcd}
$$

In this example the pair of vertices $(y, w)$ falls into the first case, the pair $(x, w)$ into the third case, the pair $(y, z)$ into the fourth case, and all the other pairs into the second case.

Cells corresponding to tassels should be thought of as ``thin'' cells, which in a sense provide identification between their top and bottom sides. Of a particular interest may be the special
case of $2$-prenets such that every cell in their geometric realization has at least one side (either top or bottom) which is not subdivided by the gluing process (for instance we refer to
the notion of a multicategory from \cite{Leinster} chapter 5) . Unfortunately, such prenets
are not stable under the class of maps of prenets that we call weak equivalences, which will be studied in section \ref{sec_weak_equiv}. Instead we defined special prenets as a replacement
notion, since their realizations still satisfy this property, but are also stable with respect to weak equivalences (lemma \ref{lemma_w_special}).

\end{nparagraph}

\begin{nparagraph}
\label{par_cat_info}
Let us also give an interpretation of the geometric realization of a more categorical flavor. We think of the points $L_x$ and $R_x$ for various cells $C_x$ as vertices or objects.
The gluing process described above subdivides each side of a two cell $C_x$ into a sequence of intervals by the points $L_y$ and $R_y$ of adjoining cells $C_y$. We think of each such segment
as a $1$-arrow. Then a cell $C_x$ is thought of as a $2$-arrow connecting the chain of composable $1$-arrows lying on the side $s_x$ to the chain of composable arrows on the side $t_x$.

Consider what information is encoded by such a picture. Every cell $C_x$ determines its top and bottom sides, in other words the ``composition'' of $1$-arrows along those sides,
but not each individual arrow. The gluing of two cells $C_x$ and $C_y$ provides identification of the corresponding $1$-arrows, along which they are glued. For instance, in the example above,
cell $C_z$ determines its bottom side $t_z$ consisting of two segments $[R_x, L_w]$ and the top side of cell $C_w$. While the entire side $t_z$ is determined by $z$ and $s_w$ is determined
by $w$, the segment $[R_x, L_w]$ is not determined by this picture. In other words we think of it as ``$t_z$ factors through $s_w$''.

Informally speaking, if $f\from S \to T$ is a map in the category of prenets, then prenet $T$ contains a subset of information encoded by $S$. The class of weak equivalences studied
in section \ref{sec_weak_equiv} consists of those maps that do not forget any information. We also refer to example \ref{exa_map_c} for an illustration of this property of maps.

\end{nparagraph}

\vskip 5em

\section{Weak equivalences}
\label{sec_weak_equiv}

Next we define and study properties of a class of maps that we will call {\em weak equivalences}. Even though we borrowed the terminology
from the Quillen's model category theory, we do not claim here that there is a model structure on the category of prenets. Instead we use it merely to communicate
intuition behind these definitions.

\begin{definition}
\label{def_weak_equiv}
A map of $2$-prenets with tassels $f\from S \to T$ is a weak equivalence if and only if the following two conditions hold
\begin{enumerate}[label=\alph*)]
\item the map induced by $f_2$ between $c(S_2)$ and $c(T_2)$ is a bijection,
\item for all $x \in S_2$ we have $p_S(x) = f^{-1} p_T f(x)$.
\end{enumerate}
\end{definition}

We will denote the class of weak equivalences by $\cW$.

\begin{lemma}
\label{lemma_w_23}
Class of weak equivalences satisfies 2-out-of-3 condition. In other words let $f$ and $g$ be two composable maps of $2$-prenets, then whenever any two of the three maps
$\{f, g, gf\}$ are weak equivalences all three are weak equivalences. In particular $\cW$ is a subcategory of $2\PreNet$.
\end{lemma}
\proof
First of all, it is clear from the definition that identity maps are weak equivalences. Now consider two composable maps
$$
\begin{tikzcd}
S \ar[r, "f"] & T \ar[r, "g"] & R
\end{tikzcd}
$$
and denote their composition by $h$.

{\bf a)} Assume that both $f$ and $g$ are weak equivalences. Then according to (a) of definition \ref{def_weak_equiv} they both induce isomorphism between core
vertices at level $2$. Therefore, the same is true for $h$ as well. Using condition (b) for $f$ and $g$ we have for any $x \in S_2$
$$
p_S(x) = f^{-1} p_T f(x) = f^{-1} g^{-1} p_R (gf(x)) = h^{-1} p_R h(x).
$$
Hence $h$ is a weak equivalence.

{\bf b)} Assume that $g$ and $h$ are weak equivalences. Then once again it is clear that $f$ satisfies condition (a). Condition (b) for $f$ follows by replacing
$g^{-1} p_R g(f(x))$ with $p_T(f(x))$ in $p_S(x) = f^{-1} g^{-1} p_R (gf(x))$.

{\bf c)} Finally, assume that $f$ and $h$ are weak equivalences. As before, it is clear that $g$ satisfies condition (a). We need to check that for any $x \in T_2$
we have $p_T(x) = g^{-1} p_R g(x)$. Since map $f_1$ is surjective (according to lemma \ref{lemma_prenet_map} (a')) it is enough to check that
$$
f^{-1} p_T(x) = f^{-1} g^{-1} p_R g(x).
$$
Since $f_2$ is also surjective there exists $x' \in S_2$, such that $f(x') = x$ and therefore we need to check that
$$
f^{-1} p_T f(x') = f^{-1} g^{-1} p_R g(f(x)).
$$
Now the left hand side coincides with $p_S(x')$, since $f$ is a weak equivalence, which in turn coincides with the right hand side, since $h$ is also a weak equivalence.

\qed

The following lemma will be convenient for establishing factorization of weak equivalences.

\begin{lemma}
\label{lemma_p_order}
Let $f\from S \to T$ be a map of prenets. Denote by $\pless_f$ the partial order on $S_2$ generated by the collection of subsets $C'_y = \{x \in S_2 \mid y \in p(f(x))\}$ for all vertices $y \in T_1$.
Then $x \pless_S y$ implies $x \pless_f y$. Moreover, if $f$ is a weak equivalence then the two partial orders $\pless_S$ and $\pless_f$ coincide.
\end{lemma}
\proof
For any vertex $t \in S_1$ the set of children $p^{-1}(t)$ is contained in $C'_{f(t)}$ by lemma \ref{lemma_prenet_map} (c'). Since the partial order $\pless_S$ is generated by
subsets $p^{-1}(t)$ for all $t \in S_1$ we immediately conclude that $x \pless_S y$ implies $x \pless_f y$.
%

Now, assume that $f$ is a weak equivalence, and consider two vertices $x$ and $y$ such that $x \pless_f y$. This implies that there exist chains of vertices $\{x_0, x_1, \ldots, x_n\}$ in $S_2$
with $x = x_0$ and $y = x_n$, and $\{y_1, \ldots, y_n\}$ in $T_1$, such that $y_i \in p(f(x_{i-1})) \cap p(f(x_i))$, for all $1 \le i \le n$. Since $f$ is a weak equivalence
the preimage $f^{-1}(y_i) \subset p(x_{i-1}) \cap p(x_i)$, so by picking any element $\tilde y_i \in f^{-1}(y_i)$ we obtain a chain of vertices in $S_1$ connecting $x$ and $y$, therefore
$x \pless_S y$.
\qed

\begin{lemma}
\label{lemma_w_special}
If $f\from S \to T$ is a weak equivalence and either $S$ or $T$ is a special prenet, then both  $S$ and $T$ are special.
\end{lemma}
\proof
{\bf a)} Assume that $T$ is special. Consider two vertices $x, y \in S_2$ such that $x \pless_S y$. Then $f(x) \pless_T p(y)$ and because $T$ is special vertices $f(x)$ and $f(y)$ have
a common parent $t$ in $T_1$. Arguing as in the proof of the previous lemma we see that any vertex in $f^{-1}(t)$ is a common parent of $x$ and $y$ in $S$, which establishes property
(a) of \ref{def_sp_prenet}.

Now consider a vertex $x \in S_2$, $t \in p(x)$ and let $y$ be another child of $t$. Then their images $f(x)$, $f(t)$ and $f(y)$ are connected by the same parenthood relations.
Since by assumption $T$ is special it satisfies property \ref{def_sp_prenet} (b), but because $f$ is a weak equivalence, the inclusion $p_T(f(x)) \subset p_T(f(y))$ implies
that $p_S(x) \subset p_S(y)$, which establishes property (b) for $S$ as well.

{\bf b)} Assume that $S$ is special. Consider two vertices $x, y \in T_2$ such that $x \pless_T y$, and take $\tilde x \in f^{-1}(x)$ and $\tilde y \in f^{-1}(y)$. By definition
of the partial order $\pless_f$ we have $\tilde x \pless_f \tilde y$, since $f$ is a weak equivalence using lemma \ref{lemma_p_order} we see that $\tilde x \pless_S \tilde y$.
Therefore $\tilde x$ and $\tilde y$ have a common parent, say $t \in S_1$, and hence $p(t) \in T_1$ is a common parent of $x$ and $y$, which establishes property \ref{def_sp_prenet} (a).

Now consider a vertex $x \in T_2$, $t \in p(x)$ and let $y$ be another child of $t$. As was shown in the proof of the previous lemma, $f$ being a weak equivalence allows us to lift
them to vertices $\tilde x, \tilde y \in S_2$ and $\tilde t \in S_1$ satisfying the same parenthood relations. Since $S$ is special it satisfies property \ref{def_sp_prenet} (b).
However, in virtue of $f$ being a weak equivalence we have $p_S(\tilde x) = f^{-1}(p_T(x))$ and $p_S(\tilde y) = f^{-1}(p_T(y))$, so whenever $p_S(\tilde x) \subset p_S(\tilde y)$ we also
have $p_T(x) \subset p_T(y)$, which establishes property \ref{def_sp_prenet} (b) for $T$.
\qed

The following lemma simplifies verification that a collection of level-wise maps forms a weak equivalence of prenets.

\begin{lemma}
\label{lemma_easy_w}
Let $S$ and $T$ be two prenets in $2\PreNet$, and $f_i\from S_i \to T_i$ be a collection of maps for all $i$, such that
\begin{enumerate}[label=\alph*)]
\item each $f_i$ preserves that partial orders $\pless$ on $S_i$ and $T_i$,
\item the map induced by $f_2$ between $c(S_2)$ and $c(T_2)$ is a bijection,
\item for all $x \in S_2$ we have $p_S(x) = f^{-1} p_T f(x)$.
\end{enumerate}
Then $f$ is a weak equivalence.
\end{lemma}
\proof
Clearly, property (c) implies property \ref{def_prenet_map} (b). In order to check \ref{def_prenet_map} (c) let us pick a vertex $x \in T_2$. Then by property (c) of this lemma
we see that all $\tilde x \in f_2^{-1}(x)$ have the same set of parents equal to $f_1^{-1}(p_T(x))$. Therefore, the subsets $A(f_2^{-1}(x))$ and $B(f_2^{-1}(x))$ both
consist of a single vertex, who's parents are again the entire set $f_1^{-1}(p_T(x))$, which establishes property \ref{def_prenet_map} (c).
\qed

We will now define two special types of weak equivalences. In what follows we will see that any weak equivalence can be factorized into a composition of weak equivalences
of these two types, which will be helpful to both clarify the meaning of the notion of weak equivalence of prenets and make some arguments easier
as we will only need to consider one type or the other.

\begin{definition}
\label{def_tc}
A map $f\from S \to T$ of prenets with tassels is called a {\em tassel collapse} if it induces an isomorphism between cores $c(S)$ and $c(T)$, and
for every $x \in S_2$ we have $f(p_S(x)) = p_T(f(x))$.
\end{definition}

We will also say that a tassel collapse map $f\from S \to T$ is {\em simple} if preimages $f_2^{-1}(x)$ consist of a single element for all except possibly one vertex $x \in T_2$.
It is easy to see that any tassel collapse $f\from S \to T$ is a weak equivalence. Indeed, since $f$ is an isomorphism on cores we immediately have condition (a) of definition \ref{def_weak_equiv}.
Furthermore, since $f_1$ is a bijection, the condition $f(p_S(x)) = p_T(f(x))$ can be rewritten as $p_S(x) = f^{-1} p_T f(x)$, which gives condition (b).

\begin{lemma}
\label{lemma_tc_factor}
Any tassel collapse map $f\from S \to T$ can be factorized into a chain of simple tassel collapses.
\end{lemma}
\proof
It is clear that the surjective map $f_2$ can be factorized into a chain of maps with at most one non-trivial preimage. We need to check that resulting objects are in fact prenets,
and that the induced maps between them are tassel collapses. We proceed by induction on the number of vertices in $T_2$ with non-trivial preimages. If it has no more than one such
vertex there is nothing to prove. Otherwise pick one of such vertices, say $y \in T_2$ and put $P_i = T_i \isom S_i$ for $i \le 1$ and define $P_2$ as the pushout
$$
P_2 = \{y\} \bigsqcup\limits_{f^{-1}(y)} S_2.
$$
The subset of tassels in $P_2$ is inherited from subsets $t(S)$ and $t(T)$.

If we know
that the subset $f^{-1}(y)$ is an interval in $S_2$ then the total order of $S_2$ clearly descends to $P_2$. Let us take $x, x' \in f^{-1}(y)$ and consider vertex $z \in S_2$, such that
$x < z < x'$. Since $f$ is a weak equivalence $p(x) = p(x')$, and we investigate three possibilities. First, if $z$ is comparable to both $x$ and $x'$ then we must have $x \pless z \pless x'$,
but since $f$ preserves the partial order, we conclude that $f(x) = f(z) = f(x') = y$. Now assume that $z$ is comparable to neither $x$ nor $x'$ then by definition \ref{def_prenet} we have
$p(x) < p(z) < p(x') = p(x)$ which is a contradiction. Finally, assume that $z \pless x'$ and incomparable with $x$. Then $p(x) < p(z)$ and we argue by induction on the length
of the chain connecting $z$ and $x'$ similarly to remark \ref{rem_total_order}. If $z$ and $x'$ share a common parent then we immediately arrive to a contradiction. Otherwise,
consider a minimal chain $\{z_0, z_1, \ldots, z_n\}$ with $z = z_0$ and $x' = z_n$. Since $z_{n-1}$ and $z_n$ share a parent we conclude that $x \pless z_{n-1} \pless x'$, therefore
$f(z_{n-1}) = y$ and $p(z_{n-1}) = p(x) = p(x')$. As $z$ and $z_{n-1}$ are connected by a shorter chain we use the inductive process to arrive to a contradiction.
This shows that $f^{-1}(y)$ is an interval in $S_2$.

So far we have the following level-wise factorization
$$
\begin{tikzcd}
f\from S \ar[r, "g"] & P \ar[r, "h"] & T,
\end{tikzcd}
$$
where maps $g$ and $h$ are clear from the construction of $P$. Next we define the parent map $p_P\from P_2 \to \Int(P_1)$ as the composition $p_P(x) = p_T(h(x))$. Let us check that the condition of
definition \ref{def_prenet} is satisfied. Consider two incomparable vertices $x$ and $x'$ in $P_2$ such that $x < x'$. By definition of the total order on $P_2$ this means that for any preimages $\tilde x$ and $\tilde x'$
in $S_2$ of $x$ and $x'$ we have $\tilde x < \tilde x'$. Since, as we just saw,
a tassel collapse preserves total order on the second level, we conclude that $h(x) < h(x')$ in $T_2$ and the condition of definition \ref{def_prenet} for $T$ implies that it holds for $P$ as well.

It remains to check that $g$ and $h$ are tassel collapse maps. First, as we saw both $g_2$ and $h_2$ preserve total orders on the second level, hence they also preserve partial orders
which establishes (a) of lemma \ref{lemma_easy_w}. Combining the fact that $f$ induced isomorphism between cores $c(S)$ and $c(T)$ and construction of $P$ we also see that both
$g$ and $h$ induce isomorphisms of cores. In particular they induce isomorphisms between core vertices at level $2$, which establishes \ref{lemma_easy_w} (b).
We need only to check that $g p_S = p_P g$ and $h p_P = p_T h$ as this would also imply \ref{lemma_easy_w} (c) for $g$ and $h$ respectively, as both $g_1$ and $h_1$ are isomorphisms at level $1$.
The second equality is just the definition of $p_P$, since $h$ at level $1$ is identity. For the first equality we first observe that $g(p_S(x)) = p_T(f(x))$, since $f$ is a tassel collapse.
And on the other hand, $p_P(g(x)) = p_T(hg(x)) = p_T(f(x))$ by definition of $p_P$.
\qed

\begin{definition}
\label{def_vc}
A map $f\from S \to T$ of prenets with tassels is called a {\em vertex collapse} if $f_2$ is a bijection between $S_2$ and $T_2$ preserving decomposition into core and tassels.
\end{definition}

Similarly, we say that a vertex collapse $f\from S \to T$ is {\em simple} if the preimages $f_1^{-1}(x)$ consist of a single element for all except possibly one vertex $x \in T_1$.
It is easy to see that a vertex collapse is a weak equivalence. Indeed, since bijection $f_2$ sends core vertices to core vertices we immediately see that condition (a) of definition
\ref{def_weak_equiv} holds. To check condition (b) we observe that any vertex $x \in S_2$ is the preimage of some vertex $y \in T_2$. Therefore, using (b) of definition \ref{def_prenet_map}
we find
$$
p_S(x) = p_S(f^{-1}(y)) = f^{-1}(p_T(y)) = f^{-1} p_T f(x).
$$

\begin{lemma}
\label{lemma_vc_factor}
Any vertex collapse map $f\from S \to T$ can be factorized into a chain of simple vertex collapses.
\end{lemma}
\proof
Arguing as in lemma \ref{lemma_tc_factor} we first construct a level-wise factorization
$$
\begin{tikzcd}
f\from S \ar[r, "g"] & P \ar[r, "h"] & T,
\end{tikzcd}
$$
by putting $P_2 = S_2 \isom T_2$ and defining $P_1$ as the pushout
$$
P_1 = \{y\} \bigsqcup_{f^{-1}(y)} S_1.
$$
Since on $S_1$ the partial order $\pless$ is the total order and map $f$ preserves it, we immediately see that $f^{-1}(y)$ is
an interval in $S_1$, and so the total order on $S_1$ induces a well defined total order on $P_1$. We define the parent map in $P$ as the composition $p_P(x) = g_1(p_S(x))$.

Let us check that $P$ is a prenet. Since $f$ is a weak equivalence, we have $\pless_S = \pless_f$, and moreover, since it is a vertex collapse, identifying $S_2$ with $T_2$
the latter partial order $\pless_f = \pless_T$. Since the partial order $\pless_P$ is squeezed between $\pless_S$ and $\pless_T$ we conclude that $\pless_S$ and $\pless_P$ coincide.
Now let $x$ and $x'$ be two incomparable vertices in $S_2 = P_2$ such that $x < x'$. Using the fact that partial orders $\pless_S$ and $\pless_P$ coincide, $x$ and $x'$ are also incomparable as vertices in $S$
and therefore $p_S(x) < p_S(x')$. By definition of $p_P$ this implies that $p_P(x) < p_P(x')$, as $x$ and $x'$ have no common parents in $P$.

It remains to check that $g$ and $h$ are maps of prenets. The condition (a) of lemma \ref{lemma_easy_w} follows immediately from the fact that the partial orders $\pless_S$, $\pless_P$ and $\pless_T$
coincide. Condition (b) is immediately satisfied, since both maps $g$ and $h$ are isomorphisms at level $2$. In order to verify condition (c) for $h$ observe that
$$
p_S(x) = f^{-1}p_T f(x),
$$
since $f$ is a weak equivalence. By definition of $p_P$, identifying $P_2 = S_2$ and because $g_1$ is a surjective map (\ref{lemma_prenet_map} (a')) we have
$$
p_P(x) = g p_S(x) = g f^{-1} p_T f(x) = h^{-1} p_T h(x).
$$
And for $g$ condition (c) follows from
$$
p_S(x) = f^{-1}p_T f(x) = g^{-1} h^{-1} p_T f(x) = g^{-1} p_P(g(x)).
$$
This implies that both maps $g$ and $h$ are weak equivalences by lemma \ref{lemma_easy_w} and therefore they are also vertex collapses.


\qed

For a general weak equivalence we have the following two factorization properties.

\begin{lemma}
\label{lemma_w_factor_tv}
Any weak equivalence can be factorized as a composition of a tassel collapse and a vertex collapse.
\end{lemma}
\proof
Consider factorization of $f$ constructed in lemma \ref{lemma_map_factor}.
$$
\begin{tikzcd}
f\from S \ar[r, "g"] & P \ar[r, "h"] & T.
\end{tikzcd}
$$
By construction map $h$ is a vertex collapse, hence a weak equivalence. Using the 2-out-of-3 property \ref{lemma_w_23} we conclude that $g$ is also a weak equivalence, and therefore induces
a bijection between core vertices $c(S_2)$ and $c(P_2)$. Since by construction of factorization, $g_1$ is also a bijection, we see that $g$ is a tassel collapse.
\qed

\begin{lemma}
\label{lemma_w_factor_vt}
Any weak equivalence can be factorized as a composition of a vertex collapse and a tassel collapse.
\end{lemma}
\proof
We begin by constructing a level-wise factorization
$$
\begin{tikzcd}
f\from S \ar[r, "g"] & P \ar[r, "h"] & T,
\end{tikzcd}
$$
by putting $P_1 = T_1$, $P_2 = S_2$ with $t(P) = t(S)$. And for the maps we put $g_1 = f_1$, $g_2 = \id_{S_2}$, $h_1 = \id_{T_1}$ and $h_2 = f_2$. We define parents of a vertex $x \in P_2 = S_2$ as $p_P(x) = p_T(h(x))$.
Since $f$ is a weak equivalence, by \ref{def_weak_equiv} (b) this is equivalent to $p_P(x) = g(p_S(x))$.

First we check that $P$ is a prenet. It is clear from the definition of $p_P$ that the partial order $\pless_P$ on $P_2 = S_2$ coincides with $\pless_f$ defined in lemma \ref{lemma_p_order}.
Therefore, the two partial orders $\pless_S$ and $\pless_P$ coincide. Let $x$ and $x'$ be two incomparable vertices in $P_2$, such that $x < x'$. Then arguing as in lemma \ref{lemma_vc_factor}
they are also incomparable as vertices in $S$, and therefore we have $p_S(x) < p_S(x')$ which in turn implies $p_P(x) < p_P(x')$.

Next, let us check that $h$ is a tassel collapse. Since $h_1 = \id_{T_1}$ and $h_2 = f_2$ induces a bijection between $c(S_2) = c(P_2)$ and $c(T_2)$ we see that $h$ induces an isomorphism between cores $c(P)$ and $c(T)$.
Since $\pless_S = \pless_P$ and $f$ preserves partial orders we see that $h$ also preserves partial orders, thus satisfying \ref{lemma_easy_w} (a). Condition (c) follows immediately from the definition of $p_P$.

Finally, let us check that $g$ is a vertex collapse. By construction $g_2$ is a bijection preserving decomposition into core and tassels, in particular it satisfied condition (b) of lemma \ref{lemma_easy_w}.
Since $g_1 = f_1$, it preserves the partial orders at level $1$, and since $\pless_S$ coincides with $\pless_P$ we see that $f_2 = \id_{S_2}$ also preserves the partial order, thus establishing condition (a).
Moreover, using definition of $p_P$ and the fact that $f$ is a weak equivalence we have
$$
p_S(x) = f^{-1} p_T f(x) = g^{-1} h^{-1} p_T h g(x) = g^{-1} p_P g(x),
$$
where the last identity follows from the fact that $h$ is a weak equivalence.

\qed

The following lifting property will be of a particular importance in the construction of the localized category.

\begin{proposition}
\label{prop_lift}
Let $f\from S \to T$ be a map of prenets and $g\from T' \to T$ a weak equivalence. Then there exists a commutative square in $2\PreNet$
$$
\begin{tikzcd}[sep=3em]
S' \ar[r, "f'"] \ar[d, "g'"'] & T' \ar[d, "g", ] \\
S \ar[r, "f"'] & T
\end{tikzcd}
$$
such that $g'$ is also a weak equivalence.
\end{proposition}
\proof
In virtue of the previous lemmas it is enough only to consider two cases, when $g$ is either a simple vertex collapse or a simple tassel collapse.

{\bf Simple vertex collapse.} Let $y \in T_1$ be the target of the vertex collapse $g$. Since $T'_2 = T_2$ we can put $S'_2 = S_2$ and $f'_2 = f_2$. Now, by definition
of vertex collapse the set of children of any vertex in $g^{-1}(y)$ is the same as the set of children of $y$ itself. Therefore, we can send any vertex from $f^{-1}(y)$
to any vertex in $g^{-1}(y)$ while satisfying \ref{lemma_prenet_map} (b'). Clearly, we can also do it while preserving the order on the first level. To address the issue of
surjectivity, if the size $|f^{-1}(y)|$ is less than $|g^{-1}(y)|$ then we can duplicate any vertex in $f^{-1}(y)$ until we have enough vertices to cover
$g^{-1}(y)$.

More formally, let $z$ be the last vertex of $f^{-1}(y)$ (with respect to the total order on $S_1$), and denote by $Z = g^{-1}(y)$, then we can put $S'_1 = (S_1 - \{z\}) \sqcup Z$.
We define $g'$ by putting it to be identity on $S_1 - \{z\}$ and sending $g^{-1}(y)$ to $z$. The total order on $S'_1$ is unambiguously determined by requiring
that $g'$ preserves the order and that on $Z$ it coincides with the total order on $g^{-1}(y) \subset T'_1$. We define $p_{S'}(x) = g'^{-1}(p_S(x))$ for $x \in S'_2$, which clearly makes
$g'$ a simple vertex collapse (with target $z$), hence a weak equivalence. Finally, we define $f'_1\from S'_1 \to T'_1$ as follows. Let $Y = f^{-1}(y) - \{z\} \subset S_1$,
then for $x \in S'_1 - Y - Z$ we put $f'(x) = g^{-1}fg'(x)$, which is well defined, since $g$ is a bijection outside of $y$. Any $x \in Y$ we send to the first vertex in $g^{-1}(y)$ with
respect to the total order on $T'_1$. And finally a vertex $x \in Z = g^{-1}(y)$ we send to the corresponding element in $g^{-1}(y) \subset T'_1$. It is straightforward to check
that $f'$ is a map of prenets.

{\bf Simple tassel collapse.} Let $y \in T_2$ be the target of the tassel collapse $g$. Since $T'_1 = T_1$ we can put $S'_1 = S_1$ and $f'_1 = f_1$.
Recall that as we saw in the proof of lemma \ref{lemma_tc_factor} $g^{-1}(y)$ is an interval in $T'_2$. Now, we have two possibilities,
if $y$ is a tassel then the preimage $g^{-1}(y)$ consists entirely of tassels, in which case we put $v$ to be the first vertex of $g^{-1}(y)$ with respect to the total order in $T'_2$.
On the other hand, if $y$ is a core vertex then there is exactly one core vertex in $g^{-1}(y)$ and in this case we put $v$ to be this unique core vertex. In either case we
denote by $V_1$ the set of tassels $\{x \in g^{-1}(y) \mid x < v\}$ and by $V_2$ the set of tassels $\{x \in g^{-1}(y) \mid x > v\}$.

Let $A = A(f^{-1}(y))$ and $B = B(f^{-1}(y))$ (in the notation of paragraph \ref{par_AB_notation}). We define
$$
S'_2 = A \times V_1 \sqcup S_2 \times \{v\} \sqcup B \times V_2, \quad t(S'_2) = A \times V_1 \sqcup t(S_2) \times \{v\} \sqcup B \times V_2,
$$
and order $S'_2$ lexicographically using total orders in $S_2$ and $T'_2$. Furthermore, we define the parent map by putting $p_{S'}(x, u) = p_S(x)$ for any vertex $(x, u) \in S'_2$.
It is easy to check that it satisfies the condition of definition \ref{def_prenet}. Indeed, if $(x, u)$ and $(x', u')$ are incomparable with respect to the partial order $\pless$ on $S'_2$
and $(x, u) < (x', u')$, then $x$ and $x'$ are incomparable with respect to $\pless$ on $S_2$ and $x < x'$ in $S_2$, which implies that $p(x) < p(x')$.

We define $g'_2(x, u) = x$ for any vertex $(x, u) \in S'_2$ and put
$$
f'_2(x, u) = \begin{cases}
u,&\text{if $u \neq v$},\\
v,&\text{if $u = v$ and $f(x) = y$},\\
f(x),&\text{if $u = v$ and $f(x) \neq y$}.
\end{cases}
$$

Let us first check that $g'$ is a tassel collapse. By construction it is clear that $g'_2$ preserves the partial orders $\pless$ on $S'_2$ and $S_2$, and by definition of tassels in $S'$
it induces bijection between $c(S'_2)$ and $c(S_2)$. Moreover, by definition of $p_{S'}$ we have $g'p_{S'} = p_S g'$, and since $g'_1$ is the identity we obtain property \ref{lemma_easy_w} (c).
Hence $g'$ is a weak equivalence, and since it is identity at level $1$ it is a tassel collapse.

It remains to check that $f'$ is a map of prenets. First, we notice that if $v$ is a tassel, then $y$ is a tassel and the entire preimage $f^{-1}(y)$ consists of tassels. Therefore $f'$
sends core vertices to core vertices. Next we check that it preserves the partial orders at level $2$. Consider two vertices $(x, u), (z, w) \in S'_2$. Since $g'$ is
a map $(x, u) \pless_{S'} (z, w)$ implies that $x \pless_S z$ and since $f$ is a map we have $f(x) \pless_T f(z)$. By construction $gf' = fg'$, therefore we have
$f'(x, u) \pless_g f'(z, w)$, and since $g$ is a weak equivalence, by lemma \ref{lemma_p_order} we conclude that $f'(x, u) \pless_{T'} f'(z, w)$.

Consider a vertex $t \in T'_2$, then since $g$ is a tassel collapse we have $p_{T'}(t) = p_T(g(t))$. If $t \not\in g^{-1}(y)$ then the preimage $f'^{-1}(t) = f^{-1}(g(t)) \times \{v\} \subset S'_2$.
And for the parents we have
$$
p_{S'}f'^{-1}(t) = p_S f^{-1} g(t) = f^{-1} p_T g(t) = f'^{-1} p_{T'}(t).
$$
If $t = v$ then $f'^{-1}(t) = f^{-1}(y) \times \{v\}$ and the same argument applies. If $t \in V_1$ then $f'^{-1}(t) = A \times \{t\}$. Since $g$ is vertex collapse, the parents $p_{T'}(t) = p_T(y)$.
And using the fact that by definition $A$ forms a disjoint cover of $f^{-1} p_T(y)$ we have
$$
p_{S'} f'^{-1}(t) = p_S(A) = f^{-1} p_T(y) = f'^{-1} p_{T'}(t).
$$
The case $t \in V_2$ is handled similarly, thus we have established property \ref{def_prenet_map} (b).

To check property \ref{def_prenet_map} (c) let us look at each of the aboves cases once again. If $t \not\in g^{-1}(y)$ or $t = v$ then $f'^{-1}(t) = f^{-1}(g(t)) \times \{v\}$, the subsets
$A(f'^{-1}(t)) = A(f^{-1} g(t)) \times \{v\}$ and $B(f'^{-1}(t)) = B(f^{-1} g(t)) \times \{v\}$ and we conclude by property (c) for map $f$. It $t \in V_1$ then using the fact that
$A$ is a disjoint cover of $f^{-1} p_T(y)$ we see that
$$
A(f'^{-1}(t)) = B(f'^{-1}(t)) = A \times \{t\}.
$$
Hence they are both disjoint covers of $f^{-1} p_T(y)$. The case $t \in V_2$ is handled similarly.
\qed

\begin{remark}
\label{rem_lift_w}
It is clear that constructions in both cases follow a similar pattern. To highlight it let us consider the case when both $f$ and $g$ are weak equivalences. In this situation the construction for the tassel
collapse significantly simplifies, since both subsets $A$ and $B$ consist of a single vertex, and we can combine two cases into one.

If $f$ and $g$ are weak equivalences, then both of them preserve the total orders on all levels. So for every $i$ we obtain a subdivision of $S_i$ into intervals $f_i^{-1}(y)$
indexed by vertices $y \in T_i$, and similarly $T'_i$ is subdivided into intervals $g_i^{-1}(y)$. For every $y$ we pick a vertex $v_y \in g_i^{-1}(y)$. When $y$ is a core vertex
on level $2$ we have to take $v_y$ to be the unique core vertex of $g_i^{-1}(y)$ in order for maps to be well defined, in all other cases $v_y$ can be chosen arbitrary, but
to be specific we take $v_y$ to be the minimal vertex in $g_i^{-1}(y)$. Denote by $g_i^{-1}(y)_{< v_y}$ the subset of $g_i^{-1}(y)$ consisting of elements less than $v_y$,
and similarly for $g_i^{-1}(y)_{> v_y}$. We then define merge of the two preimages at $v_y$ to be the totally ordered set
$$
\mu(y, v_y) := (g_i^{-1}(y)_{< v_y}, f_i^{-1}(y), g_i^{-1}(y)_{> v_y}).
$$
Then we can put
$$
S'_i := \bigsqcup_{y \in T_i} \mu(y, v_y)
$$
with the total order induced by the total order on $T_i$ and those of $\mu(y, v_y)$. We define map $f'_i$ on each merge $\mu(y, v_y)$ by putting it to identity on $g_i^{-1}(y)_{< v_y}$
and $g_i^{-1}(y)_{> v_y}$, and sending $f_i^{-1}(y)$ to $v_y$. Similarly, we define $g'_i$ by putting it to identity on $f_i^{-1}(y)$ and by sending $g_i^{-1}(y)_{< v_y}$
and $g_i^{-1}(y)_{> v_y}$ to the minimal and maximal vertices of $f_i^{-1}(y)$ respectively. Finally, we define the parent map on $S'$ as the composition
$$
p_{S'} = f'^{-1} g^{-1} p_T g f'.
$$
\end{remark}

\begin{nparagraph}
It may be helpful to present the construction of proposition \ref{prop_lift} visually. To simplify the picture we will again consider the case when both $f$ and $g$ are weak
equivalences and look at the construction on the second level. Denote
$$
R = S_2 \times_{T_2} T'_2 \subset S_2 \times T'_2.
$$
Since $f_2$ and $g_2$ preserve the total order, the subset $R$ is a union of rectangles
$$
R = \bigcup_{y \in T_2} f_2^{-1}(y) \times g_2^{-1}(y).
$$
The set $S'_2$ constructed in the proposition \ref{prop_lift} is a path inside $R$ passing through pairs $(w_y, v_y)$, where $w_y$ and $v_y$ are the core vertices in the preimages
$f_2^{-1}(y)$ and $g_2^{-1}(y)$ respectively, if they exist. In the picture below the shaded rectangles form the subset $R$, the bold dots are the pairs of core vertices and the thick blue line depicts
the path $S'_2$.
\vskip 1em
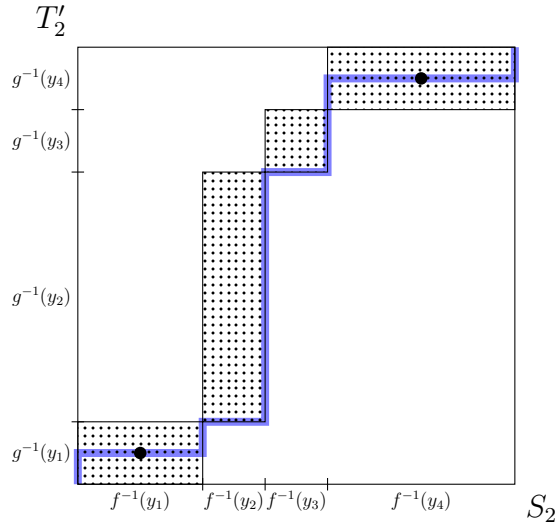
$$
\begin{tikzpicture}
\def\u{2em}
\def\dx{7*\u}
\def\dy{7*\u}

\draw [-,color=blue!50,line width=3pt] (0, 0) \foreach \x/\y in {0/0.5, 2/0.5, 2/1, 3/1, 3/5, 4/5, 4/6.5, 7/6.5, 7/7} { -- (\x * \u, \y * \u) };
\draw (0, 0) rectangle (\dx, \dy);
\draw (\dx, 0) node[below right] {$S_2$};
\draw (0, \dy) node[above left] {$T'_2$};
\foreach \x/\y in {1/0.5, 5.5/6.5} \fill (\x * \u, \y * \u) circle [radius=0.2em];

\foreach \x in {2, 3, 4} \draw (\x * \u, -0.2em) -- (\x * \u, 0.2em);
\foreach \y in {1, 5, 6} \draw (-0.2em, \y * \u) -- (0.2em, \y * \u);
\foreach \x/\l in {1/1, 2.5/2, 3.5/3, 5.5/4} \draw (\x * \u, 0) node[below, scale=0.6] {$f^{-1}(y_\l)$};
\foreach \y/\l in {0.5/1, 3/2, 5.5/3, 6.5/4} \draw (0, \y * \u) node[left, scale=0.6] {$g^{-1}(y_\l)$};

\draw[pattern=dots] (0, 0) \foreach \x/\y in {2/1, 3/5, 4/6, 7/7} { rectangle (\x*\u, \y*\u) };
\end{tikzpicture}
$$
\vskip 1em
The total order on each of the segments of the path is induced from the total orders on $S_2$ and $T'_2$ and we order the segments as prescribed by the path. The two maps $f'$ and $g'$
are projections on $T'_2$ and $S_2$ respectively.

Clearly, such a path is not unique and the one we chose is not even minimal. In principle we could have taken any path inside $R$ passing through the core vertices.

\end{nparagraph}

\begin{corollary}
\label{cor_w_lift}
Every cospan $\begin{tikzcd}[cramped,sep=small]T \ar[r, "g'"] & Q & P \ar[l, "f'"'] \end{tikzcd}$ of weak equivalences
in $2\PreNet$ can be completed to a commutative square, such that all maps are weak equivalences.
$$
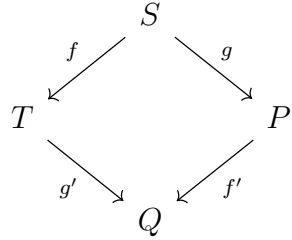
\begin{tikzcd}
& S \ar[ld, "f"'] \ar[rd, "g"] & \\
T \ar[rd, "g'"'] & & P \ar[ld, "f'"] \\
& Q &
\end{tikzcd}
$$
\end{corollary}
\proof
The statement follows from combination of proposition \ref{prop_lift} and 2-out-of-3 property of weak equivalences \ref{lemma_w_23}.
\qed

\begin{lemma}
\label{lemma_ww_pushout}
Every span $\begin{tikzcd}[cramped,sep=small]T & S \ar[l, "f"'] \ar[r, "g"] & P\end{tikzcd}$ of weak equivalences in $2\PreNet$ can be completed to a commutative square.
$$
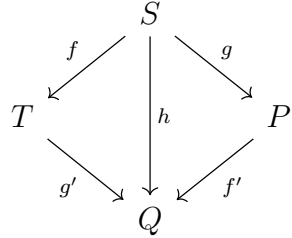
\begin{tikzcd}
& S \ar[dd, "h"] \ar[ld, "f"'] \ar[rd, "g"] & \\
T \ar[rd, "g'"'] & & P \ar[ld, "f'"] \\
& Q &
\end{tikzcd}
$$
Moreover, if for every core vertex $x \in c(Q_2)$ the preimage $h^{-1}(x)$ contains exactly one core vertex, then both $g'$ and $f'$ are also weak equivalences.
\end{lemma}
\proof
We begin the proof by replacing prenets $S$, $T$, $P$ with $S^\circ$, $T^\circ$ and $P^\circ$ respectively.
We define $Q$ and maps $f'$ and $g'$ level-wise via pushout squares.
$$
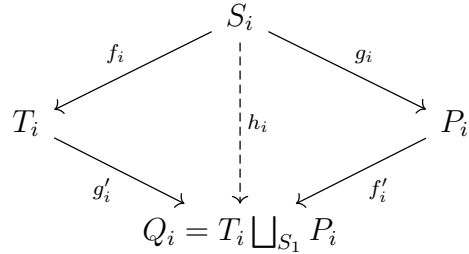
\begin{tikzcd}
& S_i \ar[ld, "f_i"'] \ar[rd, "g_i"] \ar[dd, dashed, "h_i"]& \\
T_i \ar[rd, "g'_i"'] & & P_i \ar[ld, "f'_i"] \\
& Q_i = T_i \bigsqcup_{S_1} P_i &
\end{tikzcd}
$$

Observe that for any $y \in T_i$ and $z \in P_i$ the preimages $f_i^{-1}(y)$ and $g_i^{-1}(z)$ are intervals in $S_i$. For $i = 1$ it follows from the fact that maps $f_1$ and $g_1$ preserve the
total orders, and for $i = 2$ it follows by a combination of arguments in the proof of lemma \ref{lemma_tc_factor} and lemma \ref{lemma_w_factor_tv}. Therefore for any vertex $x \in Q_i$
the preimage $h_i^{-1}(x)$ is also an interval in $S_i$, and so the total order on $S_i$ descends to the total order on $Q_i$.

Since at level $1$ all four arrows are surjective maps of totally ordered sets, they map intervals to intervals. Consider two compositions
$$
\begin{tikzcd}[row sep=0.5em]
q_T\from T_2 \ar[r, "p_T"] & \Int(T_1) \ar[r] & \Int(Q_1),\\
q_P\from P_2 \ar[r, "p_P"] & \Int(P_1) \ar[r] & \Int(Q_1).
\end{tikzcd}
$$
Since $f$ and $g$ are weak equivalences we have $q_T f_2 = q_P g_2$, and therefore by the universal property of the coproduct we obtain the parent map
$p_Q\from Q_2 \to \Int(Q_1)$. Let $x$ and $x'$ be two vertices in $Q_2$ which are incomparable with respect to the partial order $\pless$ on $Q$ and such that $x < x'$. Then for any lifts $\tilde x$ and $\tilde x'$ to $T_2$
of $x$ and $x'$ respectively we see that $\tilde x$ and $\tilde x'$ are also incomparable in $T$, therefore $p_T(\tilde x) < p_T(\tilde x')$ and hence $p_Q(x) < p_Q(x')$. This shows that $Q$ is a prenet,
and we put $t(Q) = Q_2$.

Let us show that $f'$ and $g'$ are weak equivalences. Since the picture is symmetric it is enough to focus just on one of these maps, say $g'$. As was already discussed $g'$ preserves the partial
orders $\pless$ on $T$ and $Q$ (in fact they even preserve the total order). Since the sets of core vertices at level $2$ in all four prenets are empty, $g'$ obviously induces bijection $c(T_2) \isom c(Q_2)$.
It remains to show that \ref{lemma_easy_w} (c) holds. Take $y \in T_2$, and a vertex $\tilde y \in f^{-1}(y)$ then it is enough to check that
$$
p_S(\tilde y) = f^{-1} p_T(y) = f^{-1} g'^{-1} p_Q g'(y) = h^{-1} p_Q h(\tilde y) = h^{-1} h p_S(\tilde y).
$$
Assume that $p_S(\tilde y) \subsetneq h^{-1}h p_S(\tilde y)$, then there is a vertex $t \in S_1$ and $t' \in p_S(\tilde y)$ such that either $f(t) = f(t')$ or $g(t) = g(t')$.
Both cases are treated similarly, so to be specific, assume that latter is the case, but then $g(t)$ is a parent of $g(\tilde y)$ in $P$, and since $g$ is a weak equivalence
we conclude that $t \in p_S(\tilde y)$, which contradicts the assumption. This proves that $g'$ is a weak equivalence.

Now let us recall the core vertices in the original prenets $S$, $T$, $P$, and define core vertices in $Q_2$ by saying that $x \in c(Q_2)$ if its preimage $h^{-1}(x)$ contains a core vertex. Clearly,
with this definition $g'$ and $f'$ remain maps of prenets, however, they may no longer be weak equivalences, as condition \ref{def_weak_equiv} (a) may no longer be satisfied. However,
under assumption that the preimage $h^{-1}(x)$ contains exactly one core vertex, combined with the fact that both $f$ and $g$ induce bijection between core vertices at level $2$, we see
that both $f'$ and $g'$ also induce bijections between core vertices at level $2$, hence are weak equivalences.
\qed

\begin{corollary}
\label{cor_vcw_pushout}
Every span $\begin{tikzcd}[cramped,sep=small]T & S \ar[l, "f"'] \ar[r, "g"] & P\end{tikzcd}$, where $f$ is a vertex collapse and $g$ is a weak equivalences in $2\PreNet$, can be completed to a commutative square
as in lemma \ref{lemma_ww_pushout}, where $f'$ is a vertex collapse and $g'$ is a weak equivalence.
\end{corollary}
\proof
Since $f$ is a vertex collapse, $f_2$ is a bijection, hence $f'_2$ is also a bijection. Therefore, combined with the fact that $g$ induces bijection $c(S_2) \isom c(P_2)$ we conclude
that for every vertex $x \in Q_2$ the preimage $h^{-1}(x)$ contains exactly one core vertex.
\qed

The following proposition will allow us to refer to a weak equivalence class of prenets unambiguously by its minimal representative, as well as facilitate  description
of maps in the localized category. We will denote by $\cW(S)$ the connected component of $\cW$ containing prenet $S$.

\begin{definition}
A prenet $S$ is called {\em minimal} if the only weak equivalence map $f\from S \to T$ is the identity.
\end{definition}

\begin{proposition}
\label{prop_min_prenet}
Every connected component $\cW(S)$ contains a unique minimal prenet $m(S)$. Moreover, there exists, not necessarily unique, weak equivalence $S \to m(S)$.
\end{proposition}
\proof
First we establish existence of a minimal prenet. Consider an object $S \in \cW$, if it is minimal then we are done. If it is not minimal, then there exists
a weak equivalence $f\from S \to T$ which is not identity. Since, weak equivalences preserve total orders on each level, $f$ can not be an isomorphism, therefore the total
size of $T$, i.e. $\sum_i |T_i|$ has to be strictly less than that of $S$. And because prenets are finite, we conclude that this process has to terminate at a minimal object.

Now, let us show the uniqueness. Let $S$ and $T$ be two minimal objects in the same connected component of $\cW$. Then they can be connected by a chain of objects $\{P_1, \ldots, P_n\}$ and
a zig-zag of weak equivalences.
$$
\begin{tikzcd}
S & P_1 \ar[l, "f"'] \ar[r, "g"] & P_2 & \cdots \ar[l] \ar[r] & P_{n-1} & P_n \ar[l] \ar[r] & T
\end{tikzcd}
$$
By lemma \ref{lemma_w_factor_vt} we factorize $g = g_t g_v$, where $g_v$ is a vertex collapse and $g_t$ is a tassel collapse. In view of corollary \ref{cor_vcw_pushout} we can form a pushout $Q$,
and since by assumption $S$ was minimal we have $Q = S$.
$$
\begin{tikzcd}[sep=3em]
P_1 \ar[d, "f"'] \ar[r, "g_v"] & P'_1 \ar[d, "f'"'] \ar[r, "g_t"] & P_2 \ar[dashed, dl, "w"'] \ar[d, "u"] \\
S \ar[r, "="] & Q \ar[r, "v"] & R
\end{tikzcd}
$$
Now, applying lemma \ref{lemma_ww_pushout} we can construct pushout $R$. If for every core vertex $x \in c(R_2)$ preimages $u^{-1}(x)$ and $v^{-1}(x)$ contain exactly one core vertex, then
both $u$ and $v$ are weak equivalences and again by minimality of $S$ we conclude that $R = S$, thus $u$ gives us a map $P_2 \to S$.

Let us consider a core vertex $x \in c(R_2)$ such that
preimages $u^{-1}(x)$ and $v^{-1}(x)$ contain more than one vertex. Observe, that we necessarily have $|u^{-1}(x)| = |v^{-1}(x)|$. Moreover, the subsets $u^{-1}(x)$ and $v^{-1}(x)$ are
intervals in $(P_2)_2$ and $Q_2$ respectively, therefore by minimality of $Q = S$ we conclude that $v^{-1}(x)$ has only core vertices. Hence we can construct a map $u^{-1}(x) \to v^{-1}(x)$
bijective on the core vertices, which preserves the order.

Now, if $x \in t(R_2)$ is a tassel, then $v^{-1}(x)$ consists of a single tassel, and again we can construct a map $u^{-1}(x) \to v^{-1}(x)$ in an obvious manner. Taking coproduct
over all vertices $x \in R_2$ we obtain a map $w_2\from (P_2)_2 \to Q_2$. Furthermore, since by construction of the pushout $R$ we have $R_1 = Q_1$ we can put $w_1 = u_1$. It is
clear that the map $w$ obtained this way is a tassel collapse map, hence a weak equivalence.

In effect we have connected prenets $S$ and $T$ by a shorter zig-zag of weak equivalences, and proceeding inductively on $n$ we conclude that $S = T$.

Finally, let $T = m(S)$ be the minimal prenet, then using the same inductive argument on the length of the zig-zag connecting $S$ and $T$ we construct a weak equivalence $S \to T$.
\qed

\begin{lemma}
\label{lemma_term_prenet}
Let $\cW'(S) \into \cW(S)$ be the subcategory consisting of weak equivalences $f\from S \to T$, such that for every vertex $x \in c(T_2)$ the preimage
$f^{-1}(x)$ consists of a single vertex, which is necessarily a core vertex of $S_2$. Then the category $\cW'(S)$ has a terminal object.
\end{lemma}
\proof
The argument goes the same as in the proof of proposition \ref{prop_min_prenet}. Except when we deal with the zig-zag connecting $S$ and $T$ we no longer need to factorize
map $g$ into a vertex collapse and a tassel collapse. Instead we apply lemma \ref{lemma_ww_pushout} directly to pair $(f, g)$ to form the pushout $Q$.
Because of the additional condition that maps of $\cW'(S)$ satisfy we see that all four maps of the pushout square are weak equivalences, and in fact belong to $\cW'(S)$.
Hence we conclude the uniqueness of the minimal prenet, that we denote by $m'(S)$ and existence of a map $S \to m'(S)$ in the component $\cW'(S)$.

Now let $T = m'(S)$ be the minimal prenet in $\cW'(S)$ and consider two weak equivalences $f,g\from S \to T$. Then again by lemma \ref{lemma_ww_pushout} we can construct a commutative square
$$
\begin{tikzcd}
& S \ar[dl, "f"'] \ar[dr, "g"] & \\
T \ar[dr, "g'"'] && T, \ar[dl, "f'"] \\
& Q &
\end{tikzcd}
$$
where all four maps are weak equivalences in $\cW'(S)$. Furthermore, as $T$ is minimal, maps $f'$ and $g'$ are identities, and therefore $f = g$.
\qed

\begin{corollary}
\label{cor_term_prenet}
If $S$ is a prenet such that $t(S) = S_2$, then the minimal prenet $m(S)$ is the terminal object of $\cW(S)$.
\end{corollary}
\proof
Follows immediately from the previous lemma, as in this case we have $\cW'(S) = \cW(S)$.

\qed

We will also make use of the following partial minimality properties.

\begin{definition}
\label{def_12_min}
A prenet $S$ is called {\em $1$-minimal} if the only vertex collapse map $f\from S \to T$ is the identity. And $S$ is called {\em $2$-minimal} if the only tassel collapse map $f\from S \to T$ is the identity.
\end{definition}

Let us denote by $\mathcal{VC}(S)$ and $\mathcal{TC}(S)$ the connected components of $S$ in the subcategories of vertex collapses and tassel collapses respectively. Then using the same argument as in the proof
of proposition \ref{prop_min_prenet} we conclude that $\mathcal{VC}(S)$ has a unique $1$-minimal prenet, denoted by $m_1(S)$ and similarly, $\mathcal{TC}(S)$ also has a unique $2$-minimal prenet,
denoted by $m_2(S)$.

\begin{corollary}
\label{cor_vc_term_prenet}
The $1$-minimal prenet $m_1(S)$ is a terminal object of the connected component $\mathcal{VC}(S)$.
\end{corollary}
\proof
We use the same argument as in the proof of corollary \ref{lemma_term_prenet}, substituting \ref{cor_vcw_pushout} in place of lemma \ref{lemma_ww_pushout}.
\qed

We will denote by $\cW(S)_i$ the full subcategory of $\cW(S)$ consisting of $i$-minimal prenets, for $i = 1, 2$.

\begin{remark}
If in the setup of lemma \ref{lemma_ww_pushout} we do not require that both maps $f$ and $g$ are weak equivalences, then in general we can not construct a pushout prenet $Q$
together with maps $f'$ and $g'$ (for an example we refer to \ref{exa_no_pushout}). However, we have the following version of the pushout lemma.
\end{remark}

\begin{lemma}
\label{lemma_vc_pushout}
Every span $\begin{tikzcd}[cramped,sep=small]T & S \ar[l, "f"'] \ar[r, "g"] & P\end{tikzcd}$, where $f$ is a vertex collapse and $g$ any map in $2\PreNet$, can be completed to a commutative square
$$
\begin{tikzcd}
& S \ar[ld, "f"'] \ar[rd, "g"] & \\
T \ar[rd, "g'"'] & & P \ar[ld, "f'"] \\
& Q &
\end{tikzcd}
$$
such that $f'$ is a vertex collapse.
\end{lemma}
\proof
Since at first level both maps $f_1$ and $g_1$ preserve total orders, the construction of $Q_1$ is done the same way as in the proof of lemma \ref{lemma_ww_pushout}.
Moreover, as $f_2$ is an isomorphism, we put $Q_2 = P_2$ and $t(Q) = t(P)$ and for the maps $g'_2 = g_2 f_2^{-1}$ and $f'_2 = \id_{P_2}$.
For any $x \in Q_2 = P_2$ define the parents $p_Q(x) = f'_1(p_P(x))$. Clearly this satisfies the
requirement of definition \ref{def_prenet}, indeed if $x < x'$ in $Q_2$ such that $x$ and $x'$ are incomparable with respect to $\pless$, then they are also incomparable with respect to $\pless$ in $P$,
therefore $p_P(x) < p_P(x')$ and since $f'_1$ preserves the total order we have $p_Q(x) < p_Q(x')$.

The constructed maps $f'_i$ obviously satisfy conditions (a) and (b) of lemma \ref{lemma_easy_w}, so it remains to check condition (c). Take a vertex $y \in P_2$
and consider the subset $V := p_S(g^{-1}(y)) = g^{-1}p_P(y)$. Since $f$ is a vertex collapse, we have $f^{-1}f(V) = V$, and therefore $f'^{-1} p_Q(y) = f'^{-1} f' p_P(y) = p_P(y)$.
Hence, $f'$ is a weak equivalence and a vertex collapse.

It remains to show that $g'$ is a map of prenets. By construction it satisfies condition (a) of \ref{def_prenet_map}. Let us take a vertex $x \in Q_2$ and show that
$g'^{-1}(p_Q(x)) = p_T(g'^{-1}(x))$. It is enough to check that
$$
f^{-1} g'^{-1}(p_Q(x)) = g^{-1} f'^{-1}(p_Q(x)) = f^{-1} p_T g'^{-1}(x) = p_S f^{-1} g'^{-1}(x),
$$
where in the last equality we used the property \ref{def_prenet_map} (b) for $f$. Using what we established for $p_Q$ we find
$$
g^{-1} f'^{-1}(p_Q(x)) = g^{-1} p_P(x) = p_S g^{-1}(x),
$$
which coincides with the right hand side of the previous equation. For the last property \ref{def_prenet_map} (c) the cases of subsets $A$ and $B$ are handled similarly.
We observe that in virtue of $f$ being a vertex collapse $f_2$ induces a bijection between subsets $A(f^{-1}g'^{-1}(x))$ and $A(g'^{-1}(x))$. And since the composition
$g' f = f' g$ is a map of prenets the former is a disjoint cover of $p_S(f^{-1} g'^{-1}(x))$ and hence the latter is also a disjoint cover of $p_T g'^{-1}(x)$.
\qed

\vskip 5em


\section{2-Nets with tassels}
\label{sec_nets}

We begin this section with the definition of the category of $2$-nets and a description of the spaces of maps between two nets.

\begin{definition}
The category of $2$-nets is the localization of the category of connected, pruned $2$-prenets with tassels with respect to the subcategory of weak equivalences.
$$
2\Net = 2\PreNet[\cW^{-1}].
$$
\end{definition}
We denote the localization functor by $L\from 2\PreNet \to 2\Net$.
We can also consider a full subcategory of $2\Net$ consisting of special prenets that will be denoted by $2\Net^s$. There is a natural comparison functor
$$
2\PreNet^s[\cW^{-1}] \to 2\Net^s,
$$
and although we do not know whether it is an equivalence, we would like to point out here a related result of corollary \ref{cor_comp_special}.

In general it can be difficult to describe effectively the spaces of maps in the localized category. In what follows we will provide a practical way to compute them,
in a spirit reminiscent of ideas from both Ore and Quillen localizations.

We begin with the following simple lemma regarding maps in the localized category.

\begin{lemma}
\label{lemma_vc_tc_pair}
Consider a pair of maps of prenets $\begin{tikzcd}[cramped,sep=small]S \ar[shift left, r, "f"] \ar[shift right, r, "g"'] & T\end{tikzcd}$.
If $f$ and $g$ are either both vertex collapse maps or both tassel collapse maps then their images in the category of $2$-nets coincide.
\end{lemma}
\proof
We consider each case separately.

{\bf A.} Let $f$ and $g$ be vertex collapses, then the component $\mathcal{VC}(S)$ contains both prenets $S$ and $T$. As was stated in corollary \ref{cor_vc_term_prenet}
the component $\mathcal{VC}(S)$ has a terminal object $m_1(S)$, therefore in the diagram
$$
\begin{tikzcd}
S \ar[shift left, r, "f"] \ar[shift right, r, "g"'] & T \ar[r, "z"] & m_1(S),
\end{tikzcd}
$$
where $z$ denotes the unique vertex collapse to the terminal object $m_1(S)$ we have $zf = zg$. Since $z$ is an isomorphism in the localized category, we conclude that
images of $f$ and $g$ coincide.

{\bf B.} Assume that $f$ and $g$ are tassel collapses. We construct a prenet $\wtilde T$ by putting $\wtilde T_0 = T_0$, $\wtilde T_1 = T_1$ and on the second level
for each core vertex $x \in c(T_2)$ we add two new tassels $x^-$ and $x^+$ immediately before and after $x$ respectively. In other words, we define the total order
on $\wtilde T_2$ by $y < x^- < x$ for all $y < x$ and $x < x^+ < z$ for all $x < z$. We define parents $p_{\wtilde T}(x^-) = p_{\wtilde T}(x^+) = p_T(x)$ and
for the rest of the vertices parents are the same as in $T$. Clearly the map that sends $x^-$ and $x^+$ to $x$ and is identity of the rest of the vertices is
a tassel collapse $w\from \wtilde T \to T$.

Next, we lift maps $f$ and $g$ to $\tilde f, \tilde g \from S \to \wtilde T$. For a vertex $y \in c(T_2)$ we write $\tilde y \in c(S_2)$ for the unique core vertex in $f^{-1}(y)$.
Then for any vertex $x \in S_2$ we put
$$
\tilde f(x) = \begin{cases}
f(x),&\text{if $x \in c(S_2)$ or $f(x) \in t(T_2)$},\\
f(x)^-,&\text{if $x < \wtilde{f(x)}$},\\
f(x)^+,&\text{if $\wtilde{f(x)} < x$}.
\end{cases}
$$
Similarly we define $\tilde g$. It is straightforward to check that both $\tilde f$ and $\tilde g$ are tassel collapses and that $f = w\tilde f$ and $g = w \tilde g$.
$$
\begin{tikzcd}[sep=3em]
S \ar[dr, "z_S"'] \ar[shift left, r, "\tilde f"] \ar[shift right, r, "\tilde g"'] & \wtilde T \ar[d, "z_{\wtilde T}"] \ar[r, "w"] & T.\\
& m'(S) &
\end{tikzcd}
$$

The constructed maps $\tilde f$ and $\tilde g$ in fact lie in the subcategory $\cW'(S)$ from lemma \ref{lemma_term_prenet}. Denote by $m'(S)$ the terminal object of
$\cW'(S)$ and by $z_S$ and $z_{\wtilde T}$ the corresponding unique maps. Since $z_{\wtilde T}$ is an isomorphism in the localized category we see that the images of $\tilde f$
and $\tilde g$ coincide in the category $2\Net$. And therefore images of $f$ and $g$ also coincide.
\qed

In what follows, an important role will be played by the notion of a fiber of a map of prenets $f\from S \to T$ over a vertex $x \in T_2$, which can be either a core vertex
or a tassel. First, we define a prenet $\<x\> \subset T$ generated by the vertex $x$. Here the inclusion is purely set theoretical, and does not correspond to a map
in the category $2\PreNet$. Level-wise we put $\<x\>_2 = \{x\}$, $\<x\>_1 = p_T(x)$ and $\<x\>_0 = T_0$ with the parent map obviously induced by $p_T$.

\begin{definition}
\label{def_fiber}
The fiber of $f\from S \to T$ over a vertex $x \in T_2$ is a prenet $F$ with levels
$$
F_2 = f_2^{-1}(x) \subset S_2,\quad\quad F_1 = f_1^{-1}(p_T(x)) = p_S(F_2),
$$
and the parent map induced by $p_S$.
\end{definition}
We will often write $F = f^{-1}(x)$, as usually there is no confusion whether we mean the set-theoretic preimage of the vertex $x$ or the prenet-theoretic fiber.
It is immediate to check that $F$ is indeed a prenet, moreover it comes with a map of prenets
$$
F = f^{-1}(x) \to \<x\>.
$$

\begin{lemma}
\label{lemma_fib_w}
Let $f\from S \to T$ be a weak equivalence, then the fiber $f^{-1}(x)$ is weakly equivalent to $I$ if $x \in c(T_2)$ and weakly equivalent to $I^\circ$ is $x \in t(T_2)$.
\end{lemma}
\proof
Let $F = f^{-1}(x)$, then level-wise there is only one map $z\from F \to I$, which can be factored through $z^\circ \from F \to I^\circ$ if there are no core vertices in $F_2$.
We will use lemma \ref{lemma_easy_w} to show that they are weak equivalences. Indeed, property (a) is trivially satisfied. If $x$ is a core vertex in $T_2$ then $F_2$ contains
exactly one core vertex and map $z\from F \to I$ satisfies property (b). Similarly, if $x$ is a tassel in $T_2$ then there are no core vertices in $F_2$ and $z^\circ\from F \to I^\circ$
satisfies (b). Finally, since $f$ is a weak equivalence, for every vertex $y \in F_2$ we have $p_S(y) = f_1^{-1}p_T(x) = F_1$, hence property (c) is also fulfilled.
\qed

Next we describe an analog of the homotopy equivalence between maps.

\begin{definition}
\label{def_equiv}
Denote by $\sim$ an equivalence relation on the sets of maps $\Hom_{2\PreNet}(S, T)$ generated by pairs
$\begin{tikzcd}[cramped,sep=small]S \ar[shift left, r, "f"] \ar[shift right, r, "g"'] & T\end{tikzcd}$ such that the following conditions hold.
\begin{enumerate}[label=\alph*)]
\item Factorizations of $f$ and $g$ via a vertex collapse constructed in lemma \ref{lemma_map_factor} pass through the same prenet $T'$.
$$
\begin{tikzcd}
S \ar[shift left, r, "f'"] \ar[shift right, r, "g'"'] & T' \ar[shift left, r, "f_v"] \ar[shift right, r, "g_v"'] & T.
\end{tikzcd}
$$
\item The two maps $f'$ and $g'$ can be factorized via $S'$
$$
\begin{tikzcd}
S \ar[r, "h"] & S' \ar[shift left, r, "f_w"] \ar[shift right, r, "g_w"'] & T' \ar[shift left, r, "f_v"] \ar[shift right, r, "g_v"'] & T,
\end{tikzcd}
$$
such that the fibers of $f_w$ and $g_w$ are pairwise weakly equivalent.
\end{enumerate}
\end{definition}

The advantage of this notion over the general equivalence of maps in the localized category is that while we
require existence of some prenet $S'$ in the factorization, it is squeezed between $S$ and $T'$, and so there is only a finite number of prenets that need to be considered. Moreover,
the verification of the pairwise equivalence of fibers of $f_w$ and $g_w$ is also a finite process due to the existence of minimal prenets.

We will denote the set of equivalence classes of maps between prenets $S$ and $T$ by
$$
\pi\Hom(S, T) := \Hom_{2\PreNet}(S, T) / \sim.
$$
It is easy to see that the sets $\pi\Hom$ satisfy the necessary functoriality. Let $l\from P \to S$ be a map of prenets the we can construct the following diagram.
\begin{equation}
\label{equ_piHom_pull}
\begin{tikzcd}[sep=3em]
& P \ar[dl, "\tilde l"'] \ar[d] \ar[shift left, dr] \ar[shift right, dr] && \\
\wtilde S \ar[d, "q"'] \ar[r, "\tilde h"] & \wtilde S' \ar[d, "q'"'] \ar[shift left, r, "\tilde f_w"] \ar[shift right, r, "\tilde g_w"'] & \wtilde T' \ar[d, "r"] & \\
S \ar[r, "h"] & S' \ar[shift left, r, "f_w"] \ar[shift right, r, "g_w"'] & T' \ar[shift left, r, "f_v"] \ar[shift right, r, "g_v"'] & T
\end{tikzcd}
\end{equation}
First we construct $\wtilde T'$ be factorizing maps $f'l$ and $g'l$ via vertex collapse using lemma \ref{lemma_map_factor}. They can be factored via the same map $r$ since
$f'_1 = g'_1 = \id_{S_1}$ and so $r_1 = l_1$. Apply the lift construction of proposition \ref{prop_lift} to two pairs of maps $(f_w, r)$ and $(g_w, r)$. Since $r$ is a vertex
collapse and both maps $f_w$ and $g_w$ are identities at level one, the construction of lift is particularly simple. In fact one can easily see that the prenet $\wtilde S'$
completing the spans to commutative squares in both cases coincide with the prenet obtained by factorizing composition $hl$ via a vertex collapse $q'$.
The fibers of lifts $\tilde f_w$ and $\tilde g_w$ can be obtained as lifts of the corresponding fibers of $f_w$ and $g_w$ respectively along vertex collapse $r$, hence they
are also pairwise weakly equivalent.

Similarly, one constructs $\wtilde S$ both as a lift of $h$ along $q'$ and as a factorization of $l$ via a vertex collapse $q$. Thus we obtain factorization of the pair of maps $(fl, gl)$ into
$$
\begin{tikzcd}
P \ar[r, "\tilde h \tilde l"] & \wtilde S' \ar[shift left, r, "\tilde f_w"] \ar[shift right, r, "\tilde g_w"'] & \wtilde T' \ar[shift left, r, "f_v r"] \ar[shift right, r, "g_v r"'] & T,
\end{tikzcd}
$$
establishing equivalence $fl \sim gl$.

Let $m\from T \to Q$ be another map of prenets, and consider the following diagram.
\begin{equation}
\label{equ_piHom_push}
\begin{tikzcd}[sep=3em]
S \ar[r, "h"] & S' \ar[shift left, r, "f_w"] \ar[shift right, r, "g_w"'] & T' \ar[d, "s"'] \ar[shift left, r, "f_v"] \ar[shift right, r, "g_v"'] & T \ar[d, "\tilde m"] \\
& & \wtilde T' \ar[shift left, r, "\tilde f_v"] \ar[shift right, r, "\tilde g_v"'] & \wtilde T \ar[d, "\tilde m'"] \\
& & & Q
\end{tikzcd}
\end{equation}
Here we first construct $\wtilde T$ by factorizing map $m$ via a vertex collapse $\tilde m'$. Then we factorize compositions $\tilde m f_v$ and $\tilde m g_v$ via vertex
collapses $\tilde f_v$ and $\tilde g_v$ respectively. As before, it is not hard to see by looking at the preimages of parents of vertices in $\wtilde T_2 = Q_2$
that they can be factorized via the same prenet $\wtilde T'$.

Consider the fibers of the composition $s f_w$. On the second level we have
$$
(sf_w)^{-1}(x) = \bigsqcup_{y \in s^{-1}(x)} f_w^{-1}(y),
$$
while their parent are induced by the parents in $S'$. Thus by taking the coproduct of maps realizing weak equivalences between fibers $f_w^{-1}(y)$ and $g_w^{-1}(y)$ for all $y \in s^{-1}(x)$
we conclude that fibers $(sf_w)^{-1}(x)$ and $(sg_w)^{-1}(x)$ are weakly equivalent. In effect, we obtain decomposition
$$
\begin{tikzcd}
S \ar[r, "h"] & S' \ar[shift left, r, "sf_w"] \ar[shift right, r, "sg_w"'] & \wtilde T' \ar[shift left, r, "\tilde m'\tilde f_v"] \ar[shift right, r, "\tilde m' \tilde g_v"'] & Q,
\end{tikzcd}
$$
establishing equivalence $mf \sim mg$.

Putting it all together, for $l\from P \to S$ and $m\from T \to Q$ we obtain maps
\begin{equation}
\label{equ_piHom_cat}
l^*\from \pi\Hom(S, T) \to \pi\Hom(P, T), \quad \text{and} \quad m_*\from \pi\Hom(S, T) \to \pi\Hom(S, Q).
\end{equation}

\begin{lemma}
\label{lemma_loc_pi}
The localization map $L$ factors through the equivalence classes of maps.
$$
\begin{tikzcd}[row sep=3em]
\Hom_{2\PreNet}(S, T) \ar[dr] \ar[rr, "L"] && \Hom_{2\Net}(S, T) \\
& \pi\Hom(S, T) \ar[ur]
\end{tikzcd}
$$
\end{lemma}
\proof
Consider two maps $f$ and $g$ between prenets $S$ and $T$, such that $f \sim g$. Then by definition we have factorization
$$
\begin{tikzcd}
S \ar[r, "h"] & S' \ar[shift left, r, "f_w"] \ar[shift right, r, "g_w"'] & T' \ar[shift left, r, "f_v"] \ar[shift right, r, "g_v"'] & T.
\end{tikzcd}
$$
Due to lemma \ref{lemma_vc_tc_pair} the images of $f_v$ and $g_v$ in the localized category coincide. In order to show that the images
of $f_w$ and $g_w$ also coincide it is enough to construct a lift
$$
\begin{tikzcd}[sep=3em]
\wtilde S' \ar[shift left, d, "f_t"] \ar[shift right, d, "g_t"'] \ar[dr] & \\
S' \ar[shift left, r, "f_w"] \ar[shift right, r, "g_w"'] & T',
\end{tikzcd}
$$
such that both $f_t$ and $g_t$ are tassel collapses.

For every vertex $x \in T'_2$ we denote by $F_x$ and $G_x$ the fiber prenets $f_w^{-1}(x)$ and $g_w^{-1}(x)$ respectively. Since they are weakly equivalent
we have $m(F_x) = m(G_x)$, and using proposition \ref{prop_lift} we can form a lift $H_x$
$$
\begin{tikzcd}[sep=3em]
H_x \ar[r, "(g_t)_x"] \ar[d, "(f_t)_x"'] & G_x \ar[d] \\
F_x \ar[r] & m(F_x),
\end{tikzcd}
$$
such that all four maps are tassel collapses. Then we combine the prenets $H_x$ into a single prenet $\wtilde S'$, by putting
$$
\wtilde S'_2 = \bigsqcup_{x \in T'_2} (H_x)_2, \quad\quad f_t = \bigsqcup_{x \in T'_2} (f_t)_x, \quad\quad g_t = \bigsqcup_{x \in T'_2} (g_t)_x.
$$
The partial order $\pless = \pless_{\wtilde S'_2}$ is generated by the partial orders on $H_x$ and relations $y \pless y'$ if $y \in H_x$ and $y' \in H_{x'}$ share a parent and
$x < x'$. Moreover, if $y$ and $y'$ are incomparable with respect to this order then their images under $f_t$ and $g_t$ are also incomparable (since $(f_t)_x$ and $(g_t)_x$ are tassel collapses)
and we can put $y < y'$ if $p(y) < p(y')$. Due to remark \ref{rem_total_order} this determines the total order on $\wtilde S'_2$.

It is straightforward to check that maps $f_t$ and $g_t$ are indeed tassel collapses, and hence compositions $f_w f_t$ and $g_w g_t$ are also maps of prenets, which coincide by
construction.
\qed

\begin{remark}
\label{rem_equiv_tc_vc}
Using the assembly process from proposition \ref{prop_assembly} in the next section we can reformulate the notion of equivalence as follows. For two maps
$\begin{tikzcd}[cramped,sep=small]S \ar[shift left, r, "f"] \ar[shift right, r, "g"'] & T\end{tikzcd}$ we have $f \sim g$ if they admit factorization
$$
\begin{tikzcd}
S \ar[r, "h"] & P \ar[shift left, r, "f_t"] \ar[shift right, r, "g_t"'] & Q \ar[r, "l"] &  R \ar[shift left, r, "f_v"] \ar[shift right, r, "g_v"'] & T,
\end{tikzcd}
$$
where $f_t$ and $g_t$ are tassel collapses and $f_v$ and $g_v$ are vertex collapses. In general we can not combine maps $h$ and $l$ in this factorization,
however, we refer to the next section and in particular to remark \ref{rem_equivo_tc_vc} for further discussion on this matter.
\end{remark}

\begin{lemma}
\label{lemma_w_pair}
Let $f$ and $g$ be two weak equivalences of prenets $\begin{tikzcd}[cramped,sep=small]S \ar[shift left, r, "f"] \ar[shift right, r, "g"'] & T\end{tikzcd}$, then their classes
in $\pi\Hom(S, T)$ and their images in the category of $2$-nets coincide.
\end{lemma}
\proof
First we factorize both maps via a vertex collapses.
$$
\begin{tikzcd}
S \ar[shift left, r, "f_t"] \ar[shift right, r, "g_t"'] & T' \ar[shift left, r, "f_v"] \ar[shift right, r, "g_v"'] & T.
\end{tikzcd}
$$
To show that they factor through the same prenet $T'$ we need to check that for every vertex $x \in T_2$ the preimages $f^{-1}p_T(x)$ and $g^{-1}p_T(x)$ coincide.
Consider the induced pair of maps $f^\circ$ and $g^\circ$
$$
\begin{tikzcd}
S^\circ \ar[shift left, r, "f^\circ"] \ar[shift right, r, "g^\circ"'] & T^\circ \ar[r, "z"] & m(S^\circ),
\end{tikzcd}
$$
which are also weak equivalences, and let $z$ be the unique map to the terminal object $m(S^\circ)$ (corollary \ref{cor_term_prenet}). Since $z$ is a weak equivalence
we have $p_T(x) = z^{-1}p_{m(S^\circ)}(z(x))$ and therefore since compositions $zf^\circ = zg^\circ$ we conclude that
$$
f^{-1}p_T(x) = (zf^\circ)^{-1} p_{m(S^\circ)}(z(x)) = g^{-1}p_T(x).
$$

Now, since both maps $f_t$ and $g_t$ are weak equivalences, we see by lemma \ref{lemma_fib_w} that their fibers are pairwise weakly equivalent. Therefore $f \sim g$,
and their classes coincide in $\pi\Hom(S, T)$. Finally, using lemma \ref{lemma_loc_pi} we conclude that images of $f$ and $g$ in the localized category also coincide.
\qed

We observe that the existence of minimal representatives in each weak equivalence class (proposition \ref{prop_min_prenet}) alongside lemma \ref{lemma_w_pair}
immediately tells us that the spaces of isomorphisms in the category of $2$-nets contain at most one element.
\begin{equation}
\label{equ_net_isom}
\Iso_{2\Net}(S, T) = \begin{cases}
1,&\text{if $\cW(S) = \cW(T)$},\\
\emptyset,&\text{if $\cW(S) \neq \cW(T)$}.
\end{cases}
\end{equation}

The following lemma is a part of the proof of theorem \ref{thm_net_homs}, but we will state it separately in order to make the exposition less cumbersome.

\begin{lemma}
\label{lemma_for_thm_loc}
Assume $f\from S \to T$ and $g\from S \to T$ are two maps of prenets, such that $f_1 = g_1 = \id_{S_1}$, and $h\from T \to T'$ is a tassel collapse, such that $hf = hg$.
Then there exists a tassel collapse $t\from\wtilde S \to S$, such that $ft \sim gt$.
$$
\begin{tikzcd}
\wtilde S \ar[r, "t"] & S \ar[shift left, r, "f"] \ar[shift right, r, "g"'] & T \ar[r, "h"] & T'
\end{tikzcd}
$$
\end{lemma}
\proof
We begin by introducing some notation. Let $x$ and $y$ be two vertices in $T_2$, then
$$
S_{xy} = f^{-1}(x) \cap g^{-1}(y), \quad\quad A_{xy} = A(S_{xy}), \quad\quad B_{xy} = B(S_{xy}),
$$
$$
P_{xy} = p_S(S_{xy}), \quad\quad P_x = p_T(x), \quad\quad P_y = p_T(y).
$$
Since all constructions in the proof can be performed independently for each vertex of $T'_2$, we may as well assume that $T'_2$ consists of a single point.
Thus we assume that $P_x = P_y = T_1 = S_1$.

{\bf A.} Let us show that subsets $A_{xy}$ and $B_{xy}$ are disjoint covers of $P_{xy}$. Due to the symmetry of the situation it is enough to consider only the case of $A_{xy}$.
Fix $x \in T_2$ and consider subset $Y_x = \{y \in T_2 \mid S_{xy} \neq \emptyset \}$. Let $y_0$ be the minimal element of $Y_x$, then since maps $f$ and $g$ preserve partial orders,
$A_{xy_0}$ is a subset of $A(f^{-1}(x))$. Property \ref{def_prenet_map} (c) guarantees that the latter is a disjoint cover of $P_x$, and therefore $A_{xy_0}$ has to be a disjoint
cover of $P_{xy_0}$.

Assume that $y \in Y_x$ is not minimal. Then the case when $x$ is the minimal element of subset $X_y = \{x \in T_2 \mid S_{xy} \neq \emptyset\}$ is handled similarly. So we may
assume that $x$ is not the minimal element of $X_y$.

Consider a maximal subset $A' \subset A_{xy}$, such that $A'$ is disjoint. By definition $A_{xy}$ is a cover of $P_{xy}$, and if it is not disjoint, then $A'$ is a proper subset
and thus there exists a parent $z \in P_{xy}$ which is not covered by $A'$. Now, consider the subset
$$
p_S^{-1}(z) \cap \bigcup_{y' < y} S_{xy'}.
$$
We claim that it is non-empty, and thus we can select a vertex $v$ in it. Indeed, if it is empty, then the element $a_z \in p_S^{-1}(z) \cap A_{xy}$ would also be an element of $A(f^{-1}(x))$,
but that would contradict the fact that $A(f^{-1}(x))$ is a disjoint cover. Similarly, we can select a vertex
$$
w \in p_S^{-1}(z) \cap \bigcup_{x' < x} S_{x'y}.
$$
However this leads to a contradiction, as both $v$ and $w$ share a common parent, namely $z$, but we also have
$$
f(w) = x' < x = f(v) \quad\text{and}\quad g(v) = y' < y = g(w).
$$
Therefore, we conclude that $A_{xy}$ was a disjoint cover from the beginning.

In fact the same argument can be used to show that for any two intervals $I, J \subset T_2$ the corresponding subsets $A_{IJ}$ and $B_{IJ}$ are disjoint covers of $P_{IJ}$.

{\bf B.} Assume that for some pair of vertices $x, y \in T_2$ the subsets $A_{xy}$ and $B_{xy}$ are not disjoint (from each other, and not in the sense that they do not form disjoint covers).
Take a common vertex $z$ and duplicate it, in the sense that we add a new tassel $z'$ to $S_2$ adjacent to $z$ in the total order and put $p(z') = p(z)$. Repeating this process we construct
a new prenet $\wtilde S$ together with a tassel collapse $t\from \wtilde S \to S$ and such that for $\wtilde S$ the corresponding subsets $A_{xy}$ and $B_{xy}$ are disjoint for all $x, y \in T_2$.

{\bf C.}
For convenience of notation we rename the prenet $\wtilde S$ constructed in step (B) into $S$ and replace $f$ and $g$ with the respective compositions with $t$.
We need to construct a factorization 
$$
\begin{tikzcd}
S \ar[r, "l"] & S' \ar[shift left, r, "f_w"] \ar[shift right, r, "g_w"'] & T,
\end{tikzcd}
$$
such that $f_w$ and $g_w$ have weakly equivalent fibers. If among the sets $S_{uv}$ the only non-empty sets are the diagonal, i.e. when $u = v$, then $f = g$ and there is nothing left to prove.

Assume first that the collection $S_{uv}$ is diagonal except for two vertices $x$ and $y$ and to be specific assume $x < y$. In other words the only potentially non-empty subsets
$S_{uv}$ not on the diagonal are $S_{xy}$ and $S_{yx}$.
We allow the possibility that some of these subsets, including the diagonal ones, are in fact empty.

To construct $S'$ we put $S'_{uv}$ for $u \neq v$ to be the set of maximal intervals in $P_{uv}$, and declare all elements of $S'_{uv}$ to be tassels. We also define $S'_{uu} = S_{uu}$ if
$u$ is neither $x$ nor $y$.

Using the analysis performed in step (A) we see that
the partition of the set of parents $S_1 = T_1$ induced by the disjoint cover $B_{xx}$ is compatible with the partition into $P_{xy}$ and $P_{yx}$, therefore we define
$S'_{xx} = (S_{xx} - B_{xx}) \sqcup S'_{xy} \sqcup S'_{yx}$ and the map $S_{xx} \to S'_{xx}$ is identity on $S_{xx} - B_{xx}$ and sends a vertex of $B_{xx}$ to the maximal interval containing
its parents. Similarly, using subset $A_{yy}$ in place of $B_{xx}$ we construct $S_{yy} \to S'_{yy}$ and put
$$
S'_2 = S'_{xy} \sqcup S'_{yx} \sqcup \bigsqcup_{u \in T_2} S'_{uu}.
$$
Finally, we put $S'_1 = S_1$, and the parent map $p_{S'}$ as well as map $l$ is clear from the construction above. This is a variation on the assembly construction of
proposition \ref{prop_assembly} and for the details we refer the reader to its proof.
This way we obtain factorizations $f = f_w l$ and $g = g_w l$. The fibers $f_w^{-1}(x)$ and $g_w^{-1}(x)$ are weakly equivalent, as they are
connected by tassel collapses of $S'_{xy}$ and $S'_{yx}$ to $S'_{xx}$, and similarly for $f_w^{-1}(y)$ and $g_w^{-1}(y)$.

{\bf D.} The general case is treated inductively. Denote by $X = X(f, g)$ the set of vertices $x \in T_2$, such that there exists a non-empty set among $S_{xy}$ or $S_{yx}$ with $y \neq x$.
In the previous steps we have established the case of $|X| = 2$. Assume now that $|X| > 2$.
We first replace maps $f$ and $g$ with modified maps $f'$ and $g'$. Denote temporarily by $x_1$ and $x_2$ the first two vertices of $X$, and define $f'$ by
$$
f'(z) = \begin{cases}
f(z),&\text{if $f(z) > x_1$},\\
x_1,&\text{if $f(z) = g(z) = x_1$},\\
x_2,&\text{if $f(z) = x_1$ and $g(z) > x_1$}.
\end{cases}
$$
Analysis in step (A) shows that $f'$ is indeed a map of prenets. Considering the pair of maps $(f, f')$ we see that we can apply the argument of step (C) since it has only one non-empty
non-diagonal subset, namely $f^{-1}(x_1) \cap f'^{-1}(x_2)$. Therefore, we conclude that $f \sim f'$. Defining $g'$ in a similar fashion we also conclude that $g \sim g'$.

Observe that for the pair $(f', g')$ the block of the vertices $X' = X(f', g') = X - \{x_1\}$ is smaller than $X$, thus by induction we have $f' \sim g'$.
\qed

Now we are ready to give a description of maps in the category of $2$-nets.

\begin{theorem}
\label{thm_net_homs}
For any two prenets $S, T$ we have
$$
\Hom_{2\Net}(S, T) \ \isom\ \colim_{\wtilde S \in \cW(S)} \pi\Hom(\wtilde S, m(T)).
$$
\end{theorem}
\proof
To establish the isomorphism we construct two mutually inverse maps $\alpha$ and $\beta$.

\begin{equation}
\label{equ_alpha_beta}
\begin{tikzcd}[column sep=4em]
\Hom_{2\Net}(S, T) \ar[start anchor=east, end anchor=west, bend left, shift left=0.5em, r, "\beta"]& \colim\limits_{\wtilde S \in \cW(S)} \pi\Hom(\wtilde S, m(T)). \ar[start anchor=west, end anchor=east, bend left, shift left=0.5em, l, "\alpha"]
\end{tikzcd}
\end{equation}

{\bf Construction of $\alpha$.} Consider a map $f\from \wtilde S \to m(T)$, and define $\alpha(f)$ by the following zig-zag of maps
$$
\alpha(f) = \left(\begin{tikzcd}[cramped,sep=small]
S \ar[r] & m(S) & \wtilde S \ar[l] \ar[r, "f"] & m(T) & T \ar[l]
\end{tikzcd}\right),
$$
where the maps to $m(S)$ as well as the last map to $m(T)$ can be chosen to be any weak equivalences to the respective minimal object.
In virtue of lemma \ref{lemma_w_pair} the resulting map in the category $2\Net$ doesn't depend on these choices.

We need to check that this is well defined. Assume we have a weak equivalence $c\from\wtilde S' \to \wtilde S$ and consider the commutative diagram.
$$
\begin{tikzcd}
& \wtilde S' \ar[d, "c"] \ar[dl, "a"'] \ar[dr, "f'"] & \\
m(S) & \wtilde S \ar[l, "b"] \ar[r, "f"']  & \wtilde T
\end{tikzcd}
$$
Then in the category of nets we have $f' a^{-1} = f c c^{-1} b^{-1} = f b^{-1}$, hence $\alpha(f) = \alpha(f')$. The fact that it doesn't depend on the choice of a representative $f$
of the equivalence class $\sim$ follows immediately from lemma \ref{lemma_loc_pi}.

{\bf Construction of $\beta$.} A map from $S$ to $T$ in the category of nets is a chain of objects $\{S = P_0, P_1, \ldots, P_n, P_{n+1} = T\}$ connected by a zig-zag of maps
$$
\begin{tikzcd}
\phi\from S & P_1 \ar[dashed, l] \ar[r] & P_2 & \cdots \ar[dashed, l] \ar[r] & P_{n-1} & P_n \ar[dashed, l] \ar[r] & T
\end{tikzcd}
$$
where dashed lines are weak equivalences. It is enough to consider the case $n = 3$, as the general situation is reduced to this by induction on the length of the zig-zag.
By proposition \ref{prop_lift} we can construct a lift $\wtilde P$ fitting into a commutative diagram
$$
\begin{tikzcd}
&& \wtilde P \ar[dashed, dl] \ar[dr] && \\
& P_1 \ar[dashed, dl] \ar[dr] && P_3 \ar[dashed, dl] \ar[dr] & \\
S && P_2 && T
\end{tikzcd}
$$
and we put $\beta(\phi)$ to be represented by an element $\wtilde P \to P_3 \to T \to m(T)$ in the colimit on the right hand side of \ref{equ_alpha_beta}.
We need to show that $\beta$ is well defined. First we observe that it doesn't depend on the choice of the last map $T \to m(T)$, because according to
lemma \ref{lemma_w_pair} these choices produce the same element in $\pi\Hom(T, m(T))$ and therefore compose into the same element in $\pi\Hom(\wtilde P, m(T))$.

{\bf Independence of the lift.} Next, we check that $\beta(\phi)$ doesn't depend on the choice of the lift $\wtilde P$. If we have another object $\wtilde P'$ that
fits into the same diagram, then using corollary \ref{cor_w_lift} we can form a lift over $P_1$ producing a new prenet $Q$ with a map to $P_1$ and two maps to $P_3$.

Consider the following diagram.
$$
\begin{tikzcd}[sep=3em]
Q \ar[shift left, d, "f'"] \ar[shift right, d, "g'"'] \ar[shift left, r, "f"] \ar[shift right, r, "g"'] & P_3 \ar[d, "h_v"] \ar[r] & T \ar[r] & m(T) \ar[equal, d] \\
R'\ar[d, "h'_t"'] \ar[shift left, r, "f_v"] \ar[shift right, r, "g_v"'] & R \ar[d, "h_t"] \ar[rr, "l"'] && m(T) \\
P'_2 \ar[r, "v"'] & P_2 &&
\end{tikzcd}
$$
We construct $R$ as the factorization of $h\from P_3 \to P_2$ into a vertex collapse $h_v$ and a tassel collapse $h_t$ (lemma \ref{lemma_w_factor_vt}). Using pushout
of the map $P_3 \to m(T)$ along the vertex collapse $h_v$ as in lemma \ref{lemma_vc_pushout} we obtain the map $l\from R \to m(T)$, since $m(T)$ is a minimal prenet.
We need to show that the classes of $lh_v f$ and $lh_v g$ coincide in the colimit in \ref{equ_alpha_beta}, in fact it is enough to check that
the classes of $h_v f$ and $h_v g$ coincide.

Let us factor $h_v f$ and $h_v g$ via vertex collapse maps $f_v$ and $g_v$. To see that these two factorizations pass through the same prenet $R'$ consider a vertex
$x \in R'_2 = R_2$, then its parents in $R'$ should be on one hand $f^{-1} h_v^{-1} p_R(x)$ and on the other hand $g^{-1} h_v^{-1} p_R(x)$. However, since
$h_t$ is a tassel collapse we have $p_R(x) = p_{P_2}(h_t(x))$ and since $h_t h_v f = h_t h_v g$ the two sets of parents coincide, thus giving rise to a well defined
prenet $R'$. Moreover, we have $f_v = g_v$, since by construction they are identity at level $2$ and at level $1$ we have $(h_t)_1 (f_v)_1 = (h_t)_1 (g_v)_1$ where
$(h_t)_1$ is a bijection.

Now we construct $P'_2$ by factorizing the composition $h_t f_v = h_t g_v$ via a vertex collapse $v$ and a tassel collapse $h'_t$. We see that by construction
the first column in the diagram above satisfies conditions of lemma \ref{lemma_for_thm_loc}, thus we conclude that classes of $f'$ and $g'$ coincide in the colimit
in \ref{equ_alpha_beta} which demonstrates independence of the construction of $\beta$ of the choice of a lift.

{\bf Compatibility with relations.} Let us show that $\beta$ is compatible with relations defining the localization of the category of prenets. First, consider
a factorization $f = gh$ in prenets and a zig-zag $\phi$ representing a map in the category of nets.
$$
\begin{tikzcd}
\phi\from S \ar[r, "h"] & P & P \ar[l, "="'] \ar[r, "g"] & T.
\end{tikzcd}
$$
By definition of $\beta$ the class $\beta(\phi)$ involves the lift of $h$ along the identity, which can be chosen to be $h$ itself. Hence we see that $\beta(\phi)$ is
represented by the composition $gh = f$ as required.
A similar argument also applies when we have a factorization of a weak equivalence $w = uv$ into two weak equivalences and zig-zag
$$
\begin{tikzcd}
\psi\from S & P \ar[l, "v"'] \ar[r, "="] & P & T \ar[l, "u"'].
\end{tikzcd}
$$

Now let $w$ be a weak equivalence and consider a zig-zag $\begin{tikzcd}[cramped,sep=small] \phi\from S \ar[r, "w"] & T & S \ar[l, "w"'] \end{tikzcd}$.
The class $\beta(\phi)$ can be obtained using any lift of $w$ along $w$, for example we can take it to be $S$ itself.
$$
\begin{tikzcd}
S \ar[d, "="'] \ar[r, "="] & S \ar[d, "w"] \\
S \ar[r, "w"'] & T
\end{tikzcd}
$$
Therefore, $\beta(\phi) = \beta(\id_S)$ as required. Similarly, for a zig-zag $\begin{tikzcd}[cramped,sep=small] \psi\from T & \ar[l, "w"'] S \ar[r, "w"] & T \end{tikzcd}$
we see that this cospan represents a lift of identity on $T$.
$$
\begin{tikzcd}
S \ar[d, "w"'] \ar[r, "w"] & T \ar[d, "="] \\
T \ar[r, "="'] & T.
\end{tikzcd}
$$
Since $\beta$ doesn't depend on the choice of such a lift we conclude that $\beta(\psi) = \beta(\id_T)$ as required.

{\bf $\alpha$ and $\beta$ are mutually inverse.} By definitions of $\alpha$ and $\beta$ the class $\beta(\alpha(f))$ is represented by the composition in the top row of the following diagram.
$$
\begin{tikzcd}
\wtilde S_3 \ar[d, "l_1"'] \ar[r] & \wtilde S_2 \ar[d] \ar[r] & T \ar[d, "z"] \ar[r, "z"] & m(T) \\
\wtilde S_1 \ar[d] \ar[r, "l_2"'] & \wtilde S \ar[d] \ar[r, "f"'] & m(T) & \\
S \ar[r] & m(S) & &
\end{tikzcd}
$$
Since in constructions of both $\alpha$ and $\beta$ we could take arbitrary weak equivalence map $z\from T \to m(T)$ we might as well take the same map $z$ in both cases. Therefore, the class $\beta(\alpha(f))$ is also represented by
the composition $f l_2 l_1$, and as both $l_1$ and $l_2$ are weak equivalences it coincides with the class of $f$ in the colimit on right hand side of \ref{equ_alpha_beta}.

Now, let $\begin{tikzcd}[cramped,sep=small] \phi\from S \ar[squiggly, r] & T \end{tikzcd}$ be a zig-zag corresponding to a map in the category of nets. Then the composition
$\alpha(\beta(\phi))$ is represented by the zig-zag $\begin{tikzcd}[cramped,sep=small] S \ar[r] & m(S) & \wtilde S \ar[l] \ar[r, "tf"] & m(T) & T \ar[l]\end{tikzcd}$ which is
a part of the following diagram.
$$
\begin{tikzcd}
& \wtilde S \ar[dl, "w"'] \ar[dd, near end, "\tilde s"] \ar[dr, "f"] & & & \\
S \ar[dr, "s"'] \ar[squiggly, rr, near end, "\phi"'] & & T \ar[dr, "t"] & & T \ar[dl, "t"'] \\
& m(S) & & m(T) & 
\end{tikzcd}
$$
Here $\wtilde S$ is obtained from $\phi$ by the repeated application of the lift construction, map $s$ is a weak equivalence to the minimal object $m(S)$, $\tilde s = sw$ and
$t$ is a weak equivalence to the minimal object $m(T)$. Using identities for maps in the localized category we find that $\phi = f w^{-1}$, $w^{-1} = \tilde s^{-1} s$, and therefore
$\alpha(\beta(\phi)) = \phi$.
\qed

\begin{corollary}
\label{cor_net_homs}
In the colimit on the right hand side of the statement of theorem \ref{thm_net_homs} it is enough to restrict to the subcategory of $1$-minimal prenets in $\cW(S)$.
$$
\Hom_{2\Net}(S, T) \ \isom\ \colim_{\wtilde S \in \cW(S)_1} \pi\Hom(\wtilde S, m(T)).
$$
\end{corollary}
\proof
Let $\gamma$ be the natural map of colimits induced by the inclusion $\cW(S)_1 \into \cW(S)$.
$$
\begin{tikzcd}
\colim\limits_{\wtilde S \in \cW(S)_1} \pi\Hom(\wtilde S, m(T)) \ar[r, "\gamma"] & \colim\limits_{\wtilde S \in \cW(S)} \pi\Hom(\wtilde S, m(T)).
\end{tikzcd}
$$
In order to construct its inverse map $\delta$, let us take $f \from \wtilde S \to m(T)$ representing an element on the right hand side. Then using pushout lemma \ref{lemma_vc_pushout}
and minimality of $m(T)$ we obtain
$$
\begin{tikzcd}
\wtilde S \ar[d] \ar[r, "f"] & m(T) \ar[equals, d] \\
m_1(\wtilde S) \ar[r, "f'"'] & m(T),
\end{tikzcd}
$$
and define $\delta(f) = f'$. It is immediate to see that the two maps $\gamma$ and $\delta$ are mutually inverse.

\qed

\begin{proposition}
\label{prop_loc_vc}
The localization of $2\PreNet$ with respect to class of vertex collapses has the following sets of maps.
$$
\Hom_{2\PreNet[\mathcal{VC}^{-1}]}(S, T) \ \isom\ \Hom_{2\PreNet}(m_1(S), m_1(T)).
$$
\end{proposition}
\proof
The demonstration is essentially reiteration of the proof of theorem \ref{thm_net_homs} simplified by the following two facts. First, the minimal
prenet $m_1(S)$ is the terminal object of the component $\mathcal{VC}(S)$ (corollary \ref{cor_vc_term_prenet}), and second, we can always form pushout along vertex collapses (lemma \ref{lemma_vc_pushout}).
\qed

\begin{nparagraph}
\label{par_J}
The rest of this section will be dedicated to construction of a fully faithful functor from the category of $2$-trees and maps between them which are surjective at level $1$ to the
category of $2$-nets. The reason why we have to restrict to the subcategory $(2\Tree, \Epi_1)$ is because in our definition of prenets here we did not allow tassels at level $1$.
We believe that a proper definition of nets with tassel at all levels, which would be necessary to define $n$-nets for an arbitrary $n$ regardless, would avoid this restriction.

We begin the construction of functor $J\from (2\Tree, \Epi_1) \to 2\Net$ by defining it on objects. Let $T$ be a 2-tree, then we put $J(T)_0 = T_0$, consisting of a single vertex,
$J(T)_1 = T_1$, and on the second level we put
$$
c(J(T)_2) = T_2, \quad\quad t(J(T)_2) = T_1 - p(T_2), \quad\quad J(T)_2 = c(J(T)_2) \sqcup t(J(T)_2).
$$
The parent map is defined as
$$
p_{J(T)}(x) = \begin{cases}
p_T(x),&\text{if $x \in c(J(T)_2)$},\\
x,&\text{if $x \in t(J(T)_2)$}.
\end{cases}
$$
Since the sets of children of vertices in $J(T)_1$ are disjoint their total orders combine into a well-defined partial order $\pless$ on $J(T)_2$, and as was pointed out in \ref{rem_total_order}
it can be extended to a total order on $J(T)_2$.

We would like to point out here that the constructed prenet $J(T)$ is minimal. Indeed, first observe that $J(T)$ is a tree itself, in the sense that intervals $p(x)$ consist of a single vertex of $J(T)_1$
for all vertices  $x \in J(T)_2$. This implies that $J(T)$ is $1$-minimal. Moreover, for every vertex $y \in J(T)_1$ the set of children $p^{-1}(y)$ consists either entirely of core vertices, or contains
a single tassel and no other vertices. In either case no non-trivial tassel collapse is possible and we conclude that $J(T)$ is also $2$-minimal. Therefore $J(T)$ is a minimal prenet.
\end{nparagraph}

\begin{nparagraph}
\label{par_J_map}
Let $f\from S \to T$ be a map of $2$-trees, surjective at level $1$, we will construct a prenet $\wtilde S = \wtilde S(f)$ weakly equivalent to $J(S)$ and a map of prenets
$J(f)\from \wtilde S \to J(T)$. We put $\wtilde S_0 = S_0$, $\wtilde S_1 = S_1$ and $J(f)_1 = f_1$. For every vertex $x \in S_1$ we define
$$
(\wtilde S_2)_x = \begin{cases}
p_S^{-1}(x) \sqcup (p_T^{-1}(f(x)) - f(p_S^{-1}(x))),&\text{if $x \in p(S_2) \subset S_1$}, \\
p_T^{-1}(f(x)),&\text{if $x \in S_1 - p(S_2)$},
\end{cases}
$$
and put
$$
\wtilde S_2 = \bigsqcup_{x \in S_1} (\wtilde S_2)_x.
$$
On the component $(\wtilde S_2)_x$ we define the parent map $p_{\wtilde S}$ to be $x \in S_1$, and define map $J(f)_2$ by
$$
J(f)_2(y) = \begin{cases}
f_2(y),&\text{if $y \in p_S^{-1}(x)$},\\
y,&\text{otherwise}.
\end{cases}
$$
The total order on subsets $(\wtilde S_2)_x$ is determined on each component of the disjoint union by the total order on $S_2$ and $T_2$ respectively,
and if $y \in p_S^{-1}(x)$ and $z \in p_T^{-1}(f(x)) - f(p_S^{-1}(x))$ then we declare $y < z$ if $f(y) <_T z$, and $z < y$ if $z <_T f(y)$. The total order on $\wtilde S_2$ is
induced by the order on each of the subsets $(\wtilde S_2)_x$ and by putting $(\wtilde S_2)_x < (\wtilde S_2)_{x'}$ whenever $x < x'$ in $S_1$. Finally, the core vertices

$$
c(\wtilde S_2) = \bigsqcup_{x \in S_1} c(p_S^{-1}(x)).
$$

It is clear by construction that $\wtilde S$ is a prenet (in fact a tree) weakly equivalent to $J(S)$. The weak equivalence $u\from \wtilde S \to J(S)$ is not unique, in fact we
can take any map which for each $(\wtilde S_2)_x$ with $x \in p(S_2) \subset S_1$ is identity on $p_S^{-1}(x)$ and sends the second component $p_T^{-1}(f(x)) - f(p_S^{-1}(x))$ to $p_S^{-1}(x)$ while preserving the total order on
$(\wtilde S_2)_x$, and for $x \in S_1 - p(S_2)$ map $u$ sends entire set of tassels $p_T^{-1}(f(x))$ to the unique tassel in $J(S)$ with parent $x$.
We will show that $J(f)$ is a map of prenets. Indeed, by definition of subsets
$(\wtilde S_2)_x$ we immediately see that condition (b) of definition \ref{def_prenet_map} holds. Also, by definition of the orders
on each $(\wtilde S_2)_x$ we see that $J(f)$ preserves the partial orders $\pless_{\wtilde S}$ and $\pless_T$, thus satisfying condition (a). Finally, since $\wtilde S$ is a tree
the condition (c) is automatically satisfied.

For an illustration of this construction we refer to example \ref{exa_J}.
\end{nparagraph}

\begin{theorem}
The above construction of $J$ provides a fully faithful functor
$$
J\from (2\Tree, \Epi_1) \into 2\Net.
$$
\end{theorem}
\proof
The demonstration of the statement will be performed in three steps.

{\bf $J$ is a functor.} Consider a pair of maps of trees
$$
\begin{tikzcd}
S \ar[r, "f"] & T \ar[r, "g"] & P.
\end{tikzcd}
$$
The composition of $J(f)$ and $J(g)$ in the category of $2$-nets can be represented by the composition in the middle row of the diagram
$$
\begin{tikzcd}[sep=3em]
\wtilde S(gf) \ar[rrd, "J(gf)"] & & \\
\wtilde S' \ar[d] \ar[u, "w"] \ar[r, "J(f)'"'] & \wtilde T(g) \ar[d, "u"] \ar[r, "J(g)"'] & J(P), \\
\wtilde S(f) \ar[r, "J(f)"'] & J(T) &
\end{tikzcd}
$$
where $\wtilde S'$ is obtained by using proposition \ref{prop_lift}. In order to show that $J$ is a functor we will construct a weak equivalence $w\from \wtilde S' \to \wtilde S = \wtilde S(gf)$,
such that $J(g) J(f)' = J(gf) w$. Since at level $1$ we have $\wtilde S'_1 = S_1 = \wtilde S(gf)_1$ we put $w_1 = \id_{S_1}$. Moreover, as we are working with trees, we can consider children of each vertex $x \in S_1$ independently. To be specific we assume that vertices $x \in S_1$ and $f(x) \in T_1$
have children, as the other cases are treated similarly. By definition of $J$ we have
\begin{align*}
(\wtilde S(f)_2)_{x} &= p_S^{-1}(x) \sqcup (p_T^{-1} f(x) - f p_S^{-1}(x)), \\
(\wtilde T(g)_2)_{f(x)} &= p_T^{-1}f(x) \sqcup (p_P^{-1} gf(x) - g p_T^{-1}f(x)), \\
(\wtilde S(gf)_2)_{x} &= p_S^{-1}(x) \sqcup (p_P^{-1} gf(x) - gfp_S^{-1}(x)).
\end{align*}
Using construction of proposition \ref{prop_lift} we find that regardless of the choice of weak equivalence $u$, the set of children of $x$ in $\wtilde S'$ is
$$
(\wtilde S'_2)_x = p_S^{-1}(x) \sqcup (p_T^{-1} f(x) - f p_S^{-1}(x)) \sqcup (p_P^{-1} gf(x) - g p_T^{-1}f(x)).
$$
We define a map $(w_2)_x \from (\wtilde S'_2)_x \to (\wtilde S(gf)_2)_x$ by putting it to be identity on the first and third components of $(\wtilde S'_2)_x$ and induced by $g$ on the second component.
Clearly, for each $x$ this is a tassel collapse, hence the map $w_2$ obtained as the coproduct of all $(w_2)_x$ gives rise to a weak equivalence $w$. The equality $J(g) J(f)' = J(gf) w$ is clear from the
definition of $w$.

{\bf $J$ is faithful.} Let $S$ and $T$ be two $2$-trees, then in order to show the faithfulness of $J$ we will construct the inverse map
\begin{equation}
\label{equ_tree_J}
\begin{tikzcd}[sep=3em]
\colim\limits_{\wtilde S \in \cW(J(S))_1} \pi\Hom(\wtilde S, J(T)) \ar[r, "J^{-1}"] & \Hom_{2\Tree}(S, T).
\end{tikzcd}
\end{equation}
Notice, that by construction the core vertices of $J(S)$ are in bijection with the vertices of the tree $S$. Moreover, since every map in the category $\cW(J(S))_1$ is
a tassel collapse, it induces bijection on the set of core vertices $c(\wtilde S) \isom c(J(S)) \isom S$. Since a map of prenets sends core vertices to core vertices we can define
for a map $f\from \wtilde S \to J(T)$
$$
J^{-1}(f) := f|_{c(\wtilde S)} \from S \to T.
$$
From the discussion above it is clear that $J^{-1}(f)$ doesn't depend on the representative of the colimit. Moreover, if we have two maps $f, g\from \wtilde S \to J(T)$, such that $f \sim g$,
then again by definition of the equivalence relation restrictions of $f$ and $g$ to the core vertices coincide. Hence $J^{-1}(f)$ is well defined and by definition of $J$ we see that
for any map of $2$-trees $g\from S \to T$ we have $J^{-1}J(g) = g$.

{\bf $J$ is full.} We need to show that for any map $f\from \wtilde S \to J(T)$ the map $JJ^{-1}(f)$ coincides with $f$ in the colimit in (\ref{equ_tree_J}). In order to do so we will
construct a tassel collapse $w\from \wtilde S \to \wtilde S(J^{-1}(f))$, such that the following diagram commutes.
$$
\begin{tikzcd}[sep=3em]
& \wtilde S \ar[dl] \ar[d, "w"] \ar[dr, "f"] & \\
J(S) & \wtilde S(J^{-1}(f)) \ar[l] \ar[r, "JJ^{-1}(f)"'] & J(T).
\end{tikzcd}
$$
We put $w_1 = \id_{S_1}$.
Since $w$ is a map between trees it is enough to consider children of each vertex $x \in S_1$ separately. To be specific, let us take vertex $x \in p(S_2) \subset S_1$, as the other case
is treated similarly. To simplify notation we put $g = J^{-1}(f)$, then we have
$$
(\wtilde S(g)_2)_x = p_S^{-1}(x) \sqcup \left(p_T^{-1}g(x) - g p_S^{-1}(x) \right).
$$
To define $(w_2)_x$ we observe that the core vertices $c(\wtilde S_2)_x := c(\wtilde S_2) \cap p_{\wtilde S}^{-1}(x) \isom p_S^{-1}(x)$ and so we can send them identically to the first component of $(\wtilde S_2)_x$.
And for a tassel $y \in t(\wtilde S) \cap p_{\wtilde S}^{-1}(x)$ we send it to the second component of $(\wtilde S_2)_x$ by putting $w_2(y) = f(y)$. Since $f_2$ preserved the order on the children of $x$
then so does $(w_2)_x$. Finally, we define $w_2$ as the coproduct of $(w_2)_x$ over all vertices $x \in S_1$.

It is clear from the construction that $w$ is a tassel collapse and that the diagram commutes.

\qed

\vskip 5em

\section{Operadic category of locally constant $2$-nets}
\label{sec_lc_nets}

Our goal here is to construct an operadic category that extends the category of $2$-trees. Unfortunately the category of $2$-nets defined in section \ref{sec_nets} is not operadic,
as it doesn't possess a well defined fiber functor. Let us clarify what the issue here is. Consider a map of $2$-nets $f\from S \to T$, and let $\tilde f\from \wtilde S \to T$ be
its representative map between $2$-prenets. It seems natural to attempt to define the fiber functor $\Fib_x \from 2\Net / T \to 2\Net$ for a vertex $x \in T_2$ by taking
the fiber prenet $\tilde f^{-1}(x)$ as in definition \ref{def_fiber}. As we will see below this construction sends weak equivalences to weak equivalences, so it induces
a well defined functor between localizations.

However, this functor also depends on the representative $T$ of its equivalence class. One might attempt to resolve this ambiguity by replacing $T$ with the minimal prenet $m(T)$,
but since the minimal prenet is not the terminal object of $\cW(T)$ we don't have a canonical map $T \to m(T)$, and for different choices of such a map the fiber prenets 
for the composition $S \to T \to m(T)$ are generally speaking not weakly equivalent. For an illustration we refer the reader to example \ref{exa_net_not_operadic}.

The notion of a locally constant map defined in this section was designed to address this ambiguity, thus giving rise to well defined fiber functors.

\begin{nparagraph}
First we restrict to the category $2\PreNet / I$ of $2$-prenets with a map into $I$. Clearly, if such a map exists then it is necessarily unique, however there are
prenets that do not have a map to $I$, which are excluded here from consideration. It is important to point out that this restriction is well behaved with respect to localization
procedure of section \ref{sec_nets}, since the subcategory $2\PreNet/I$ is closed under weak equivalences. Indeed, if $f\from S \to T$ is a weak equivalence, and $T$ is a prenet over $I$
then obviously, $S$ is also prenet over $I$. So let us assume instead that $S$ is a prenet over $I$, and check that $T$ is also prenet over $I$. The only thing that needs to be checked here
is that the only level-wise map $g\from T \to I$ satisfies property (c) of \ref{def_prenet_map}. As we saw in the proof of lemma \ref{lemma_map_composition}
$$
A((gf)^{-1}(x)) = A(f^{-1} A(g^{-1}(x))),
$$
where $x$ is the only point in $I_2$. Since $f$ is a weak equivalence for any vertex $y \in A(g^{-1}(x))$ the subset $A(f^{-1}(y))$ consists of a single vertex, which covers the entire
$f^{-1}(p_T(y))$. As, by assumption prenet $S$ is over $I$, the subset $A((gf)^{-1}(x))$ forms a disjoint cover of $S_1$, and therefore, $A(g^{-1}(x))$ is also a disjoint cover of $T_1$.
Using the same argument for $B$ we conclude that $g\from T \to I$ is a map of prenets.
\end{nparagraph}

\begin{lemma}
\label{lemma_fib_map}
Let $f$ and $g$ be two composable maps $\begin{tikzcd}[cramped,sep=small] S \ar[r, "f"] & T \ar[r, "g"] & P\end{tikzcd}$.
Then for any vertex $x \in P_2$, map $f$ induces a map between the fiber prenets
$$
f_x\from (gf)^{-1}(x) \to g^{-1}(x).
$$
Moreover, if $f$ is a weak equivalence, then $f_x$ is also a weak equivalence.
\end{lemma}
\proof
Since map $f$ preserved the partial orders $\pless$ on $S$ and $T$, it is clear that its restriction to a subset of vertices will also preserve the partial order, hence
condition (a) of \ref{def_prenet_map} is satisfied. Condition (b) is also clearly satisfied, since it is satisfied by map $f$ and the fiber prenet $(gf)^{-1}(x)$ contains
all vertices in the preimage $f^{-1}(y)$ for any $y \in g^{-1}(x)$ as well as all their parents. Finally, condition (c) is also fulfilled, since it holds for $f$
and as before the preimage $f^{-1}(y)$ is contained entirely in the fiber and has the same parents as in $S$.

Moreover, if $f$ is a weak equivalence, then it clearly restricts to a bijection between core vertices in the fibers and by the same argument as before condition (b) of \ref{def_weak_equiv}
is satisfied.
\qed

\begin{proposition}[Assembly]
\label{prop_assembly}
Let $f\from S \to T$ be a map of prenets, such that $S_1 = T_1$ and $f_1 = \id_{S_1}$. For every vertex $x \in T_2$ denote by $S_x$ the fiber prenet $f^{-1}(x)$, and consider
a collection of maps $l_x\from S_x \to P_x$, such that $P_x$ is a prenet over $I$, and at level $1$ we have $(S_x)_1 = (P_x)_1$ with $(l_x)_1$ being the identity map.
Then there exists a pair of maps of prenets $g$ and $h$ fitting into a diagram
$$
\begin{tikzcd}[sep=3em]
S \ar[dr, "f"'] \ar[rr, "h"] && P \ar[dl, "g"] \\
& T &
\end{tikzcd}
$$
such that $P_x$ is the fiber prenet $g^{-1}(x)$ and the induced maps between fibers $h_x$ coincide with $l_x$. Moreover, if each $l_x$ is a weak equivalence, then $h$ is also a weak equivalence.
\end{proposition}
\proof
The construction of $P$ is performed using the same assembly process as in the proof of lemma \ref{lemma_loc_pi}. First we put $P_1 = S_1 = T_1$ and define all the maps
at level $1$ to be identities. Then at level two we combine all $(P_x)_2$ together and put
$$
P_2 = \bigsqcup_{x \in T_2} (P_x)_2,\quad\quad t(P) = \bigsqcup_{x \in T_2} t(P_x).
$$
The partial order $\pless$ on $P_2$ is generated by the partial orders on subsets $P_x$ and relations $y \pless y'$ for $y \in P_x$ and $y' \in P_{x'}$ such that
$y$ and $y'$ share a parent and $x < x'$ in $T_2$. Furthermore, if $y$ and $y'$ are incomparable with respect to $\pless$ then the sets of their parents are disjoint
and we put $y < y'$ if $p(y) < p(y')$. In view of remark \ref{rem_total_order} this determines the total order on $P_2$.

The map $g$ is defined by sending vertices from $(P_x)_2 \subset P_2$ to $x \in T_2$. By definition of the partial order on $P$ map $g$ preserves it. The parents $p_T(x)$
by assumptions of the proposition coincide with the subset $(S_x)_1 = (P_x)_1$, which in turn is the image $p((P_x)_2)$, hence the property (b) of \ref{def_prenet_map} is
also satisfied. Finally, since by assumption each $P_x$ was a prenet over $I$, the property (c) is also fulfilled. Hence $g$ is a map of prenets.

Now, we define $h$ at level $2$ as the coproduct of maps
$$
h_2 = \bigsqcup_{x \in T_2} l_x,
$$
and check that $h$ is also a map of prenets. Indeed, let $z$ and $z'$ be two vertices in $S_2$ sharing a parent, and such that $z < z'$. Then either they are in the same
fiber $S_x$, or they are in different fibers $S_x$ and $S_{x'}$ respectively, such that $x < x'$. In either case by definition of the partial order on $P$ we have
$h(z) \pless_{P} h(z')$, hence $h$ preserves the partial orders on $S$ and $P$. Since $l_x$ is a map of prenets we immediately see that both properties (b) and (c) of \ref{def_prenet_map}
is also fulfilled. Therefore, $h$ is a map of prenets. Finally, since properties (a) and (b) of \ref{def_weak_equiv} are also verified independently in each fiber, we conclude that
if each $l_x$ is a weak equivalence, then $h$ is also a weak equivalence.

\qed

\begin{definition}
\label{def_loc_const}
A map of prenets $f\from S  \to T$ is called {\em locally constant} if for every tassel collapse $g\from T \to T'$ the following two properties are satisfied.
\begin{enumerate}[label=\alph*)]
\item Let $x \in c(T'_2)$ is a core vertex and denote by $y \in g^{-1}(x)$ the only core vertex in the preimage of $x$, then there is a weak equivalence map between fibers
$$
w_{x\to y}\from (gf)^{-1}(x) \to f^{-1}(y).
$$
\item Let $x \in t(T'_2)$ be a tassel then for every tassel $y \in g^{-1}(x)$ we again have a weak equivalence map between fibers
$$
w_{x\to y}\from (gf)^{-1}(x) \to f^{-1}(y).
$$
\end{enumerate}
\end{definition}

\begin{remark}
The motivation behind the term ``locally constant'' map is of topological nature. Imagine prenets $S$ and $T$ as topological spaces, and think of vertices in $T_2$ as
points in this space. Assume that $x \in T_2$ is a core vertex, then subsets of tassels in $T$ that can be collapsed to $x$ are analogs of contractible neighborhoods of $x$.
Let $U$ be one such neighborhood, then the local constancy of $f$ can be interpreted by saying that the preimage $f^{-1}(U)$ can be homotopically retracted to the fiber
of $f$ over $x$. In other words the preimages $f^{-1}(U)$ for various contractible neighborhoods $U$ of $x$ coincide up to homotopy with $f^{-1}(x)$, and so the
presheaf $U \mapsto f^{-1}(U)$ is locally constant.

\end{remark}


\begin{lemma}
Composition of locally constant maps is again locally constant.
\end{lemma}
\proof
Consider a pair of locally constant maps $f$ and $g$, and let $u\from P \to P'$ be a tassel collapse.
$$
\begin{tikzcd}[sep=3em]
S \ar[r, "f"]  & T \ar[d, "v"'] \ar[r, "g"] & P \ar[d, "u"] \\
& T' \ar[r, "g'"'] & P'.
\end{tikzcd}
$$
Let us pick a vertex $x \in P'_2$, and to be specific assume that it is a core vertex, as the case of a tassel is handled in a similar fashion. Denote by $y \in P_2$ the only core
vertex in the preimage $u^{-1}(x)$. Since by assumption $g$ is a locally constant map we have a weak equivalence
$$
w_{x \to y}\from (ug)^{-1}(x) \to g^{-1}(y).
$$
Applying the assembly proposition above to the map $ug \from T \to P'$ we can use map $w_{x \to y}$ to construct a prenet $T'$ with fiber over $x$ replaced with $g^{-1}(y)$, and all
other fibers left intact. Since $w_{x \to y}$ is a tassel collapse the induced map $v$ is also a tassel collapse.

We need to construct a weak equivalence $(ugf)^{-1}(x) \to (gf)^{-1}(y)$. Observe that by construction of $T'$ the left side term $(ugf)^{-1}(x) = (g'vf)^{-1}(x) = (vf)^{-1}(g'^{-1}(x))$.
Let $z \in g'^{-1}(x)_2$ be a either a core vertex or a tassel, and choose a vertex $\tilde z \in v^{-1}(z)$. In fact, using the identification $g'^{-1}(x) = g^{-1}(y)$, we can take
$\tilde z \in g^{-1}(y)_2$ corresponding to $z$, which will be a core vertex whenever $z$ is a core vertex. Since $f$ is locally constant we have collection of weak equivalences
$$
w'_{z \to \tilde z}\from (fv)^{-1}(z) \to f^{-1}(\tilde z).
$$
As $z$ ranges over all vertices in $g'^{-1}(x)$ we obtain a family of weak equivalences between fibers in the diagram below.
$$
\begin{tikzcd}[sep=3em]
(g'vf)^{-1}(x) \ar[dr] \ar[rr] && (gf)^{-1}(y) \ar[dl] \\
& g'^{-1}(x) &
\end{tikzcd}
$$
Therefore the top horizontal arrow is a weak equivalence.
\qed

It follows trivially from the definition that any map $f\from S \to T$ to a minimal prenet $T$ is a locally constant map. The following lemma provides another large class of locally constant maps.

\begin{lemma}
\label{lemma_w_lc}
Any weak equivalence $f\from S \to T$ is a locally constant map.
\end{lemma}
\proof
Let $g\from T \to T'$ be a tassel collapse, and take a vertex $x \in T'_2$, again to be specific we assume that it is a core vertex, as the case of a tassel is handled similarly.
Then every vertex $y \in g^{-1}(x)$ has the same set of parents $p_T(y) = g^{-1}(p_{T'}(x))$. We denote by $\tilde x$ the unique core vertex in the preimage of $x$. We already saw
in lemma \ref{lemma_fib_w} that the fiber prenet $f^{-1}(\tilde x)$ is weakly equivalent to $I$ and fiber prenets over a tassel $y$ is weakly equivalent to $I^\circ$.

To construct retraction $w_{x \to \tilde x}\from (gf)^{-1}(x) \to f^{-1}(\tilde x)$ we observe, that since both maps $f$ and $g$ are weak equivalences the preimage
$(gf)^{-1}(x)$ is an interval in $S_2$ and all vertices in it have the same set of parents. Therefore we define $w_{x \to \tilde x}$ by sending all vertices $z$ with $f(z) < \tilde x$
to the minimal vertex in $f^{-1}(\tilde x)$ and all vertices $z$ with $\tilde x < f(z)$ to the maximal vertex in $f^{-1}(\tilde x)$.
\qed

Let us explore behavior of locally constant maps with respect to the lifting construction of proposition \ref{prop_lift}.

\begin{lemma}
\label{lemma_lc_vc_lift}
Let $f$ be a locally constant maps and $g$ a vertex collapse, then any lift $f'$ is a locally constant map.
\end{lemma}
\proof
Consider a tassel collapse $h'\from T' \to P'$. Then by lemma \ref{lemma_vc_pushout} we can construct the pushout $P$ such that $g''$ is a vertex collapse and $h$ is a tassel collapse.
$$
\begin{tikzcd}[sep=3em]
S' \ar[d, "g'"'] \ar[r, "f'"] & T' \ar[d, "g"] \ar[r, "h'"] & P' \ar[d, "g''"] \\
S \ar[r, "f"'] & T \ar[r, "h"'] & P.
\end{tikzcd}
$$
As before we will only consider the case of a core vertex, as the tassel case is handled similarly. Let $x \in P'_2$ be a core vertex and $y \in T'_2$ be the only core vertex in $h'^{-1}(x)$.
Using local constancy of $f$ we obtain a weak equivalence
$$
w_{g''(x) \to g(y)}\from (hf)^{-1}(g''(x)) \to f^{-1}(g(y)).
$$
Now consider the diagram
$$
\begin{tikzcd}[column sep=5em,row sep=3em]
(h'f')^{-1}(x) \ar[d] \ar[r, "w_{x \to y}"] & f'^{-1}(y) \ar[d] \\
(hf)^{-1}(g''(x)) \ar[r, "w_{g''(x) \to g(y)}"] & f^{-1}(g(y)).
\end{tikzcd}
$$
Since $g'$ is a weak equivalence it induces weak equivalences between the fibers, which means that both vertical arrows are weak equivalences. Since the bottom arrow is also
a weak equivalence by the 2-out-of-3 property (\ref{lemma_w_23}) we conclude that the top arrow is a weak equivalence. Therefore, $f'$ is a locally constant map.
\qed

However, a lift with respect to a general weak equivalence doesn't preserve the property of being locally constant. Which leads us to the following definition.

\begin{definition}
A map of prenets $f\from S \to T$ is called {\em universally} locally constant (abbreviated as ULC) if for any weak equivalence $g\from T' \to T$ there exists a locally constant
lift of $f$ along $g$.
\end{definition}

From lemma \ref{lemma_w_lc} it immediately follows that weak equivalences are ULC maps. The following characterization provides a rather convenient way of determining
if a map is ULC.

\begin{definition}
Let $f\from S \to T$ be a map of prenets, and consider two vertices $x, y \in T_2$.
\begin{enumerate}[label=\alph*)]
\item We say that $x$ and $y$ are {\em adjacent} if $p_T(x) = p_T(y)$, and for each of their parents $t$, vertices $x$ and $y$ are adjacent in the total order on the children of $t$.
\item We say that two adjacent vertices $x$ and $y$ such that $x < y$ have {\em compatible interface} if the partition of the set of their parents induced by $B(f^{-1}(x))$ and $A(f^{-1}(y))$
coincide.
\end{enumerate}
\end{definition}

\begin{lemma}
\label{lemma_ulc_criterion}
Let $f\from S \to T$ be a map of prenets. Then $f$ is ULC if and only if $f$ is locally constant and every pair of adjacent core vertices in $T$ have compatible interface.
\end{lemma}
\proof
First we prove the ``if'' part of the statement. Consider the diagram obtained by the lift in proposition \ref{prop_lift}
$$
\begin{tikzcd}[sep=3em]
S' \ar[d, "g'"'] \ar[r, "f'"] & T' \ar[d, "g"] \ar[r, "h"] & P \\
S \ar[r, "f"'] & T. &
\end{tikzcd}
$$
In virtue of lemma \ref{lemma_lc_vc_lift} it is enough to assume that $g$ is a tassel collapse, and moreover we may assume it is a simple tassel collapse with the target $z \in T_2$.
As before we assume that $z$ is a core vertex.
Let $h$ be another tassel collapse. If it doesn't involve vertices collapsed by $g$, in the sense that $h$ is injective on $g^{-1}(z)$, then by using the same argument as in the
proof of lemma \ref{lemma_lc_vc_lift} we obtain a retraction $w_{x \to y}$ for any appropriate pair of vertices $x \in P_2$ and $y \in h^{-1}(x)$. So we may assume, that $h$
is another simple tassel collapse that collapses a tassel $y \in g^{-1}(z)$ with some vertex $y' \in T'_2$, which is not in the preimage of $z$. Since by assumption $g$ is a simple
tassel collapse, vertex $y'$ maps injectively to some vertex $z' \in T_2$. To simplify the argument we may assume that $y$ is the only tassel in $g^{-1}(z)$, as the general case
can be treated by iterating the argument below.

To summarize the picture, we have a core vertex $z \in T_2$, another vertex $z' \in T_2$ and a tassel $y \in T'_2$. Let us denote by $\tilde z$ the only core vertex in the preimage of $z$
and by $\tilde z'$ the only vertex in the preimage of $z'$. To be specific, assume they are ordered as $\tilde z < y < \tilde z'$. Since both $\tilde z$ and $\tilde z'$ are candidates for
a tassel collapse with $y$ they must have the same set of parents in $T'$ and therefore $z$ and $z'$ have the same set of parents in $T$. Thus, vertices $z$ and $z'$ are adjacent in $T$.

Next we investigate the preimage $f'^{-1}(y)$. It consists entirely of tassels which by $g'$ collapse to the preimage $f^{-1}(z)$. More precisely, they have to collapse to the set of vertices
in $B(f^{-1}(z))$. Therefore, the preimage $f'^{-1}(y)$ maps surjectively to the set $B(f^{-1}(z))$ and for every vertex $t \in f'^{-1}(y)$ the set of its parents is the same as the set of
parents of its image in $B(f^{-1}(z))$. In fact this map gives rise to the retraction $w_{z \to \tilde z} \from (gf')^{-1}(z) \to f'^{-1}(\tilde z)$.

Now, we consider two cases. First, assume that $z'$ is a tassel, then we can consider a simple tassel collapse $l\from T \to R$ which collapses two vertices $z$ and $z'$ into a single
vertex $\bar z \in R_2$. Using local constancy of $f$ we obtain a retraction $w_{\bar z \to z}\from (lf)^{-1}(\bar z) \to f^{-1}(z)$. Existence of this map implies that
the preimage $f^{-1}(z')$ has a similar description as we had for $f'^{-1}(y)$, and as $g'$ is also a tassel collapse the same holds for the preimage $f'^{-1}(\tilde z')$.

Map $h$ collapses two tassels $y$ and $\tilde z'$, and as we just saw both preimages of $y$ and $\tilde z'$ have a similar description, therefore the same holds for the preimage of their union.
Hence we can construct retractions $w_{h(y) \to y}$ and $w_{h(y) \to \tilde z'}$ as required in \ref{def_loc_const} (b).

Now consider the second case, when $z'$ is a core vertex. Since $z$ and $z'$ are adjacent in $T$ then by assumption they have compatible interface, which identifies $B(f^{-1}(z))$ and $A(f^{-1}(z'))$.
Using this identification and the description of the preimage $f'^{-1}(y)$ we construct a retraction $w_{h(y) \to \tilde z'}$ as required in \ref{def_loc_const} (a).

The opposite implication in the statement of the lemma is done by reversing the argument above, as existence of retractions induced by tassel collapses $g$ and $h$ ensures identification
of sets $B(f^{-1}(z))$ and $A(f^{-1}(z')$ for all adjacent vertices.
\qed

\begin{remark}
In fact we can state a somewhat stronger property of ULC maps. Consider a lift of $f$ along a weak equivalence $g$.
$$
\begin{tikzcd}[sep=3em]
S' \ar[d, "g'"'] \ar[r, "f'"] & T' \ar[d, "g"] \\
S \ar[r, "f"'] & T.
\end{tikzcd}
$$
Let $z$ be a core vertex in $c(T_2)$ and let $\tilde z$ be the unique core vertex in the preimage $g^{-1}(z)$. We say that the lift is {\em nice} if for every core vertex $z$ map $g'$ induces a surjection
of $f'^{-1}(\tilde z)$ onto $f^{-1}(z)$. For example, the lift constructed in proposition \ref{prop_lift} is nice.
It is clear from the proof of the lemma that if $f$ is a ULC map, then every nice lift of $f$ is locally constant. For an example of a lift of a ULC map that is not locally constant we
refer to \ref{exa_ulc_lift_not_lc}.

It is clear that the composition of ULC maps is again a ULC map. We will denote the subcategory of $2\PreNet$ consisting of ULC maps by $2\PreNet^\ulc$.

\end{remark}

\begin{remark}
It is reasonable to consider a stronger notion. We say that a map $f\from S \to T$ is {\em strongly} locally constant (abbreviated SLC), if {\em any} lift with respect to any weak equivalence
is locally constant. Although, this notion appears to be too restrictive for our purposes, it may be interesting to highlight the difference between ULC and SLC maps in more elementary terms.
In addition to being locally constant and having compatible interfaces of the adjacent core vertices, being an SLC map also imposes a restriction on the fibers $f^{-1}(z)$ for core vertices
$z \in T_2$, that can be expressed in the same flavor as the property (c) of \ref{def_prenet_map}.

For every parent vertex $y \in S_1$ we denote by $a_t(y) \subset p_S^{-1}(y)$ the maximal initial interval of tassels among children of $y$ (that is $a_t(y)$ is an interval, consists entirely
of tassels that have $y$ as one of their parents, there is no vertex $z \in p_S^{-1}(y)$ such that $z < a_t(y)$, and $a_t(y)$ is maximal among all subsets satisfying these three properties).
Let us denote by $A_t = A_t(f^{-1}(z))$ the union
$$
A_t := \bigcup_{y \in f^{-1}(p_T(z))} a_t(y).
$$
Moreover, we define $a_c(y)$ to be the first core vertex in $f^{-1}(z) \cap p_S^{-1}(y)$ and define
$$
A_c := \bigcup_{y \in f^{-1}(p_T(z))} a_c(y).
$$
Then SLC condition implies that whenever $A_t$ covers the set of parents $f^{-1}(p_T(z))$ there is a tassel collapse retracting set of tassels $A_t$ to $A_c$.

Clearly, SLC maps form a subcategory of $2$-prenets that we will denote $2\PreNet^\slc$.

\end{remark}

\begin{nparagraph}[Locally constant nets.]
We define the category of (strongly) locally constant $2$-nets as the localization of the category $2\PreNet^\ulc$ (respectively $2\PreNet^\slc$) with respect to the class of weak equivalences.
\begin{align*}
2\Net^\ulc \ &:=\ 2\PreNet^\ulc[\cW^{-1}], \\
2\Net^\slc \ &:=\ 2\PreNet^\slc[\cW^{-1}].
\end{align*}
Because the class of ULC maps admits lifts along weak equivalences, we can give a description of maps in this localized category similar to theorem \ref{thm_net_homs}. In fact
the situation here is somewhat simplified. We begin by defining an analog of the equivalence relation \ref{def_equiv}.
\end{nparagraph}

\begin{definition}
Denote by $\simo$ a relation on the sets of maps $\Hom_{2\PreNet}(S, T)$ consisting of pairs
$\begin{tikzcd}[cramped,sep=small]S \ar[shift left, r, "f"] \ar[shift right, r, "g"'] & T\end{tikzcd}$ admitting a factorization
$$
\begin{tikzcd}
S \ar[shift left, r, "f_w"] \ar[shift right, r, "g_w"'] & T' \ar[shift left, r, "f_v"] \ar[shift right, r, "g_v"'] & T,
\end{tikzcd}
$$
where $f_v$ and $g_v$ are vertex collapses obtained by factorization in \ref{lemma_map_factor}, and the fibers of $f_w$ and $g_w$ are pairwise weakly equivalent.
\end{definition}
It is immediate to see that $\simo$ is already an equivalence relation, unlike the situation with its general counterpart $\sim$. We may consider sets of classes with
respect to this equivalence relation $\Hom_{2\PreNet}(S, T) / \simo$, however while they admit an analog of the push forward map $m_*$ in (\ref{equ_piHom_cat}) the pullback $l^*$ is in general not well defined.
Nevertheless, the following lemma shows that in the subcategory $2\PreNet^\ulc$ the equivalence relation $\simo$ has both functoriality properties.

\begin{lemma}
\label{lemma_lc_equiv}
Let $h\from S \to T$ be a locally constant map, and assume that a pair of maps $f_w, g_w\from T \to P$ have pairwise weakly equivalent fibers, and are such that $T_1 = P_1$ and $(f_w)_1 = (g_w)_1 = \id_{T_1}$.
Then the compositions $f_w h$ and $g_w h$ also have weakly equivalent fibers.
\end{lemma}
\proof
For every vertex $x \in P_2$ the fiber prenets $f_w^{-1}(x)$ and $g_w^{-1}(x)$ are weakly equivalent, and denote by $M_x$ their minimal representative. Choose weak equivalence maps (in fact tassel collapses)
$u_x\from f_w^{-1}(x) \to M_x$ and $v_x\from g_w^{-1}(x) \to M_x$. Then using the assembly process of \ref{prop_assembly} we construct a prenet $T'$ and two tassel collapses $f_t$ and $g_t$ as in the following
diagram.
$$
\begin{tikzcd}[sep=3em]
S \ar[r, "h"] & T \ar[shift left, dl, "\tilde f_t"] \ar[shift right, dl, "\tilde g_t"'] \ar[shift left, d, "f_t"] \ar[shift right, d, "g_t"'] \ar[shift left, r, "f_w"] \ar[shift right, r, "g_w"'] & P \\
\wtilde T' \ar[d, "q"'] \ar[r] & T' \ar[ur, "l"'] & \\
m'(T) & &
\end{tikzcd}
$$

Since the compositions $f_w h = l f_t h$ and $g_w h = l g_t h$ it is enough to show that $f_t h$ and $g_t h$ have pairwise weakly equivalent fibers. We consider two cases, first let $z \in T'_2$
be a core vertex and denote by $\tilde z$ the only core vertex in the preimages $f_t^{-1}(z)$ and $g_t^{-1}(z)$. Notice that this is the same vertex $\tilde z$ in both cases, as
$f_t$ and $g_t$ induce bijection between the core vertices. Then local constancy of $h$ gives us two weak equivalences
\begin{align*}
w^f_{z \to \tilde z}&\from (f_t h)^{-1}(z) \to h^{-1}(\tilde z), \\
w^g_{z \to \tilde z}&\from (g_t h)^{-1}(z) \to h^{-1}(\tilde z),
\end{align*}
which imply that the fibers of compositions $f_t h$ and $g_t h$ over $z$ are weakly equivalent.

Now, we assume that $z$ is a tassel. If the preimages $f_t^{-1}(z)$ and $g_t^{-1}(z)$ have at least one common tassel $\tilde z$, then we use the same argument as before to show that
fibers $(f_t h)^{-1}(z)$ and $(g_t h)^{-1}(z)$ are weakly equivalent to $h^{-1}(\tilde z)$ and therefore weakly equivalent to each other. If such common tassel $\tilde z$ doesn't exist,
then we utilize the construction in the proof of lemma \ref{lemma_vc_tc_pair} (B). We form the prenet $\wtilde T'$ by adding two tassels around each core vertex of $T'$ and lift
$f_t$ and $g_t$ to $\tilde f_t$ and $\tilde g_t$ respectively, with the latter two maps belonging to the class $\cW'$ (see lemma \ref{lemma_term_prenet}). Let $m'(T)$ be the terminal
object of $\cW'(T)$ and $q$ the unique tassel collapse in $\cW'$. Then under the composition $q \tilde f_t = q \tilde g_t$ all tassels in the preimages $f_t^{-1}(z)$ and $g_t^{-1}(z)$
go to the same tassel $\bar z$ in $m'(T)$. Since $f$ is locally constant we see that for any two tassels $\tilde z \in f_t^{-1}(z)$ and $\tilde z' \in g_t^{-1}(z)$ the fibers
$h^{-1}(\tilde z)$ and $h^{-1}(\tilde z')$ are weakly equivalent, and therefore the fibers $(f_t h)^{-1}(z)$ and $(g_t h)^{-1}(z)$ are also weakly equivalent.
\qed

\begin{remark}
\label{rem_equivo_tc_vc}
Let $f$ and $g$ be a pair of maps, then using the assembly process as in the proof of the previous lemma we see that $f \simo g$ if and only if they admit factorization
\begin{equation}
\label{equ_equivo_tc_vc}
\begin{tikzcd}
S \ar[shift left, r, "f_t"] \ar[shift right, r, "g_t"'] & Q \ar[r, "l"] &  R \ar[shift left, r, "f_v"] \ar[shift right, r, "g_v"'] & T,
\end{tikzcd}
\end{equation}
where $f_t$ and $g_t$ are tassel collapses and $f_v$ and $g_v$ are vertex collapses from factorization \ref{lemma_map_factor}.
Moreover, if both $f$ and $g$ are ULC maps, then map $l$ is also ULC, and if both $f$ and $g$ are SLC maps, then map $l$ is also SLC.
Comparing it with the factorization in remark \ref{rem_equiv_tc_vc}, we can highlight the purpose of locally constant maps.
While we can {\em slide down} a vertex collapse along any map of prenets thanks to factorization \ref{lemma_map_factor}, in general we can not
slide a tassel collapse either way. What local constancy allows us to do is to be able to {\em slide up} a tassel collapse along a locally constant map.
\end{remark}

In view of the discussion above, let us denote
$$
\pi^\circ\Hom(S, T) := \Hom_{2\PreNet^\ulc}(S, T) / \simo.
$$
From the previous remark it follows that the localization functor factors via these sets.
$$
\begin{tikzcd}
L\from \Hom_{2\PreNet^\ulc}(S, T) \ar[r] & \pi^\circ\Hom(S, T) \ar[r] & \Hom_{2\Net^\ulc}(S, T).
\end{tikzcd}
$$
We have the following description of maps in the category of locally constant $2$-nets.

\begin{theorem}
\label{thm_ulcnet_homs}
For any two prenets $S$, $T$ we have
$$
\Hom_{2\Net^\ulc}(S, T) \ \isom\ \colim_{\wtilde S \in \cW(S)} \pi^\circ\Hom(\wtilde S, m(T)) \ \isom\ \colim_{\wtilde S \in \cW(S)_1} \pi^\circ\Hom(\wtilde S, m(T)).
$$
\end{theorem}
\proof
The proof goes the same way as in theorem \ref{thm_net_homs} and corollary \ref{cor_net_homs}, as all maps can be considered to have ULC property. The major change
and simplification, is the replacement of lemma \ref{lemma_for_thm_loc} with the following statement.

Let $f, g\from S \to T$ be a pair of locally constant maps, such that $f_1 = g_1 = \id_{S_1}$ and let $h\from T \to T'$ be a tassel collapse, such that $hf = hg$. Then we claim
that $f \simo g$.
$$
\begin{tikzcd}[sep=3em]
S \ar[shift left, r, "f"] \ar[shift right, r, "g"'] & T \ar[r, "h"] & T'.
\end{tikzcd}
$$
We need to check that the fibers of $f$ and $g$ are pairwise weakly equivalent. Take $x \in T_2$, and assume that it is a core vertex, as the case of a tassel is handled similarly.
Denote by $S_x$ the fiber prenet $S_x = (hf)^{-1}(h(x))$. Then the local constancy of $f$ and $g$ gives us a pair of weak equivalences
\begin{align*}
w^f_{h(x) \to x} &\from S_x \to f^{-1}(x), \\
w^g_{h(x) \to x} &\from S_x \to g^{-1}(x).
\end{align*}
This immediately shows that $f$ and $g$ have pairwise weakly equivalent fibers.

\qed

In view of this description of maps in $2\Net^\ulc$ we can identify the category $2\Net^\slc$ with the subcategory of $2\Net^\ulc$, consisting of maps $f\from S \to T$, admitting
an SLC representative $\tilde f\from \wtilde S \to T$.

Moreover, denote by $2\Net^{\ulc,s}$ the full subcategory of $2\Net^\ulc$ consisting of special prenets.

\begin{corollary}
\label{cor_comp_special}
The comparison functor
$$
K^s\from 2\PreNet^{\ulc, s}[\cW^{-1}] \to 2\Net^{\ulc,s}
$$
is an equivalence of categories.
\end{corollary}
\proof
We need to check that $K$ is both injective and surjective on hom-sets. Consider a map $f\from S \to T$ in $2\Net^{\ulc,s}$ represented by a map of prenets $\tilde f\from \wtilde S \to T$.
Since $S$ is a special prenet, in view of lemma \ref{lemma_w_special} $\wtilde S$ is also a special prenet, therefore $\tilde f$ is a map in the category $2\PreNet^{\ulc,s}$. This implies surjectivity
of $K$ on hom-sets.

To see the injectivity of $K$ consider two maps $f, g \from \wtilde S \to T$ such that $f \simo g$. As we saw before $f$ and $g$ admit factorization as in (\ref{equ_equivo_tc_vc}). By assumption
both $S$ and $T$ are special prenets, therefore again using lemma \ref{lemma_w_special} we see that $Q$ and $R$ are also special prenets, hence the factorization lies entirely in the
subcategory $2\PreNet^{\ulc,s}$. Furthermore, lemma \ref{lemma_vc_tc_pair} also holds in the subcategory $2\PreNet^{\ulc,s}$ since as soon as any one of the prenets $S$ or $T$ is special all
prenets in the proof of the lemma are also special. Therefore in the localized category images of $f_t$ and $g_t$ coincide, and similarly images of $f_v$ and $g_v$ also coincide.
Hence $f$ and $g$ also coincide in $2\PreNet^{\ulc,s}[\cW^{-1}]$, which implies the injectivity of the comparison functor.

\qed

We would like to point out that this argument doesn't allow us to show that the comparison functor
$$
2\PreNet^s[\cW^{-1}] \to 2\Net^s
$$
is an equivalence. The main issue here lies in the difference between equivalence relations $\sim$ and $\simo$. Consider two maps $f, g \from S \to T$ between two special prenets $S$ and $T$,
such that $f \sim g$, then in the factorization in remark \ref{rem_equiv_tc_vc} we can only conclude that $R$ is a special prenet, while prenets $P$ and $Q$ do not have to be special.
Which means that we can not conclude that images of $f$ and $g$ coincide in the localized category $2\PreNet^s[\cW^{-1}]$.

\begin{proposition}
The inclusion $2\PreNet^\ulc \into 2\PreNet$ induces a comparison functor between localizations $K\from 2\Net^\ulc \to 2\Net$, such that for any two prenets $S$ and $T$ the map
$$
K\from \Hom_{2\Net^\ulc}(S, T) \to \Hom_{2\Net}(S, T)
$$
is surjective.
\end{proposition}
\proof
We need to check that every map of nets has a ULC representative. Consider $f\from \wtilde S \to T$ representing a map of $2$-nets from $S$ to $T$. We may assume that $T$ is minimal,
so that $f$ is already locally constant. We need to ensure that every pair of adjacent core vertices has compatible interface. Assume that $z$, $z'$ with $z < z'$ is an adjacent pair
of core vertices with incompatible interface. Consider another prenet $T'$ obtained from $T$ by adding a new tassel $y$ between $z$ and $z'$. We have a tassel collapse $u\from T' \to T$
that sends $y$ to $z$ and a tassel collapse $v\from T' \to T$ that sends $y$ to $z'$ and are identities on the other vertices.
$$
\begin{tikzcd}[sep=3em]
S' \ar[d, "u'"'] \ar[r, "f'"] & T' \ar[d, "u"] \ar[r, "v"] & T \\
S \ar[r, "f"'] & T &
\end{tikzcd}
$$
We form a lift $S'$ as in proposition \ref{prop_lift}. By construction vertices $z$ and $z'$ now have compatible interface with respect to map $vf'$, moreover in the category
of $2$-nets the image of map $vf'$ coincide with the image of $uf'$ which in turn coincides with the image of $f$. It is clear that this process doesn't change interfaces
between any other adjacent vertices, therefore iterating it we obtain a ULC map $\tilde f$ which represent the same map $f$ in $2$-nets.
\qed

\begin{nparagraph}
In the end of section \ref{sec_nets} we have constructed a fully faithful embedding $J$ of the category of $2$-trees with maps surjective at level $1$ into the category of $2$-nets.
In fact this inclusion factors through the category of locally constant nets, where it lands in the subcategory of strongly locally constant nets.
$$
\begin{tikzcd}[sep=3em]
(2\Tree, \Epi_1) \ar[into, r, "J"] & 2\Net^\slc / I \ar[into, r] & 2\Net^\ulc / I \ar[epi, r] & 2\Net / I.
\end{tikzcd}
$$
Abusing notation we will also denote by $J$ the inclusion of trees into (strongly) locally constant $2$-nets.
Indeed, first of all it is clear that every prenet $T$ which is a tree, in the sense that for every vertex $x \in T_2$
the set of parents $p_T(x)$ consists of a single vertex, admits a map to $I$. This is because the disjointedness property for subsets $A$ and $B$ from \ref{def_prenet_map} (c) is automatically
fulfilled, due to the fact that every vertex has only one parent.

It remains to see that for any map of $2$-trees $f\from S \to T$, the map of prenets $J(f)\from \tilde S \to J(T)$ constructed in paragraph \ref{par_J_map} is SLC.
We will need the following simple observation. Let $g\from P \to Q$ be a map of prenets, which are both trees, then $g$ is locally constant. To see that, consider two adjacent
vertices $x$ and $y$, in $Q_2$ which can be collapsed by some $u\from Q \to Q'$. To be specific, assume that $x$ is a core vertex and $x < y$.
The preimages $g^{-1}(x)$ and $g^{-1}(y)$ have the same set of parents, but since $P$ is a tree, for every parent $t$ the tassels in $f^{-1}(y) \cap p_P^{-1}(t)$ can be collapsed
to the last vertex in $f^{-1}(x) \cap p_P^{-1}(t)$, which gives rise to the retraction $w_{u(x) \to x}$.

Now observe, that if $P$ is a tree and $v\from P' \to P$ is a tassel collapse, then $P'$ is also necessarily a tree. This implies that any lift of a map between trees along a tassel
collapse is again a map between trees, which as we just saw is also locally constant. Therefore any map between trees is an SLC map.

\end{nparagraph}

\begin{nparagraph}[Cardinality and fiber functors.]
Now we are ready to describe the operadic structure on the category of locally constant $2$-nets over $I$.
We start by defining the cardinality functor
$$
\begin{tikzcd}
2\Net^\ulc / I \ar[r, "|-|"] & \FSets,
\end{tikzcd}
$$
by sending a $2$-net $S$ to the set of core vertices at level $2$, in other words we put
$$
|S| = c(S_2).
$$
Let $\tilde f\from \wtilde S \to T$ be a locally constant map representing a map $f\from S \to T$ of $2$-nets, in particular we assume that prenet $T$ is minimal. We define
$|f|$ to be the restriction of $\tilde f$ to the set of core vertices $c(\wtilde S_2)$. Since restriction of weak equivalences to the core vertices at level $2$ are identities,
the resulting map $|f|$ doesn't depend on the choice of the lift $\wtilde S$. Moreover, if we have two maps $\tilde f$ and $\tilde f'$ such that $\tilde f \simo \tilde f'$, then
for any vertex $x \in c(T_2)$ the fiber prenets $\tilde f^{-1}(x)$ and $\tilde f'^{-1}(x)$ are weakly equivalent, hence contain the same subset of core vertices of $\wtilde S_2$.
Therefore the restrictions of $\tilde f$ and $\tilde f'$ to $c(\wtilde S_2)$ coincide.

For an element $x \in |T| = c(T_2)$ we define
the fiber functor
$$
\Fib_x \from 2\Net^\ulc / T \to 2\Net^\ulc / I
$$
by putting
$$
\Fib_x(S) := \tilde f^{-1}(x),
$$
where the fiber prenet $\tilde f^{-1}(x)$ was defined in \ref{def_fiber}. Let us show that this notion of fiber is well-defined, in the sense that it doesn't
depend on the choice of the representing map $\tilde f$. First, consider another prenet $\wtilde S'$ with a weak equivalence $w\from \wtilde S' \to \wtilde S$.
$$
\begin{tikzcd}
\tilde f'\from \wtilde S' \ar[r, "w"] & \wtilde S \ar[r, "\tilde f"] & T.
\end{tikzcd}
$$
Then according to lemma \ref{lemma_fib_map} $w$ induces a map between fibers $w_x\from \tilde f'^{-1}(x) \to \tilde f^{-1}(x)$, which is also a weak equivalence.

Now, consider two maps $\tilde f, \tilde g \from \tilde S \to T$, such that $\tilde f \simo \tilde g$, thus representing the same map $f$ of $2$-nets.
By definition we have factorizations of $\tilde f$ and $\tilde g$ as
$$
\begin{tikzcd}
\tilde S \ar[shift left, r, "f_w"] \ar[shift right, r, "g_w"'] & T' \ar[shift left, r, "f_v"] \ar[shift right, r, "g_v"'] & T,
\end{tikzcd}
$$
Since maps $f_v$ and $g_v$ are identities on the second level, then by identifying $T'_2 = T_2$ we have equalities of the fiber prenets $\tilde f^{-1}(x) = f_w^{-1}(x)$
and $\tilde g^{-1}(x) = g_w^{-1}(x)$. Since $f_w$ and $g_w$ have pairwise weakly equivalent fibers we conclude that fibers of $\tilde f$ and $\tilde g$ over $x$ are also
weakly equivalent to each other.

\end{nparagraph}

\begin{nparagraph}
Consider a map $h\from P \to S$ in the category $2\Net^\ulc / T$. By choosing representatives for $P$ and $S$ and performing a lift we obtain a diagram in prenets.
$$
\begin{tikzcd}
\wtilde P \ar[dr, "\tilde g"'] \ar[rr, "\tilde h"] && \wtilde S \ar[dl, "\tilde f"] \\
& T &
\end{tikzcd}
$$
By lemma \ref{lemma_fib_map} we obtain a map between fibers
$$
\tilde h_x \from \tilde g^{-1}(x) \to \tilde f^{-1}(x),
$$
which represents a map in the category of $2$-nets between fibers $h_x\from \Fib_x(P) \to \Fib_x(S)$. Let us show that it is well defined. First, consider two weak equivalences
$w_P\from \wtilde P' \to \wtilde P$ and $w_S\from \wtilde S' \to \wtilde S$ and another representative $\tilde h'\from \wtilde P' \to \wtilde S'$ such that $\tilde h w_P = w_S \tilde h'$.
Then by restricting to the fibers we have $(w_P)_x$ and $(w_S)_x$ are still weak equivalences, and therefore $\tilde h_x$ and $\tilde h'_x$ represent the same map of $2$-nets.

Next, consider two maps $\tilde h$ and $\tilde h'$ such that $\tilde h \simo \tilde h'$. By definition they factorize through a pair of vertex collapses $h_v$ and $h'_v$ as well
as pair of maps $h_w$ and $h'_w$ with pairwise weakly equivalent fibers. Restricting to the fibers over $x \in T_2$, we can as before ignore the vertex collapses, as they don't affect
the resulting fiber prenets, and since $h_w$ and $h_w'$ had weakly equivalent fibers the same holds for their restrictions to the fiber over $x$.

Let $\tilde f \simo \tilde f'$ be two maps representing $f$, then $\tilde h$ induces two maps $\tilde h_x\from \tilde g^{-1}(x) \to \tilde f^{-1}(x)$ and
$\tilde h'_x\from \tilde g^{-1}(x) \to \tilde f'^{-1}(x)$, and we need to show that they represent the same map in $2$-nets. We know that the fibers $\tilde f^{-1}(x)$
and $\tilde f'^{-1}(x)$ are weakly equivalent, so we can denote by $M$ their common $2$-minimal representative. Choose two tassel collapses $u_x\from \tilde f^{-1}(x) \to M$
and $u'_x\from \tilde f'^{-1}(x) \to M$, and using the assembly (\ref{prop_assembly}) we construct a prenet $S'$ with a pair of tassel collapses $u$ and $u'$.
$$
\begin{tikzcd}
\wtilde P \ar[r, "\tilde h"] & \wtilde S \ar[shift left, r, "u"] \ar[shift right, r, "u'"'] & S'.
\end{tikzcd}
$$
Since $\tilde h$ is a locally constant map, by lemma \ref{lemma_lc_equiv} the compositions $u \tilde h$ and $u' \tilde h$ have pairwise weakly equivalent fibers. This property
is preserved when restricting to the fibers over $x$, therefore we conclude that $u_x \tilde h_x \simo u'_x \tilde h'_x$ and since $u_x$ and $u'_x$ are weak equivalences
maps $\tilde h_x$ and $\tilde h'_x$ represent the same map of $2$-nets. This establishes that $h_x$ is a well defined map of $2$-nets.


It is clear that this construction sends composition of maps in $2$-nets to composition of maps between fibers, therefore it gives rise to a functor $\Fib_x$
for every vertex $x \in T_2$ and in particular for every core vertex.

\end{nparagraph}

\begin{nparagraph}
Let $f\from S \to T$ be a map of locally constant $2$-nets. It is clear from the definition of the fiber functor $\Fib$ that for any core vertex $x \in T_2$
we have $|\Fib_x(S)| = |f|^{-1}(x)$.

The category $2\Net^\ulc / I$ obviously has the terminal object $I$. We assume from now on that we are working with a skeletal category of $2$-nets,
more specifically, the full subcategory of $2\Net^\ulc / I$ consisting of minimal prenets.

For any $2$-net $S$ the fiber prenet over any vertex $x$ of the identity map $\id_S$ is the prenet $\<x\>$ generated by vertex $x$, which as was seen in lemma \ref{lemma_fib_w}
is weakly equivalent to $I$ if $x$ is the core vertex and $I^\circ$ if $x$ is a tassel. Therefore, for every $x \in |S|$ the fiber $\Fib_x(\id_S)$ is the terminal object $I$.
Furthermore, for the unique map $z\from S \to I$, the fiber over the only vertex $x \in I_2$ is the entire prenet $S$, thus the resulting functor
$\Fib_x \from 2\Net^\ulc / I \to 2\Net^\ulc / I$ is the identity.

The remaining conditions of the operadic category are verified by first choosing appropriate representatives in $2$-prenets of all the maps in question and then checking the
required conditions in the category $2\PreNet$, which is straightforward.

\end{nparagraph}

\begin{remark}
We would like to point out here that the category $2\PreNet$ itself is not an operadic category. For example it does not satisfy condition that the fibers of the identity map
$\id_S$ are the terminal objects, as can be seen by considering the following prenet $S$.

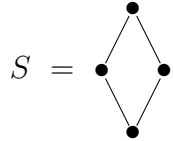
$$
S \ =\  
\begin{tikzpicture}[inner sep=0pt,baseline=(a11.base)]
\def\u{2em}
\node (a21) at (0, 0) {$\bullet$};
\node (a11) at (-0.5*\u, -1*\u) {$\bullet$};
\node (a12) at (0.5*\u, -1*\u) {$\bullet$};
\node (a01) at (0, -2*\u) {$\bullet$};

\draw (a21) -- (a11) -- (a01);
\draw (a21) -- (a12) -- (a01);
\end{tikzpicture} 
$$
\end{remark}

\vskip 1em

We finish this section with the discussion about factorization of maps in the category of locally constant nets into a chain of elementary maps. While often such a
factorization can be obtained as a consequence of the (weak) blow-up axiom of an operadic category, in our case the weak blow-up axiom is not satisfied. This is
essentially due to the fact that the map constructed using assembly process of \ref{prop_assembly} in general is not a ULC map. Nevertheless, we can still use
the assembly process to establish the following theorem. For an illustration of the construction used in its proof we refer the reader to example \ref{exa_factor_elem}

\begin{theorem}
\label{thm_factor_elem}
Let $f\from S \to T$ be a map in the category of locally constant $2$-nets, then there exists a factorization of $f$
$$
\begin{tikzcd}
S \ar[r, "f^{(1)}"] & S^{(1)} \ar[r, "f^{(2)}"] & \cdots \ar[r, "f^{(n)}"] & S^{(n)} = T,
\end{tikzcd}
$$
such that each map $f^{(i)}$ is elementary.
\end{theorem}
\proof
Let $\tilde f\from \wtilde S \to T$ be a ULC map representing $f$, assume that $T$ is minimal, and denote by $n$ the number of core vertices in $T_2$. We fix an order
on this set of core vertices, for instance we can take the order induced by the total order on $T_2$, and construct factorization inductively.

{\bf Assembly.}
First we use \ref{lemma_map_factor} to factorize $\tilde f$ into a map $f'\from \wtilde S \to T'$ and a vertex collapse $f_v \from T' \to T$.
Pick a core vertex $z$ in $T'_2$. If it has an adjacent vertex $z'$ then due to minimality of $T$ it has to be a core vertex.
To be specific, assume that $z < z'$. If this core vertex $z'$ has already been processed by the inductive construction, we don't need to perform any further operations,
and if it wasn't then we first construct a tassel collapse $u\from\wtilde T \to T'$ by adding a tassel $y$ between $z$ and $z'$ adjacent to both and putting $u(y) = z'$.
Then we construct a lift $\wtilde S'$ of $f'$ along the tassel collapse $u$ as in \ref{prop_lift}.
$$
\begin{tikzcd}[sep=3em]
\wtilde S' \ar[d, "u'"'] \ar[r, "\tilde f'"] & \wtilde T \ar[d, "u"] \\
\wtilde S \ar[r, "f'"'] & T'
\end{tikzcd}
$$
Consider the set of fiber prenets $F_x = (f'u')^{-1}(x)$ for all vertices (both core and tassels) in $T'_2$. For the vertex $z$ we consider the factorization of the unique map
to the terminal prenet $I$ via a vertex collapse
$$
\begin{tikzcd}
F_z \ar[r, "l_z"] & P_z \ar[r] & I,
\end{tikzcd}
$$
and take $l_z$ to be the first map in this decomposition.

For $z'$ we first construct a prenet $P_{z'}$ by taking the fiber prenet $\wtilde S_{z'} = f'^{-1}(z')$ and then putting $(P_{z'})_1 = (\wtilde S_{z'})_1$,
$(P_{z'})_2 = \{t\} \sqcup (\wtilde S_{z'})_2$, where the new tassel $t$ is put first in the total order, and then define parents $p(t) = (P_{z'})_1$ and for the
rest of the vertices parents are the same as in $\wtilde S_{z'}$. We take the map $l_{z'}\from F_{z'} \to P_{z'}$ that sends all vertices in $\tilde f'^{-1}(y)$ to $t$
and the rest bijectively to $(\wtilde S_{z'})_2$.

For all the other vertices we put $l_x = \id_{F_x}$. Now we use assembly process \ref{prop_assembly} to construct a prenet $S^{(1)}$.
$$
\begin{tikzcd}[sep=3em]
\wtilde S' \ar[dr, "f'u'"'] \ar[rr, "h"] && S^{(1)} \ar[dl, "g"] \\
& T' &
\end{tikzcd}
$$

{\bf Map $h$ is ULC.} Let us denote by $\tilde z$ the only vertex in the preimage $g^{-1}(z)$ and by $t$ the image of the new tassel $t$ in $P_{z'}$.
Consider a simple tassel collapse $v\from S^{(1)} \to R$. Since $h$ outside of fibers over $z$ and $z'$ is identity, if $v$ doesn't involve any vertices from these two
fibers we immediately have the existence of required retractions in definition \ref{def_loc_const}, as identity map is ULC. If $v$ involves only vertices in the fiber $P_{z'}$ except the tassel $t$,
then again restriction of $h$ to these vertices is identity and we have the required retractions.

It remains to consider two possibilities, first consider collapse of vertices $\tilde z$ and $t$. The preimages $h^{-1}(\tilde z) = F_z$ and $h^{-1}(t) = \tilde f'^{-1}(y)$ by the assembly construction.
Since $f$ is ULC, the two vertices $z$ and $z'$ have compatible interface, which implies that the subdivision of the set of parents induced by vertices in $h^{-1}(t)$ is the same as the
subdivision induced by $B(F_z)$, therefore we can construct a tassel collapse
$$
w_{v(\tilde z) \to \tilde z} \from (vh)^{-1}(h(\tilde z)) \to F_z.
$$

Now, if $t$ doesn't have an adjacent vertex in $P_{z'}$ then there are no more possibilities to consider. However, if it has an adjacent vertex, say $t'$, then by construction of $l_{z'}$ the
preimage $h^{-1}(t)$ consists of a single vertex, which implies that $h$ restricted to the fiber over $z'$ is identity and we conclude by the same argument as before.

Finally, we need to check that every pair of adjacent vertices in $S^{(1)}$ have compatible interface. Since $h$ is identity outside of vertices $\tilde z$ and $t$, the condition is trivially
satisfied. Since $t$ is a tassel, we don't need to check pairs involving $t$. Vertex $\tilde z$ has an adjacent vertex $t$ on one side, which we don't need to check, so assume
that it has another adjacent core vertex $s < \tilde z$. By inductive procedure, vertices $g(s)$ and $z$ have a compatible interface with respect to map $f'u'$ and therefore
$s$ and $\tilde z$ also have compatible interface with respect to $h$.

{\bf Map $g$ is ULC.} Since by assumption $T$ is minimal, map $g$ is locally constant. Moreover, as outside of vertices $z$ and $z'$ map $g$ coincides with $f'u'$, all pairs
of core vertices not involving either $z$ or $z'$ have compatible interface. It remains to consider three cases, first by construction the pair $(z, z')$ has compatible interface
(the corresponding partition of sets of their parents consists of a single interval). Let $s < z$ be a core vertex adjacent to $z$, then by inductive assumption $s$ and $z$
have a compatible interface with respect to $f'u'$ (again partition consists of a single interval) and therefore they are also compatible with respect to map $g$. Finally,
assuming that it exists, let $s'$ be a core vertex adjacent to $z'$ such that $z' < s'$, then $B((f'u')^{-1}(z'))$ induces the same partition of parents as $A((f'u')^{-1}(s'))$.
Moreover, we have $B(g^{-1}(z')) = B((f'u')^{-1}(z'))$, therefore $z'$ and $s'$ have compatible interface with respect to $g$.

{\bf Map $h$ is elementary.} Since outside of fibers over $z$ and $z'$ map $h$ is identity, the fibers over all core vertices there are trivial. For the fiber over $z'$ the only
vertex with non-trivial preimage is $t$, which is a tassel, hence again all core vertices have trivial fibers. Therefore, vertex $\tilde z$ is the only core vertex with non-trivial
fiber $F_z$.

\vskip 1em

We define $f^{(1)}$ as the map of $2\Net^\ulc$ represented by $h$. Since $g$ has fewer non-trivial fibers than the original map $f$, we can iterate this process and obtain the
required factorization.

\qed


\vskip 5em

\appendix
\section{2-PreNets with $q$-tassels}
\label{sec_q_tassels}
In this section we address the problem of non-functoriality of the lift of proposition \ref{prop_lift} by introducing a refined version of a prenet with tassels,
which specifies the {\em quality} of each tassel.

\begin{nparagraph}[$2$-prenets with $q$-tassels.]
A prenet with $q$-tassels is a prenet $S$ together with
the partition of $S_2$ into pairwise disjoint subsets $t_q(S) \subset S_2$, indexed by an integer $q \in \Z_{\ge 0}$. For $q = 0$ we will continue to refer to vertices in $t_0(S)$ are the
core vertices and write $c(S_2) = t_0(S)$, and for the rest we will refer to vertices in $t_q(S)$ as tassels of quality $q$.

A map of prenets with $q$-tassels $f\from S \to T$ is defined the same way as a map of prenets with tassels, except the condition of definition \ref{def_prenet_map_tassels} is replaced with
$$
f(t_q(S)) \subset \bigcup_{q' \le q} t_{q'}(T).
$$

Moreover, we say that a map $f\from S \to T$ is a weak equivalence if it satisfies condition \ref{def_weak_equiv} (b) and the following generalization of condition (a): for every vertex
$x \in t_q(T)$ the preimage $f^{-1}(x)$ contains exactly one vertex belonging to $t_q(S)$.

We denote the category of $2$-prenets with $q$-tassels by $2\PreNet_q$, and the class of weak equivalences in it by $\cW_q$.

There is a forgetful functor
$$
\begin{tikzcd}[sep=3em]
2\PreNet_q \ar[r, "F"] & 2\PreNet,
\end{tikzcd}
$$
that forgets the quality of tassels. More precisely, for prenet $S \in 2\PreNet_q$ we put the underlying prenet of $F(S)$ to be the same as underlying prenet of $S$ and define
the set of tassels as
$$
t(F(S)) := \bigcup_{q \ge 1} t_q(S).
$$
Clearly, functor $F$ sends weak equivalences $\cW_q$ to weak equivalences $\cW$.

\end{nparagraph}

\begin{nparagraph}
It is immediate to see that the factorization lemmas \ref{lemma_map_factor}, \ref{lemma_tc_factor}, \ref{lemma_vc_factor}, \ref{lemma_w_factor_tv} and \ref{lemma_w_factor_vt} all hold in
the category $2\PreNet_q$. Moreover, the pushout lemma \ref{lemma_vc_pushout} as well as the special case of lifting proposition \ref{prop_lift} for vertex collapses remain unchanged,
and therefore we have an analog of the proposition \ref{prop_loc_vc} describing the localization of $2\PreNet_q$ with respect to vertex collapses $\mathcal{VC}_q$.
$$
\Hom_{2\PreNet_q[\mathcal{VC}_q^{-1}]}(S, T) \ \isom\ \Hom_{2\PreNet_q}(m_1(S), m_1(T)).
$$

\end{nparagraph}

\begin{nparagraph}[Functoriality of the lift.]
Now we will slightly modify construction of the lift along a tassel collapse in proposition \ref{prop_lift} to incorporate the quality of tassels and to ensure certain functorial properties.
Retaining notation from the proof of \ref{prop_lift}, by definition of weak equivalence in $\cW_q$ for every $y \in T_2$ there exists unique vertex $v \in g^{-1}(y) \subset T'_2$ of
the same quality as $y$, which we denote by $q = q(v) = q(y)$. Then we assign quality to a vertex $(x, u)$ in $S'_2$ as
$$
q(x, u) = q(x) + q(u) - q.
$$
This guarantees that the map $g'$ is again a tassel collapse in $\cW_q$.

Since every vertex $y$ has a special vertex $v \in g^{-1}(y)$ we have eliminated some of the choices we had to make in the construction of the lift. In fact it is enough to ensure the
following functoriality properties. Let us denote by $L(f, g)$ the prenet $S'$ that we just constructed, and consider the following diagram.
$$
\begin{tikzcd}[sep=3em]
& L(f, gv) \ar[d, "v'"'] \ar[r, "f''"] & R \ar[d, "v"] \\
L(fu, g) \ar[d, "g''"'] \ar[r, "u'"] & L(f, g) \ar[d, "g'"'] \ar[r, "f'"] & Q \ar[d, "g"] \\
P \ar[r, "u"'] & S \ar[r, "f"'] & T
\end{tikzcd}
$$
Then we have
\begin{enumerate}[label=\alph*)]
\item $L(fu, g) = L(u, g')$ and the composition $f'u'$ coincides with the lift of $fu$ along $g$;
\item $L(f, gv) = L(f', v)$ and the composition $g'v'$ coincides with the lift of $gv$ along $f$.
\end{enumerate}
For $L(fu, g)$ this essentially follows from the observation that $A(u^{-1}A(f^{-1}(y))) = A((fu)^{-1}(y))$ (see the proof of lemma \ref{lemma_map_composition}). And for $L(f, gv)$
this follows directly by examining subsets $V_1$ and $V_2$ (in the notation of \ref{prop_lift}) for each of the maps $g$, $v$ and $gv$. In fact part (b) holds even in the original
category of prenets, without introduction of quality of tassels.

This functoriality allows us to define the category of {\em $2$-pseudonets}.
\end{nparagraph}

\begin{definition}
The category of $2$-pseudonets, denoted by $2\PsNet$ has the same objects as $2\PreNet_q$ and sets of maps defined by
$$
\Hom_{2\PsNet}(S, T) \ :=\  \colim_{\wtilde S \in \cW_q \downarrow S} \Hom_{2\PreNet_q[\mathcal{VC}_q^{-1}]}(\wtilde S, T) \ \isom\ \colim_{\wtilde S \in \mathcal{TC}_q \downarrow S} \Hom_{2\PreNet_q}(m_1(\wtilde S), m_1(T)).
$$
\end{definition}

We define composition of maps in $2\PsNet$ using the refined lifts $L(f, g)$ described above. Namely, let $f$ and $g$ be representative of two composable maps in $2$-pseudonets, then we define
their composition as the map represented by the composition $gf'$ in the following diagram.
$$
\begin{tikzcd}[sep=3em]
L(f, v) \ar[d, "v'"'] \ar[r, "f'"] & \wtilde T \ar[d, "v"] \ar[r, "g"] & P \\
\wtilde S \ar[d, "u"'] \ar[r, "f"'] & T & \\
S & & 
\end{tikzcd}
$$
Let $w\from \wtilde S' \to \wtilde S$ be a weak equivalence, so that composition $fw$ is another representative of the same map as $f$, then functoriality in (a) for $L(fw, v)$ ensures
that the composition doesn't depend on the representative of $f$. Similarly, for a weak equivalence $w'\from \wtilde T' \to \wtilde T$ the functoriality of $L(f, vw')$ ensures independence
of the representative $g$. Moreover, it is easy to see by utilizing both functoriality properties (a) and (b) that the composition is associative.

\begin{nparagraph}
It is clear from the definition of category of $2$-pseudonets that vertex collapses are invertible in $2\PsNet$, while the general weak equivalences in principle may not be invertible.
However, it is easy to see that they have right inverses. Indeed, let $f\from S \to T$ be a weak equivalence, then in the category of $2$-pseudonets we have the map $g \from T \to S$ represented by the span
$$
\begin{tikzcd}
T & S \ar[l, "f"'] \ar[r, "="] & S.
\end{tikzcd}
$$
By definition the composition $fg$ is represented by the span
$$
\begin{tikzcd}
T & S \ar[l, "f"'] \ar[r, "f"] & T,
\end{tikzcd}
$$
which coincides with the identity map in the colimit. Therefore, $g$ is a right inverse to $f$. It may be interesting to investigate the universality of this type of construction
among the categories that admit right inverses of weak equivalences.

\end{nparagraph}

\begin{nparagraph}

Observe that we have the following diagram of categories.
$$
\begin{tikzcd}[sep=4em]
2\PreNet_q \ar[d, "L'"'] \ar[rr, "F"]  & & 2\PreNet \ar[d, "L"] \\
2\PsNet \ar[r, "L''"] & 2\PreNet_q[\cW_q^{-1}] \ar[r, "F{[}\cW^{-1}{]}"] & 2\Net \\
& (2\Tree, \Epi_1) \ar[ul, "J_q"] \ar[ur, "J"'] & 
\end{tikzcd}
$$
Since $F$ preserves weak equivalences, it induces a functor $F[\cW^{-1}]$ between the localized categories. Here $J$ is the fully faithful inclusion defined in \ref{par_J}, and $J_q$ is
a lift of $J$ into the category of $2$-pseudonets. Such a lift is not canonical, as it depends on the choice of qualities for the added tassels.

For example for a $2$-tree $T$, keeping notation of \ref{par_J},
we can put $t_1(J_q(T)) = T_1 - p(T_2)$ and all other $t_q$ with $q > 1$ to be empty. And in the definition of $J_q(f)$ for a map $f\from S \to T$ we put
$t_1(\wtilde S)_x = p_T^{-1}(f(x)) - f(p_S^{-1}(x))$ if vertex $x$ has children in $S$ and if it doesn't then all but one (arbitrarily chosen) vertex in $p_T^{-1}(f(x))$ we put in $t_2(\wtilde S)$
while the chosen vertex is put in $t_1(\wtilde S)$.

Regardless of the choices made, since the composition $J = F[\cW^{-1}] L'' J_q$ is faithful, the functor $J_q$ is also faithful.

\end{nparagraph}

\vskip 5em

\section{Examples}
\label{sec_examples}

\begin{example}
\label{exa_map_c}
Consider the following level-wise map $f_i\from S_i \to I_i$ as depicted below.
$$
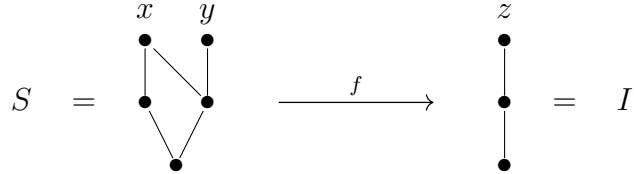
\begin{tikzcd}[sep=5em,cells={inner xsep=2em}]
S \quad= \quad
\begin{tikzpicture}[inner sep=0pt,baseline=(a11.base)]
\def\u{2em}
\node (a21) at (0, 0) {$\bullet$} node at ($(a21) + (0, 0.7em)$) {$x$};
\node (a22) at (1*\u, 0) {$\bullet$} node at ($(a22) + (0, 0.7em)$) {$y$};
\node (a11) at (0, -1*\u) {$\bullet$};
\node (a12) at (1*\u, -1*\u) {$\bullet$};
\node (a01) at (0.5*\u, -2*\u) {$\bullet$};

\draw (a21) -- (a11) -- (a01);
\draw (a21) -- (a12) -- (a01);
\draw (a22) -- (a12);
\end{tikzpicture} \ar[r, "f"] &
\begin{tikzpicture}[inner sep=0pt,baseline=(a11.base)]
\def\u{2em}
\node (a21) at (0, 0) {$\bullet$} node at ($(a21) + (0, 0.7em)$) {$z$};
\node (a11) at (0, -1*\u) {$\bullet$};
\node (a01) at (0, -2*\u) {$\bullet$};

\draw (a21) -- (a11) -- (a01);
\end{tikzpicture} \quad=\quad I
\end{tikzcd}
$$
In other words it send all vertices of $S_i$ to the unique vertex in $I_i$. The geometric realization of this map looks as follows.

$$
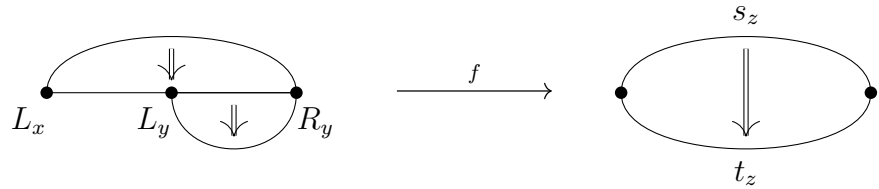
\begin{tikzcd}[sep=5em,cells={inner xsep=2em}]
\begin{tikzpicture}[baseline=-0.2em]
\def\u{4em}

\draw (0, 0) .. controls (0, 0.6*\u) and (2*\u, 0.6*\u) .. node (m11) [midway] {} (2*\u, 0);
\draw (0, 0) -- node (m12) [midway] {} (2*\u, 0);
\draw (1*\u, 0) -- node (m21) [midway] {} (2*\u, 0);
\draw (1*\u, 0) .. controls (1*\u, -0.6*\u) and (2*\u, -0.6*\u) .. node (m22) [midway] {} (2*\u, 0);

\draw [double distance=0.1em, arrows={->[width=0.9em]}] (m11) -- (m12);
\draw [double distance=0.1em, arrows={->[width=0.9em]}] (m21) -- (m22);

\fill (0, 0) circle [radius=0.2em] node [inner sep=0pt,below=0.9em,left] {$L_x$};
\fill (1*\u, 0) circle [radius=0.2em] node [inner sep=0pt,below=0.9em,left] {$L_y$};
\fill (2*\u, 0) circle [radius=0.2em] node [inner sep=0pt,below=0.9em,right] {$R_y$};;

\end{tikzpicture}
\ar[r, "f"] &
\begin{tikzpicture}[baseline=-0.2em]
\def\u{4em}

\draw (0, 0) .. controls (0, 0.6*\u) and (2*\u, 0.6*\u) .. node (m11) [midway] {} node [midway,above] {$s_z$} (2*\u, 0);
\draw (0, 0) .. controls (0, -0.6*\u) and (2*\u, -0.6*\u) .. node (m13) [midway] {} node [midway,below] {$t_z$} (2*\u, 0);

\draw [double distance=0.1em, arrows={->[width=0.9em]}] (m11) -- (m13);

\fill (0, 0) circle [radius=0.2em];
\fill (2*\u, 0) circle [radius=0.2em];
\end{tikzpicture}
\end{tikzcd}
$$

This is not a map of prenets, since it does not satisfy property (c) of \ref{def_prenet_map}. Indeed, let $z$ denote the only vertex in $I_2$, then the subset $B(f^{-1}(z))$ consists
of both vertices $x$ and $y$, but it is clear that $p_S(x)$ and $p_S(y)$ do not form a disjoint cover of $S_1$.

Let us also give an interpretation of the failure of $f$ to be a map of prenets in the language of paragraph \ref{par_cat_info}. In the geometric realization of prenet $S$
the cell $C_x$ determines the segment $t_x = [L_x, R_y]$ and cell $C_y$ determines $s_y = [L_y, R_y]$, while the segment $[L_x, L_y]$ is undetermined. However,
the segment $t_z$ on the right hand side is formed out of two pieces on the left hand side, namely $t_y$ and $[L_x, L_y]$. Since one of them is not determined, the left side
picture doesn't contain enough information to form the right side picture, and therefore $f$ is not a map of prenets.

\end{example}

\begin{example}
\label{exa_no_pushout}
Consider a span of prenets $\begin{tikzcd}[cramped,sep=small]T & S \ar[l, "f"'] \ar[r, "g"] & P\end{tikzcd}$
as depicted in the picture below.
$$
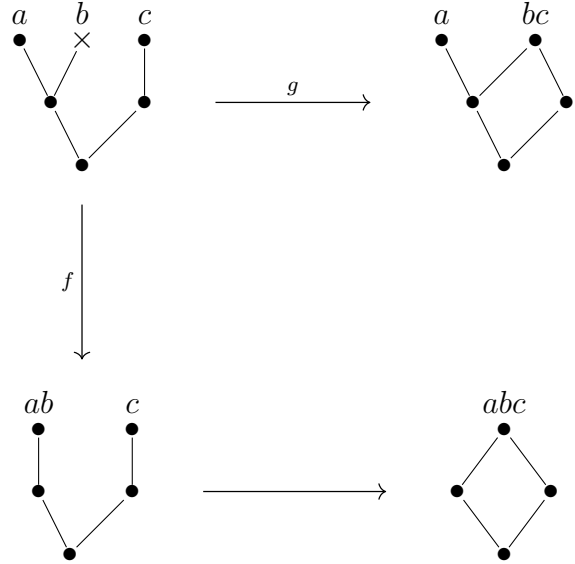
\begin{tikzcd}[sep=5em,cells={inner xsep=2em, inner ysep=1em}]
\begin{tikzpicture}[inner sep=0pt,baseline=(a11.base)]
\def\u{2em}
\node (a21) at (0, 0) {$\bullet$} node at ($(a21) + (0, 0.5em)$) {$a$};
\node (a22) at (1*\u, 0) {$\cross$} node at ($(a22) + (0, 0.5em)$) {$b$};
\node (a23) at (2*\u, 0) {$\bullet$} node at ($(a23) + (0, 0.5em)$) {$c$};
\node (a11) at (0.5*\u, -1*\u) {$\bullet$};
\node (a12) at (2*\u, -1*\u) {$\bullet$};
\node (a01) at (1*\u, -2*\u) {$\bullet$};

\draw (a21) -- (a11) -- (a22);
\draw (a23) -- (a12);
\draw (a11) -- (a01) -- (a12);
\end{tikzpicture} \ar[r, "g"] \ar[d, "f"'] &
\begin{tikzpicture}[inner sep=0pt,baseline=(a11.base)]
\def\u{2em}
\node (a21) at (0, 0) {$\bullet$} node at ($(a21) + (0, 0.5em)$) {$a$};
\node (a22) at (1.5*\u, 0) {$\bullet$} node at ($(a22) + (0, 0.5em)$) {$bc$};
\node (a11) at (0.5*\u, -1*\u) {$\bullet$};
\node (a12) at (2*\u, -1*\u) {$\bullet$};
\node (a01) at (1*\u, -2*\u) {$\bullet$};

\draw (a21) -- (a11) -- (a22);
\draw (a22) -- (a12) -- (a01);
\draw (a11) -- (a01);
\end{tikzpicture}\\
\begin{tikzpicture}[inner sep=0pt,baseline=(a11.base)]
\def\u{2em}
\node (a21) at (0.5*\u, 0) {$\bullet$} node at ($(a21) + (0, 0.5em)$) {$ab$};
\node (a22) at (2*\u, 0) {$\bullet$} node at ($(a22) + (0, 0.5em)$) {$c$};
\node (a11) at (0.5*\u, -1*\u) {$\bullet$};
\node (a12) at (2*\u, -1*\u) {$\bullet$};
\node (a01) at (1*\u, -2*\u) {$\bullet$};

\draw (a21) -- (a11);
\draw (a22) -- (a12);
\draw (a11) -- (a01) -- (a12);
\end{tikzpicture} \ar[r] &
\begin{tikzpicture}[inner sep=0pt,baseline=(a11.base)]
\def\u{2em}
\node (a21) at (0, 0) {$\bullet$} node at ($(a21) + (0, 0.5em)$) {$abc$};
\node (a11) at (-0.75*\u, -1*\u) {$\bullet$};
\node (a12) at (0.75*\u, -1*\u) {$\bullet$};
\node (a01) at (0, -2*\u) {$\bullet$};

\draw (a21) -- (a11) -- (a01) -- (a12) -- (a21);
\end{tikzpicture}
\end{tikzcd}
$$

In this diagram map $f$ is a weak equivalence, while $g$ is not. If we try to apply construction in the proof of lemma \ref{lemma_ww_pushout} then we would obtain prenet $Q$ in the bottom right corner, however,
an argument similar to that of example \ref{exa_map_c} shows that there is no map from $P$ to $Q$.

\end{example}

\begin{example}
\label{exa_J}
Consider the following two maps in the category of $2$-trees.
$$
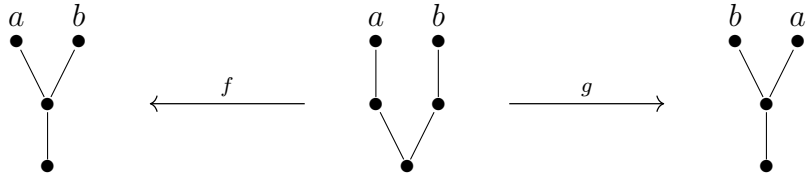
\begin{tikzcd}[sep=5em,cells={inner xsep=2em, inner ysep=1em}]
\begin{tikzpicture}[inner sep=0pt,baseline=(a11.base)]
\def\u{2em}
\node (a21) at (0, 0) {$\bullet$} node at ($(a21) + (0, 0.5em)$) {$a$};
\node (a22) at (1*\u, 0) {$\bullet$} node at ($(a22) + (0, 0.5em)$) {$b$};
\node (a11) at (0.5*\u, -1*\u) {$\bullet$};
\node (a01) at (0.5*\u, -2*\u) {$\bullet$};

\draw (a21) -- (a11) -- (a22);
\draw (a11) -- (a01);
\end{tikzpicture} &
\begin{tikzpicture}[inner sep=0pt,baseline=(a11.base)]
\def\u{2em}
\node (a21) at (0, 0) {$\bullet$} node at ($(a21) + (0, 0.5em)$) {$a$};
\node (a22) at (1*\u, 0) {$\bullet$} node at ($(a22) + (0, 0.5em)$) {$b$};
\node (a11) at (0, -1*\u) {$\bullet$};
\node (a12) at (1*\u, -1*\u) {$\bullet$};
\node (a01) at (0.5*\u, -2*\u) {$\bullet$};

\draw (a21) -- (a11) -- (a01);
\draw (a22) -- (a12) -- (a01);
\end{tikzpicture} \ar[l, "f"'] \ar[r, "g"] &
\begin{tikzpicture}[inner sep=0pt,baseline=(a11.base)]
\def\u{2em}
\node (a21) at (0, 0) {$\bullet$} node at ($(a21) + (0, 0.5em)$) {$b$};
\node (a22) at (1*\u, 0) {$\bullet$} node at ($(a22) + (0, 0.5em)$) {$a$};
\node (a11) at (0.5*\u, -1*\u) {$\bullet$};
\node (a01) at (0.5*\u, -2*\u) {$\bullet$};

\draw (a21) -- (a11) -- (a22);
\draw (a11) -- (a01);
\end{tikzpicture}
\end{tikzcd}
$$

Since these trees are already pruned, applying functor $J$ described in \ref{par_J} doesn't change their shape. However, the maps $f$ and $g$ are not maps in the category of prenets.
We will provide an explicit description of maps $J(f)$ and $J(g)$ that were constructed in \ref{par_J_map}. The map $J(f)\from \wtilde S \to J(T)$ can be depicted as
$$
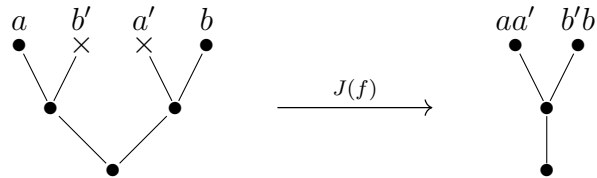
\begin{tikzcd}[sep=5em,cells={inner xsep=2em, inner ysep=1em}]
\begin{tikzpicture}[inner sep=0pt,baseline=(a11.base)]
\def\u{2em}
\node (a21) at (0, 0) {$\bullet$} node at ($(a21) + (0, 0.5em)$) {$a$};
\node (a22) at (1*\u, 0) {$\cross$} node at ($(a22) + (0, 0.5em)$) {$b'$};
\node (a23) at (2*\u, 0) {$\cross$} node at ($(a23) + (0, 0.5em)$) {$a'$};
\node (a24) at (3*\u, 0) {$\bullet$} node at ($(a24) + (0, 0.5em)$) {$b$};
\node (a11) at (0.5*\u, -1*\u) {$\bullet$};
\node (a12) at (2.5*\u, -1*\u) {$\bullet$};
\node (a01) at (1.5*\u, -2*\u) {$\bullet$};

\draw (a21) -- (a11) -- (a22);
\draw (a23) -- (a12) -- (a24);
\draw (a11) -- (a01) -- (a12);
\end{tikzpicture} \ar[r, "J(f)"] &
\begin{tikzpicture}[inner sep=0pt,baseline=(a11.base)]
\def\u{2em}
\node (a21) at (0, 0) {$\bullet$} node at ($(a21) + (0, 0.5em)$) {$aa'$};
\node (a22) at (1*\u, 0) {$\bullet$} node at ($(a22) + (0, 0.5em)$) {$b'b$};
\node (a11) at (0.5*\u, -1*\u) {$\bullet$};
\node (a01) at (0.5*\u, -2*\u) {$\bullet$};

\draw (a21) -- (a11) -- (a22);
\draw (a11) -- (a01);
\end{tikzpicture}
\end{tikzcd}
$$
And map $J(g)\from \wtilde S' \to J(T)$ can be depicted as
$$
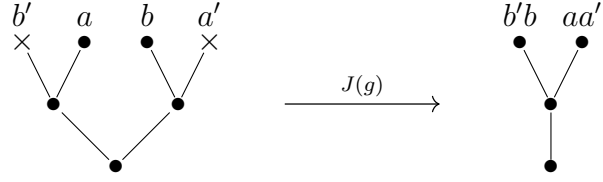
\begin{tikzcd}[sep=5em,cells={inner xsep=2em, inner ysep=1em}]
\begin{tikzpicture}[inner sep=0pt,baseline=(a11.base)]
\def\u{2em}
\node (a21) at (0, 0) {$\cross$} node at ($(a21) + (0, 0.5em)$) {$b'$};
\node (a22) at (1*\u, 0) {$\bullet$} node at ($(a22) + (0, 0.5em)$) {$a$};
\node (a23) at (2*\u, 0) {$\bullet$} node at ($(a23) + (0, 0.5em)$) {$b$};
\node (a24) at (3*\u, 0) {$\cross$} node at ($(a24) + (0, 0.5em)$) {$a'$};
\node (a11) at (0.5*\u, -1*\u) {$\bullet$};
\node (a12) at (2.5*\u, -1*\u) {$\bullet$};
\node (a01) at (1.5*\u, -2*\u) {$\bullet$};

\draw (a21) -- (a11) -- (a22);
\draw (a23) -- (a12) -- (a24);
\draw (a11) -- (a01) -- (a12);
\end{tikzpicture} \ar[r, "J(g)"] &
\begin{tikzpicture}[inner sep=0pt,baseline=(a11.base)]
\def\u{2em}
\node (a21) at (0, 0) {$\bullet$} node at ($(a21) + (0, 0.5em)$) {$b'b$};
\node (a22) at (1*\u, 0) {$\bullet$} node at ($(a22) + (0, 0.5em)$) {$aa'$};
\node (a11) at (0.5*\u, -1*\u) {$\bullet$};
\node (a01) at (0.5*\u, -2*\u) {$\bullet$};

\draw (a21) -- (a11) -- (a22);
\draw (a11) -- (a01);
\end{tikzpicture}
\end{tikzcd}
$$

\end{example}

\begin{example}
\label{exa_ulc_lift_not_lc}
We give an example of a non-locally constant lift of a ULC map. Consider the diagram of prenets.
$$
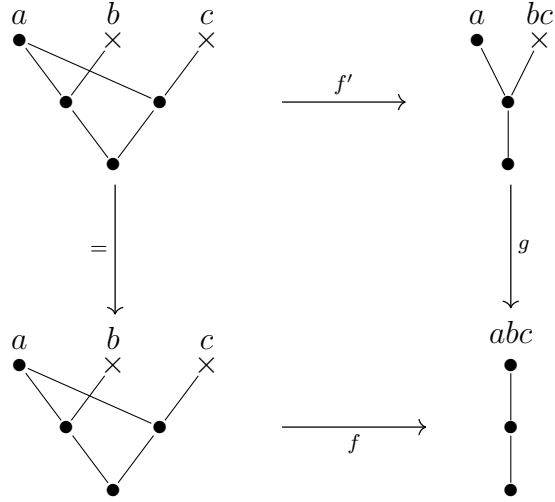
\begin{tikzcd}[sep=4em,cells={inner xsep=2em}]
\begin{tikzpicture}[inner sep=0pt,baseline=(a11.base)]
\def\u{2em}
\node (a22) at (0, 0) {$\bullet$} node at ($(a22) + (0, 0.5em)$) {$a$};
\node (a23) at (1.5*\u, 0) {$\cross$} node at ($(a23) + (0, 0.5em)$) {$b$};
\node (a24) at (3*\u, 0) {$\cross$} node at ($(a24) + (0, 0.5em)$) {$c$};
\node (a11) at (0.75*\u, -1*\u) {$\bullet$};
\node (a12) at (2.25*\u, -1*\u) {$\bullet$};
\node (a01) at (1.5*\u, -2*\u) {$\bullet$};

\draw (a22) -- (a11);
\draw (a22) -- (a12);
\draw (a23) -- (a11);
\draw (a24) -- (a12);
\draw (a11) -- (a01) -- (a12);
\end{tikzpicture} \ar[d, "="'] \ar[r, "f'"] &
\begin{tikzpicture}[inner sep=0pt,baseline=(a11.base)]
\def\u{2em}
\node (a21) at (0, 0) {$\bullet$} node at ($(a21) + (0, 0.5em)$) {$a$};
\node (a22) at (1*\u, 0) {$\cross$} node at ($(a22) + (0, 0.5em)$) {$bc$};
\node (a11) at (0.5*\u, -1*\u) {$\bullet$};
\node (a01) at (0.5*\u, -2*\u) {$\bullet$};

\draw (a21) -- (a11) -- (a22);
\draw (a11) -- (a01);
\end{tikzpicture} \ar[d, "g"] \\
\begin{tikzpicture}[inner sep=0pt,baseline=(a11.base)]
\def\u{2em}
\node (a22) at (0, 0) {$\bullet$} node at ($(a22) + (0, 0.5em)$) {$a$};
\node (a23) at (1.5*\u, 0) {$\cross$} node at ($(a23) + (0, 0.5em)$) {$b$};
\node (a24) at (3*\u, 0) {$\cross$} node at ($(a24) + (0, 0.5em)$) {$c$};
\node (a11) at (0.75*\u, -1*\u) {$\bullet$};
\node (a12) at (2.25*\u, -1*\u) {$\bullet$};
\node (a01) at (1.5*\u, -2*\u) {$\bullet$};

\draw (a22) -- (a11);
\draw (a22) -- (a12);
\draw (a23) -- (a11);
\draw (a24) -- (a12);
\draw (a11) -- (a01) -- (a12);
\end{tikzpicture} \ar[r, "f"'] &
\begin{tikzpicture}[inner sep=0pt,baseline=(a11.base)]
\def\u{2em}
\node (a21) at (0, 0) {$\bullet$} node at ($(a21) + (0, 0.7em)$) {$abc$};
\node (a11) at (0, -1*\u) {$\bullet$};
\node (a01) at (0, -2*\u) {$\bullet$};

\draw (a21) -- (a11) -- (a01);
\end{tikzpicture}
\end{tikzcd}
$$
Here map $f$ is ULC, $g$ is a tassel collapse, but $f'$ is not a locally constant map, since the map
$$
w_{abc \to a} \from (gf')^{-1}(abc) \to f'^{-1}(a)
$$
is not a weak equivalence.

\end{example}

\begin{example}
\label{exa_factor_elem}
Let us construct a factorization of the map $J(f)$ from example \ref{exa_J} into a chain of elementary ULC maps, as in the proof of theorem \ref{thm_factor_elem}.
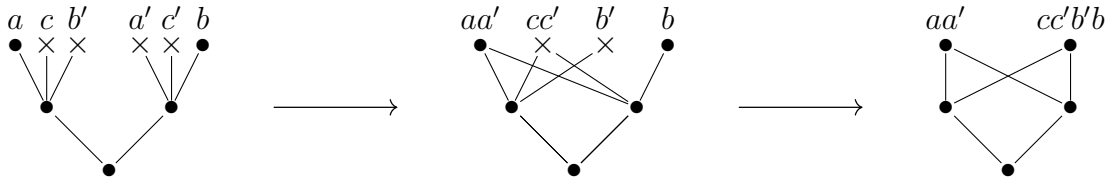
$$
\begin{tikzcd}[sep=4em,cells={inner xsep=2em}]
\begin{tikzpicture}[inner sep=0pt,baseline=(a11.base)]
\def\u{2em}
\node (a21) at (0, 0) {$\bullet$} node at ($(a21) + (0, 0.5em)$) {$a$};
\node (a215) at (0.5*\u, 0) {$\cross$} node at ($(a215) + (0, 0.5em)$) {$c$};
\node (a22) at (1*\u, 0) {$\cross$} node at ($(a22) + (0, 0.5em)$) {$b'$};
\node (a23) at (2*\u, 0) {$\cross$} node at ($(a23) + (0, 0.5em)$) {$a'$};
\node (a235) at (2.5*\u, 0) {$\cross$} node at ($(a235) + (0, 0.5em)$) {$c'$};
\node (a24) at (3*\u, 0) {$\bullet$} node at ($(a24) + (0, 0.5em)$) {$b$};
\node (a11) at (0.5*\u, -1*\u) {$\bullet$};
\node (a12) at (2.5*\u, -1*\u) {$\bullet$};
\node (a01) at (1.5*\u, -2*\u) {$\bullet$};

\draw (a21) -- (a11) -- (a22);
\draw (a23) -- (a12) -- (a24);
\draw (a215) -- (a11);
\draw (a235) -- (a12);
\draw (a11) -- (a01) -- (a12);
\end{tikzpicture} \ar[r] &
\begin{tikzpicture}[inner sep=0pt,baseline=(a11.base)]
\def\u{2em}
\node (a21) at (0, 0) {$\bullet$} node at ($(a21) + (0, 0.5em)$) {$aa'$};
\node (a22) at (1*\u, 0) {$\cross$} node at ($(a22) + (0, 0.5em)$) {$cc'$};
\node (a23) at (2*\u, 0) {$\cross$} node at ($(a23) + (0, 0.5em)$) {$b'$};
\node (a24) at (3*\u, 0) {$\bullet$} node at ($(a24) + (0, 0.5em)$) {$b$};
\node (a11) at (0.5*\u, -1*\u) {$\bullet$};
\node (a12) at (2.5*\u, -1*\u) {$\bullet$};
\node (a01) at (1.5*\u, -2*\u) {$\bullet$};

\draw (a21) -- (a11) -- (a01);
\draw (a21) -- (a12) -- (a01);
\draw (a22) -- (a11);
\draw (a22) -- (a12);
\draw (a23) -- (a11);
\draw (a24) -- (a12);
\draw (a11) -- (a01) -- (a12);
\end{tikzpicture} \ar[r] &
\begin{tikzpicture}[inner sep=0pt,baseline=(a11.base)]
\def\u{2em}
\node (a21) at (0.5*\u, 0) {$\bullet$} node at ($(a21) + (0, 0.5em)$) {$aa'$};
\node (a24) at (2.5*\u, 0) {$\bullet$} node at ($(a24) + (0, 0.5em)$) {$cc'b'b$};
\node (a11) at (0.5*\u, -1*\u) {$\bullet$};
\node (a12) at (2.5*\u, -1*\u) {$\bullet$};
\node (a01) at (1.5*\u, -2*\u) {$\bullet$};

\draw (a21) -- (a11) -- (a24);
\draw (a21) -- (a12);
\draw (a24) -- (a12);
\draw (a11) -- (a01) -- (a12);
\end{tikzpicture}
\end{tikzcd}
$$
Notice that the first and last prenet are respectively weakly equivalent to the source and target of $J(f)$. The two fibers over the core vertices of the first map are
$$
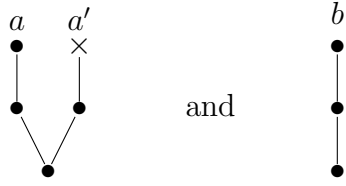
\begin{tikzcd}[sep=4em,cells={inner xsep=2em}]
\begin{tikzpicture}[inner sep=0pt,baseline=(a11.base)]
\def\u{2em}
\node (a21) at (0, 0) {$\bullet$} node at ($(a21) + (0, 0.5em)$) {$a$};
\node (a22) at (1*\u, 0) {$\cross$} node at ($(a22) + (0, 0.5em)$) {$a'$};
\node (a11) at (0, -1*\u) {$\bullet$};
\node (a12) at (1*\u, -1*\u) {$\bullet$};
\node (a01) at (0.5*\u, -2*\u) {$\bullet$};

\draw (a21) -- (a11) -- (a01);
\draw (a22) -- (a12) -- (a01);
\end{tikzpicture} \hskip3em\text{and}\hskip3em
\begin{tikzpicture}[inner sep=0pt,baseline=(a11.base)]
\def\u{2em}
\node (a21) at (0, 0) {$\bullet$} node at ($(a21) + (0, 0.7em)$) {$b$};
\node (a11) at (0, -1*\u) {$\bullet$};
\node (a01) at (0, -2*\u) {$\bullet$};

\draw (a21) -- (a11) -- (a01);
\end{tikzpicture}
\end{tikzcd}
$$
And for the second map the two fibers over the core vertices are
$$
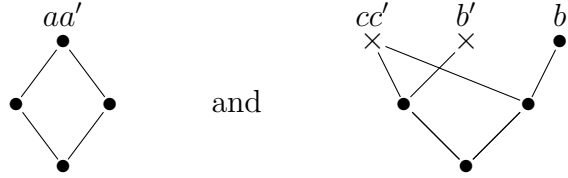
\begin{tikzcd}[sep=4em,cells={inner xsep=2em}]
\begin{tikzpicture}[inner sep=0pt,baseline=(a11.base)]
\def\u{2em}
\node (a21) at (0, 0) {$\bullet$} node at ($(a21) + (0, 0.5em)$) {$aa'$};
\node (a11) at (-0.75*\u, -1*\u) {$\bullet$};
\node (a12) at (0.75*\u, -1*\u) {$\bullet$};
\node (a01) at (0, -2*\u) {$\bullet$};

\draw (a21) -- (a11) -- (a01) -- (a12) -- (a21);
\end{tikzpicture}\hskip3em\text{and}\hskip3em
\begin{tikzpicture}[inner sep=0pt,baseline=(a11.base)]
\def\u{2em}
\node (a22) at (0, 0) {$\cross$} node at ($(a22) + (0, 0.5em)$) {$cc'$};
\node (a23) at (1.5*\u, 0) {$\cross$} node at ($(a23) + (0, 0.5em)$) {$b'$};
\node (a24) at (3*\u, 0) {$\bullet$} node at ($(a24) + (0, 0.5em)$) {$b$};
\node (a11) at (0.5*\u, -1*\u) {$\bullet$};
\node (a12) at (2.5*\u, -1*\u) {$\bullet$};
\node (a01) at (1.5*\u, -2*\u) {$\bullet$};

\draw (a11) -- (a01);
\draw (a12) -- (a01);
\draw (a22) -- (a11);
\draw (a22) -- (a12);
\draw (a23) -- (a11);
\draw (a24) -- (a12);
\draw (a11) -- (a01) -- (a12);
\end{tikzpicture}
\end{tikzcd}
$$
Notice that the first fiber here is trivial, since it is weakly equivalent to $I$.

\end{example}

\begin{example}
\label{exa_net_not_operadic}
We give an example of two maps of $2$-prenets $f, g \from S \to T$, who's images in the category of $2$-nets coincide, but who's fibers are not pairwise weakly equivalent.
Consider the following diagram of $2$-prenets.
$$
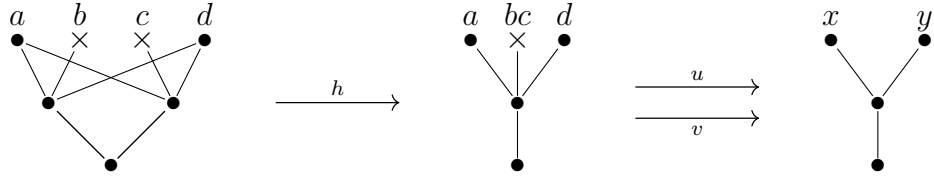
\begin{tikzcd}[sep=4em,cells={inner xsep=2em}]
\begin{tikzpicture}[inner sep=0pt,baseline=(a11.base)]
\def\u{2em}
\node (a21) at (0, 0) {$\bullet$} node at ($(a21) + (0, 0.5em)$) {$a$};
\node (a22) at (1*\u, 0) {$\cross$} node at ($(a22) + (0, 0.5em)$) {$b$};
\node (a23) at (2*\u, 0) {$\cross$} node at ($(a23) + (0, 0.5em)$) {$c$};
\node (a24) at (3*\u, 0) {$\bullet$} node at ($(a24) + (0, 0.5em)$) {$d$};
\node (a11) at (0.5*\u, -1*\u) {$\bullet$};
\node (a12) at (2.5*\u, -1*\u) {$\bullet$};
\node (a01) at (1.5*\u, -2*\u) {$\bullet$};

\draw (a21) -- (a11) -- (a01);
\draw (a21) -- (a12) -- (a01);
\draw (a22) -- (a11);
\draw (a23) -- (a12);
\draw (a24) -- (a11) -- (a01);
\draw (a24) -- (a12) -- (a01);
\end{tikzpicture} \ar[r, "h"] &
\begin{tikzpicture}[inner sep=0pt,baseline=(a11.base)]
\def\u{2em}
\node (a21) at (0, 0) {$\bullet$} node at ($(a21) + (0, 0.5em)$) {$a$};
\node (a22) at (0.75*\u, 0) {$\cross$} node at ($(a22) + (0, 0.5em)$) {$bc$};
\node (a23) at (1.5*\u, 0) {$\bullet$} node at ($(a23) + (0, 0.5em)$) {$d$};
\node (a11) at (0.75*\u, -1*\u) {$\bullet$};
\node (a01) at (0.75*\u, -2*\u) {$\bullet$};

\draw (a21) -- (a11) -- (a22);
\draw (a23) -- (a11) -- (a01);
\end{tikzpicture} \ar[shift left=0.5em, r, "u"] \ar[shift right=0.5em, r, "v"'] &
\begin{tikzpicture}[inner sep=0pt,baseline=(a11.base)]
\def\u{2em}
\node (a21) at (0, 0) {$\bullet$} node at ($(a21) + (0, 0.5em)$) {$x$};
\node (a23) at (1.5*\u, 0) {$\bullet$} node at ($(a23) + (0, 0.5em)$) {$y$};
\node (a11) at (0.75*\u, -1*\u) {$\bullet$};
\node (a01) at (0.75*\u, -2*\u) {$\bullet$};

\draw (a21) -- (a11);
\draw (a23) -- (a11) -- (a01);
\end{tikzpicture}
\end{tikzcd}
$$
Here map $u$ sends vertices $a$ and $bc$ to $x$ and vertex $d$ to $y$, while map $v$ sends $a$ to $x$ and vertices $bc$ and $d$ to $y$. We put $f = uh$ and $g = vh$.

Since $u$ and $v$ are tassel collapses their images in the category $2\Net$ coincide, hence images of $f$ and $g$ also coincide. However, they have different fibers.
For map $f$ the fibers are
$$
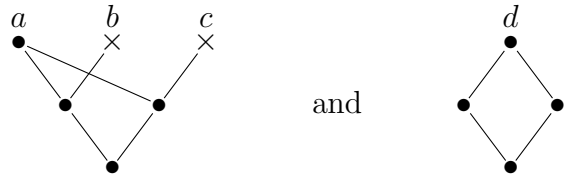
\begin{tikzcd}[sep=4em,cells={inner xsep=2em}]
\begin{tikzpicture}[inner sep=0pt,baseline=(a11.base)]
\def\u{2em}
\node (a22) at (0, 0) {$\bullet$} node at ($(a22) + (0, 0.5em)$) {$a$};
\node (a23) at (1.5*\u, 0) {$\cross$} node at ($(a23) + (0, 0.5em)$) {$b$};
\node (a24) at (3*\u, 0) {$\cross$} node at ($(a24) + (0, 0.5em)$) {$c$};
\node (a11) at (0.75*\u, -1*\u) {$\bullet$};
\node (a12) at (2.25*\u, -1*\u) {$\bullet$};
\node (a01) at (1.5*\u, -2*\u) {$\bullet$};

\draw (a22) -- (a11);
\draw (a22) -- (a12);
\draw (a23) -- (a11);
\draw (a24) -- (a12);
\draw (a11) -- (a01) -- (a12);
\end{tikzpicture} \hskip3em\text{and}\hskip3em
\begin{tikzpicture}[inner sep=0pt,baseline=(a11.base)]
\def\u{2em}
\node (a21) at (0, 0) {$\bullet$} node at ($(a21) + (0, 0.5em)$) {$d$};
\node (a11) at (-0.75*\u, -1*\u) {$\bullet$};
\node (a12) at (0.75*\u, -1*\u) {$\bullet$};
\node (a01) at (0, -2*\u) {$\bullet$};

\draw (a21) -- (a11) -- (a01) -- (a12) -- (a21);
\end{tikzpicture}
\end{tikzcd}
$$
While for map $g$ the fibers are
$$
\begin{tikzcd}[sep=4em,cells={inner xsep=2em}]
\begin{tikzpicture}[inner sep=0pt,baseline=(a11.base)]
\def\u{2em}
\node (a21) at (0, 0) {$\bullet$} node at ($(a21) + (0, 0.5em)$) {$a$};
\node (a11) at (-0.75*\u, -1*\u) {$\bullet$};
\node (a12) at (0.75*\u, -1*\u) {$\bullet$};
\node (a01) at (0, -2*\u) {$\bullet$};

\draw (a21) -- (a11) -- (a01) -- (a12) -- (a21);
\end{tikzpicture} \hskip3em\text{and}\hskip3em
\begin{tikzpicture}[inner sep=0pt,baseline=(a11.base)]
\def\u{2em}
\node (a22) at (0, 0) {$\cross$} node at ($(a22) + (0, 0.5em)$) {$b$};
\node (a23) at (1.5*\u, 0) {$\cross$} node at ($(a23) + (0, 0.5em)$) {$c$};
\node (a24) at (3*\u, 0) {$\bullet$} node at ($(a24) + (0, 0.5em)$) {$d$};
\node (a11) at (0.75*\u, -1*\u) {$\bullet$};
\node (a12) at (2.25*\u, -1*\u) {$\bullet$};
\node (a01) at (1.5*\u, -2*\u) {$\bullet$};

\draw (a22) -- (a11);
\draw (a23) -- (a12);
\draw (a24) -- (a12);
\draw (a24) -- (a11);
\draw (a11) -- (a01) -- (a12);
\end{tikzpicture}
\end{tikzcd}
$$

\end{example}


\vfill\eject

\end{document}